\documentclass[a4paper, pdftex]{amsart}

\AtBeginDocument{%
  \def\MR#1{}
}
\usepackage{tikz, tikz-cd}
\usepackage{amsmath,amssymb}
\usepackage{amsthm}
\usepackage{mathrsfs}
\usepackage[shortlabels]{enumitem}
\usepackage{mathtools}
\usepackage{array}
\usepackage{comment}
\usepackage{stmaryrd}
\usepackage{url}
\usepackage{ascmac}

\usepackage[top=30truemm,bottom=30truemm,left=35truemm,right=35truemm]{geometry}

\usepackage{hyperref}

\usepackage{fixme}
\fxsetup{status=draft}

\theoremstyle{plain}
\newtheorem{thm}{Theorem}[section]
\newtheorem{prop}[thm]{Proposition}
\newtheorem{lem}[thm]{Lemma}
\newtheorem{claim}[thm]{Claim}
\newtheorem{cor}[thm]{Corollary}
\newtheorem*{thm*}{Theorem}

\theoremstyle{definition}
\newtheorem{defi}[thm]{Definition}

\newtheorem*{NaC}{Notation and Convention}
\newtheorem*{ACK}{Acknowledgment}

\theoremstyle{remark}
\newtheorem{rem}[thm]{Remark}
\newtheorem{nota}[thm]{Notation}

\numberwithin{equation}{section}

\newcommand{\Z}{\mathbb{Z}}
\newcommand{\Q}{\mathbb{Q}}
\newcommand{\R}{\mathbb{R}}
\newcommand{\C}{\mathbb{C}}

\renewcommand{\P}{\mathbb{P}}
\newcommand{\F}{\mathbb{F}}

\renewcommand{\F}{\mathbb{F}}

\newcommand{\s}{\sigma}

\newcommand{\emp}{\varnothing}

\newcommand{\ol}{\overline}
\newcommand{\wt}{\widetilde}

\newcommand{\ra}{\Rightarrow}

\newcommand{\hra}{\hookrightarrow}
\newcommand{\epm}{\twoheadrightarrow}
\newcommand{\dra}{\dashrightarrow}

\DeclareMathOperator{\rk}{rk}
\DeclareMathOperator{\id}{id}

\DeclareMathOperator{\Sym}{\mathrm{Sym}}

\DeclareMathOperator{\Spec}{\mathrm{Spec}}

\DeclareMathOperator{\Proj}{\mathrm{Proj}}

\DeclareMathOperator{\Exc}{\mathrm{Exc}}

\DeclareMathOperator{\Pic}{\mathrm{Pic}}
\DeclareMathOperator{\Gr}{\mathrm{Gr}}
\DeclareMathOperator{\Fl}{\mathrm{Fl}}
\DeclareMathOperator{\red}{\mathrm{red}}

\DeclareMathOperator{\Bl}{\mathrm{Bl}}

\DeclareMathOperator{\codim}{\mathrm{codim}}
\DeclareMathOperator{\Sing}{\mathrm{Sing}}

\DeclareMathOperator{\pr}{pr}

\DeclareMathOperator{\Hom}{Hom}
\DeclareMathOperator{\End}{End}
\DeclareMathOperator{\Ker}{Ker}
\DeclareMathOperator{\Cok}{Cok}
\DeclareMathOperator{\Supp}{Supp}
\DeclareMathOperator{\Ext}{Ext}

\DeclareMathOperator{\Cl}{Cl}

\DeclareMathOperator{\Hilb}{Hilb}
\DeclareMathOperator{\Univ}{Univ}
\DeclareMathOperator{\ev}{ev}
\DeclareMathOperator{\coev}{coev}

\newcommand{\mcA}{\mathcal{A}}

\newcommand{\mcC}{\mathcal{C}}
\newcommand{\mcD}{\mathcal{D}}
\newcommand{\mcE}{\mathcal{E}}
\newcommand{\mcF}{\mathcal{F}}
\newcommand{\mcG}{\mathcal{G}}
\newcommand{\mcH}{\mathcal{H}}
\newcommand{\mcI}{\mathcal{I}}

\newcommand{\mcK}{\mathcal{K}}
\newcommand{\mcO}{\mathcal{O}}
\newcommand{\mcL}{\mathcal{L}}
\newcommand{\mcN}{\mathcal{N}}

\newcommand{\mcQ}{\mathcal{Q}}

\newcommand{\mcS}{\mathcal{S}}
\newcommand{\mcT}{\mathcal{T}}

\newcommand{\mcV}{\mathcal{V}}

\let\Im\relax
\DeclareMathOperator{\Im}{\mathrm{Im}}

\DeclareMathOperator{\RG}{\mathrm{R}\Gamma}

\DeclareMathOperator{\RHom}{\mathrm{RHom}}
\newcommand{\RR}{\mathbf{R}}
\newcommand{\LL}{\mathbf{L}}

\DeclareMathOperator{\Db}{\mathrm{D}^{\mathrm{b}}}

\newcommand{\wF}{\mathsf{wF}}

\renewcommand{\mod}{\mathrm{mod}}

\newcommand{\sm}{\mathrm{sm}}

\newcommand{\gen}[1]{\langle #1 \rangle}
\newcommand{\ls}[1]{\lvert #1 \rvert}

\mathchardef\mhyphen="2D
\title[Rank $2$ weak Fano bundles on Fano $3$-folds of Picard rank $1$]{Rank two Weak Fano bundles on Fano threefolds of Picard rank one}
\author[T.FUKUOKA, W.HARA, D.ISHIKAWA]{Takeru Fukuoka, Wahei Hara, Daizo Ishikawa}

\address[T.FUKUOKA]{Graduate School of Mathematical Sciences\\The University of Tokyo\\3-8-1 Komaba\\Meguro-ku, Tokyo 153-8914, Japan}
\email{takeru.fukuoka@gmail.com}

\address[W.HARA]{Kavli Institute for the Physics and Mathematics of the Universe (WPI), University of Tokyo, 5-1-5 Kashiwanoha, Kashiwa, 277-8583, Japan.}
\email{wahei.hara@ipmu.jp}

\address[D.ISHIKAWA]{Department of Mathematics, School of Science and Engineering, Waseda University, Ohkubo 3-4-1, Shinjuku, Tokyo 169-8555, JAPAN}
\email{azoth@toki.waseda.jp}

\date{\today}


\begin{document}
\maketitle

\begin{abstract}
We classify rank two vector bundles on a Fano threefold of Picard rank one whose projectivizations are weak Fano. We also prove the existence of examples for each case of the classification result. Our classification includes detailed resolutions of them on the quadric threefold. 
\end{abstract}
\tableofcontents

\section{Introduction}\label{sec-intro}

A \emph{weak Fano bundle}, which was originally introduced by Langer \cite{Langer98}, is defined as a vector bundle $\mcE$ on a smooth projective variety $X$ such that $-K_{\P_{X}(\mcE)}$ is nef and big. 
This is a natural generalization of a \emph{Fano bundle}, which is defined as a vector bundle $\mcE$ such that $-K_{\P_{X}(\mcE)}$ is ample. 
The classification study of weak Fano bundles is a research theme that spans several domains, including the classification of high-dimensional weak Fano manifolds and the study of stable vector bundles on Fano manifolds. 
By the result of Mu\~{n}oz--Occhetta--Sol\'{a}~Conde \cite{mos2}, 
the classification of rank two Fano bundles on Fano threefolds of Picard rank one was established. 
The purpose of this paper is to extend this classification to rank two weak Fano bundles on Fano threefolds of Picard rank one.

Until now, the classification of rank two weak Fano bundles has been studied in the case where the base space is a projective space $\P^{n}$ with $n \geq 3$ or a del Pezzo threefold of Picard rank one \cite{yas, Ishikawa16, FHI1, FHI2}.
In other words, when the base space $X$ is a Fano threefold of Picard rank one, the classification had been completed when the \emph{Fano index} $i_{X}:=\max\{i \in \Z_{>0} \mid \frac{1}{i}(-K_{X}) \in \Pic(X)\}$ is even. 
It is well known that the Fano index $i_{X}$ is less than or equal to $4$, and if $i_{X}=3$, $X$ is isomorphic to a quadratic hypersurface $\Q^{3}$ \cite{KO73}. 
The purpose of this paper is to classify weak Fano vector bundles of rank two when $i_{X}$ is odd, i.e., $X = \Q^{3}$ or $i_{X} = 1$. 

\subsection{Classification on $\Q^{3}$}

The property that a vector bundle $\mcE$ is weak Fano is invariant under tensoring with line bundles. 
%depends only its projectivization $\P_{X}(\mcE)$. 
Therefore, in order to classify rank two weak Fano bundles $\mcE$ on a Fano threefold $X$ of Picard rank $1$, 
it is sufficient to classify them according to the parity of their first Chern class $c_{1}(\mcE) \pmod{2}$. 
In particular, for $X=\Q^{3}$, we assume $\mcE$ is \emph{normalized}, which means $c_{1}(\mcE) \in \{-1,0\}$, to give our classification. 
As in \cite{FHI2}, we give a global resolution for every $\mcE$ using line bundles and the Spinor bundle $\mcS$ on $\Q^{3}$. 
Our first main result is as follows. 

\begin{thm}\label{main-Q3}
Let $\mcE$ be a normalized rank two vector bundle on $\Q^{3}$. 
Then $\mcE$ is weak Fano if and only if $\mcE$ is isomorphic to one of the following. 
\begin{enumerate}[label=(\roman*)]
\item $\mcO_{\Q^{3}}(a) \oplus \mcO_{\Q^{3}}(b)$ with $(a,b) = (-1,1)$, $(-2,1)$, $(-1,0)$, or $(0,0)$. 
\item The pull-back of the null-correlation bundle on $\P^{3}$ via a double covering $\Q^{3} \to \P^{3}$, which fits into 
$0 \to \mcO_{\Q^{3}}(-2) \to \mcO_{\Q^{3}}(-1)^{\oplus 4} \to \mcO_{\Q^{3}}^{\oplus 5} \to \mcE(1) \to 0$.
\item The (negative) spinor bundle $\mcS$.
\item The restriction of a Cayley bundle on $\Q^{5}$ via a linear embedding $\Q^{3} \hookrightarrow \Q^{5}$, which fits into 
$0 \to \mcO_{\Q^{3}}(-2) \to \mcO_{\Q^{3}}(-1)^{\oplus 5} \to \mcO_{\Q^{3}}^{\oplus 2} \oplus \mcS^{\oplus 2} \to \mcE(1) \to 0$.
\item A vector bundle $\mcE$ that fits into 
$0 \to \mcO_{\Q^{3}}(-2)^{\oplus 2} \to \mcO_{\Q^{3}}(-1)^{\oplus 10} \to \mcS^{\oplus 5} \to \mcE(1) \to 0$.
\item A vector bundle $\mcE$ that fits into 
$0 \to \mcO_{\Q^{3}}(-2)^{\oplus 2} \to \mcO_{\Q^{3}}(-1)^{\oplus 7} \to \mcO_{\Q^{3}}^{\oplus 7} \to \mcE(2) \to 0$.
\end{enumerate}
Furthermore, there exist examples for each case (i) -- (vi).
Among the above results, 
only $\mcO_{\Q^{3}}(-1) \oplus \mcO_{\Q^{3}}(1)$, 
$\mcO_{\Q^{3}}(-1) \oplus \mcO_{\Q^{3}}$, 
$\mcO_{\Q^{3}} \oplus \mcO_{\Q^{3}}$, 
(ii), and (iii) are Fano bundles, while the remaining cases are not. 
\end{thm}

\subsection{Application to the moduli space}
As in our previous work \cite{FHI2}, these descriptions can be used to investigate their moduli spaces. 
Let $M^{\wF}_{c_{1},c_{2}}$ (resp. $M_{c_{1},c_{2}}$) be the coarse moduli space of rank $2$ weak Fano (resp. slope stable) bundles $\mcE$ on $\Q^{3}$ with $c_{1}(\mcE)=c_{1}$ and $c_{2}(\mcE)=c_{2}$. 
For a given weak Fano bundle $\mcE$ of rank $2$ with $(c_{1}(\mcE),c_{2}(\mcE))=(c_{1},c_{2})$, 
$\mcE$ is indecomposable if and only if $(c_{1},c_{2}) \in \{(0,2),(-1,1),(-1,2),(-1,3),(-1,4)\}$, and each such $\mcE$ is slope stable (see Proposition~\ref{prop-unstableQ3}). 

Suppose $(c_{1},c_{2}) \in \{(0,2),(-1,1),(-1,2),(-1,3),(-1,4)\}$. 
Then $M^{\wF}_{c_{1},c_{2}}$ is an open subset of $M_{c_{1},c_{2}}$, and $M^{\wF}_{c_{1},c_{2}}=M_{c_{1},c_{2}}$ if $(c_{1},c_{2})=(0,2),(-1,1),(-1,2)$ (see Theorems~\ref{thm-FanoClsf} and \ref{thm-Q3CaylayResol}). 
When $(c_{1},c_{2}) \neq (-1,4)$, Ottaviani and Szurek \cite{OS94} described the moduli space $M_{c_{1},c_{2}}$ and showed that $M^{\wF}_{c_{1},c_{2}}$ is irreducible and smooth (\cite[(2,1), (2,2), (4.1), and (5.2)]{OS94}). 
However, the case when $(c_{1},c_{2})=(-1,4)$ has not yet been addressed. 
To address this case, using the description in Theorem~\ref{main-Q3}~(v), we construct an embedding $M^{\wF}_{-1,4}$ into the moduli space of specific representations of quivers. 
As a result, we obtain the following theorem. 
\begin{thm}\label{main-moduli}
The moduli space $M^{\wF}_{-1,4}$ is irreducible, smooth of dimension $18$, and fine. 
\end{thm}

\subsection{Classification on a Fano threefold whose Fano index is one}

In order to complete the classification of weak Fano bundles of rank two on a Fano threefold $X$ of Picard rank one, 
we consider the case where the Fano index $i_{X}$ is one. 
Our classification result is summarized as follows. 
\begin{thm}\label{main-index1}
Let $X$ be a Fano threefold of Picard rank one whose Fano index is one. Let $g:=\frac{-K_{X}^{3}}{2}+1$, which is called the genus of $X$. 

Then every rank two weak Fano bundle $\mcE$ is isomorphic to one of the following, up to tensoring with a line bundle.
\begin{enumerate}[label=(\roman*)]
\item $\mcO_{X}^{\oplus 2}$. 
\item $\mcO_{X} \oplus \mcO(-K_{X})$. 
\item A globally generated rank two vector bundle $\mcF$ with $c_{1}(\mcF) = c_{1}(X)$ and $\lfloor \frac{g+3}{2} \rfloor \leq -K_{X}.c_{2}(\mcF) \leq g-2$. 
This case arises only when $X$ is a prime Fano threefold of genus $g \geq 6$. 
\end{enumerate}
Moreover, for every prime Fano threefold $X$ of genus $g \geq 6$ and every integer $d$ with $\lfloor \frac{g+3}{2} \rfloor \leq d \leq g-2$, 
there exists a rank two weak Fano bundle $\mcF$ such that $c_{1}(\mcF)=c_{1}(X)$ and $-K_{X}.c_{2}(\mcF) = d$. 
\end{thm}

From the classification results obtained so far \cite{yas, Ishikawa16,FHI1,FHI2} 
and Theorems~\ref{main-Q3} and \ref{main-index1}, 
the following corollary follows immediately. 

\begin{cor}\label{maincor-parity}
Let $X$ be a Fano $3$-fold of Picard rank $1$. 
Let $\mcE$ be a rank $2$ weak Fano bundle on $X$.
\begin{enumerate}
\item If $c_{1}(\mcE) \equiv c_{1}(X) \pmod{2}$, then $\mcF:=\mcE(\frac{c_{1}(X)-c_{1}(\mcE)}{2})$ is globally generated, or $X$ is a del Pezzo $3$-fold of degree $1$ and $\mcF \simeq \mcO_{X}(1)^{\oplus 2}$. 
\item If $c_{1}(\mcE) \not\equiv c_{1}(X) \pmod{2}$, then $\mcE$ is a Fano bundle. 
\end{enumerate}
\end{cor}

As in the above corollary, a weak Fano bundle $\mcE$ of rank two can be distinguished from a Fano bundle by the parity of $c_{1}(X)-c_{1}(\mcE)$. 

\subsection{Geometric feature of rank two weak Fano bundles}

Corollary~\ref{maincor-parity} shows that the difference between Fano and weak Fano bundles of rank two occurs only in the case $c_{1}(\mcF)=c_{1}(X)$. 
Fano bundles $\mcF$ with $c_{1}(X)=c_{1}(\mcF)$ have been considered important in that they give a natural generalization of the Hartshorne conjecture \cite{Mukai-Pair}. 
Weak Fano bundles $\mcF$ with $c_{1}(\mcF)=c_{1}(X)$ are not only simple generalizations, but also possess a historical background in the classification study of prime Fano threefolds.

For example, Gushel \cite{Gushel-genus8} and Mukai \cite{MukaiDev} independently constructed a certain rank two vector bundle $\mcF$ generated by global sections such that $c_{1}(\mcF)=c_{1}(X)$ for a prime Fano threefold $X$ of genus $8$. 
It was classically known that examples of prime Fano threefolds of genus $8$ could be constructed as linear sections of the Grassmannian $\Gr(6,2)$, and by constructing such vector bundles, Gushel and Mukai proved that every prime Fano threefold of genus $8$ can be realized as such a linear section. 
Mukai extended this vector bundle methodology by constructing specific weak Fano vector bundles $\mcF$ with $c_{1}(\mcF)=c_{1}(X)$ for all prime Fano threefolds, and described all Fano threefolds of genus $g \geq 7$ as linear sections of homogeneous spaces \cite{MukaiDev}. 
When $X$ is a prime Fano threefold of genus $g \in \{6,8,10\}$, the constructed vector bundle $\mcF$ is precisely an indecomposable weak Fano bundle of rank $2$ with $c_{1}(\mcF)=c_{1}(X)$ and the smallest $c_{2}$. 
It is also known that when $g \in \{7,9\}$, the moduli space of the indecomposable rank $2$ weak Fano bundles $\mcF$ with $c_{1}(\mcF)=c_{1}(X)$ and the smallest $c_{2}$ coincides with the curve given as the projective dual of $X$ \cite{IM-genus7, IR-LG36},
and the Fourier--Mukai transform using the universal bundle also provides a description of the derived category of $X$ \cite{Kuz05,Kuz06}. 
In summary, for a fixed Fano threefold $X$ of Picard rank one, the classification of weak Fano bundles $\mcF$ of rank two on $X$ is essentially the classification of those with $c_{1}(\mcF)=c_{1}(X)$, 
and studying which kinds of $\mcF$ exist is closely related to studying $X$ itself.

Motivated by these historical backgrounds, we investigate morphisms to Grassmannian varieties induced by $\mcF$. 
Our result is the following theorem. 
\begin{thm}\label{main-emb}
Let $X$ be a Fano threefold of Picard rank one and index one. 
Let $\mcF$ be a weak Fano bundle of rank two with $c_{1}(\mcF)=c_{1}(X)$. 
Then there is a closed embedding 
\[\Phi \colon X \hra \Gr(H^{0}(\mcF),2)\]
with $\mcF \simeq \Phi^{\ast}\mcQ_{\Gr(H^{0}(\mcF),2)}$ if and only if $(X,\mcF)$ satisfies none of the following conditions. 
\begin{enumerate}
\item $X$ is a double cover of $\P^{3}$ 
%ramified along a sextic surface 
and $\mcF$ is isomorphic to the pull-back of $\mcO_{\P^{3}} \oplus \mcO_{\P^{3}}(1)$. 
\item $X$ is a double cover of $\Q^{3}$ 
%ramified along a quartic surface 
and $\mcF$ is isomorphic to the pull-back of $\mcO_{\Q^{3}} \oplus \mcO_{\Q^{3}}(1)$.  
\item $X$ is a double cover of a del Pezzo $3$-fold $V_{5}$ of degree $5$ 
%ramified along a K3 surface, 
and $\mcF$ is isomorphic to the pull-back of the restriction of the rank $2$ quotient bundle $\mcQ_{\Gr(5,2)}$ under the embedding $V_{5} \hra \Gr(5,2)$. 
\end{enumerate}
\end{thm}
If $\mcF$ has no moduli, the embedding $\Phi$ depends only on the isomorphism class of $X$. 
In particular, if $X$ is a prime Fano threefold of genus $g \in \{6,8,10\}$, then $\mcF$ provides the description of $X$ as a linear section of a certain homogeneous variety \cite{Gushel-genus6, Gushel-genus8, MukaiDev}. 
This theorem shows that $\mcF$ also gives an embedding into a Grassmannian variety even if $\mcF$ has moduli. 

\subsection{Outline of our classification}

Let $X$ be a Fano threefold of Picard rank one and the Fano index $i_{X}$ is assumed to be odd.
Let $\mcE$ be a normalized weak Fano bundle of rank $2$. 
For the proof of Theorems~\ref{main-Q3} and \ref{main-index1}, it is important to handle the case $c_{1}(X) \not\equiv c_{1}(\mcE) \pmod{2}$. 
As in Corollary~\ref{maincor-parity}, it is necessary to show that $\mcE$ is a Fano bundle. 
If $\mcE$ is not a Fano bundle, then $M:=\P(\mcE)$ has the contraction $\psi \colon M \to \ol{M}$ of the $K_{M}$-trivial extremal ray. 
If $\mcE$ is semi-stable, then the dimension of every $\psi$-fiber is at most $1$ (Claim~\ref{claim-MOS}). 
Furthermore, a curve that is contracted by $\psi$ is the negative section $\Gamma_{0}$ of $\P(\mcE|_{\Gamma})$ for a rational curve $\Gamma \subset X$. 
Since $c_{1}(X) \not\equiv c_{1}(\mcE) \pmod{2}$, it is easy to see that $\Gamma$ is not a line. 
The most important step is to show that $\Gamma$ is not conic either.
Our proof is roughly as follows. 
Suppose that $\Gamma$ is a conic. 
Then $\dim_{[\Gamma_{0}]}\Hilb(M) \geq 1$.
Hence $\psi$ contracts a surface $S$ onto a curve $C$, 
and furthermore $\Gamma_{0}$ is a fiber of $S \to C$. 
By the same technique as Bend-and-Break, all the fibers of $S \to C$ can be assumed to be irreducible.
This yields a one-dimensional family of non-degenerate conics on $X$. 
On the other hand, the discriminant locus in the Hilbert scheme of conics on $X$ is known to be an ample divisor. 
This contradicts the existence of the non-trivial family $S \to C$ of smooth conics. 
Through this observation, we conclude that $\mcE$ is Fano.
In fact, when $X$ is $\Q^{3}$, it follows that $\mcE$ is Fano by using Sols--Szurek--Wi\'{s}niewski \cite{ssw}. 
When $i_{X}=1$, it follows that $\mcE \simeq \mcO_{X}^{\oplus 2}$, since the family of conics covers $X$.

In the case of $c_{1}(X) \equiv c_{1}(\mcE) \pmod{2}$, 
take an integer $a$ such that $\mcF:=\mcE(a)$ satisfies $c_{1}(\mcF)=c_{1}(X)$. 
Then $\mcF$ is globally generated \cite[Theorem~1.7]{FHI1}. 
The possible values of $c_{2}(\mcF)$ are relatively easy to classify. 
To complete the classification, it is necessary to construct a rank two vector bundle $\mcF$ satisfying $c_{2}(\mcF)=c_{2}$ for each possible value $c_{2}$. 
In virtue of the Hartshorne--Serre correspondence, this problem is equivalent to showing the existence of a particular elliptic curve on $X$.
In the case $X=\Q^{3}$, the existence follows directly from the Arap--Cutrone--Marshburn classification result \cite{ACM17} as in \cite{FHI2}, and the resolution in Theorem~\ref{main-Q3} follows from the description of a full exceptional collection \cite{Kap88} of $\Db(\Q^{3})$. 
In the case $i_{X}=1$, as in \cite{FHI1}, it is necessary to construct such an elliptic curve on an arbitrary $X$. 
Strictly speaking, for a given prime Fano threefold $X$ of genus $g \geq 6$, we need to construct an elliptic curve $C$ such that $\mcI_{C}(-K_{X})$ is globally generated and $-K_{X}.C=d$ for each integer $d$ that satisfies $\lfloor \frac{g-3}{2} \rfloor \leq d \leq g-2$. 
The elliptic curve that we need is always projectively normal.
Moreover, elliptic normal curves on prime Fano threefolds are comprehensively studied by Ciliberto--Flamini--Knutsen \cite{CFK}. 
With the help of their research, our problem reduces to the nefness of a certain Cartier divisor on a K3 surface $\wt{S}$, which is the minimal resolution of a Du Val member $S \in \lvert -K_{X} \rvert$ containing an elliptic normal curve $C$. 
To prove the nefness, we utilize the Brill--Noether generality of this quasi-polarized K3 surface \cite{MukaiDev}, 
and the Ciliberto--Flamini--Knutsen method of studying elliptic normal curves on prime Fano threefolds \cite{CFK}. 
Our proofs will be detailed in Section 7.4 and thereafter.
\subsection{Outline of our proof of Theorem~\ref{main-emb}}
Let $X$ be a Fano threefold of Picard rank one and Fano index one, and $\mcF$ a weak Fano bundle of rank two such that $c_{1}(\mcF)=c_{1}(X)$.
Then our previous result \cite[Theorem~1.7]{FHI1} shows that $\mcF$ is globally generated, and hence it gives a morphism $\Phi_{\lvert \mcF \rvert} \colon X \ni x \mapsto (H^{0}(\mcF) \epm \mcF(x)) \in \Gr(H^{0}(\mcF),2)$. 
To see whether $\Phi_{\lvert \mcF \rvert}$ is a closed embedding, 
following the original argument by Mukai \cite{MukaiDev}, 
it is necessary to show the surjectivity of the map 
$\bigwedge^{2}H^{0}(\mcF) \to H^{0}(\bigwedge^{2}\mcF)$. 
In our case, however, this is not generally surjective. 
Thus, we proceed with another approach as follows. 
If $\mcF$ decomposes into line bundles, then it is clear that whether $\Phi_{\lvert \mcF \rvert}$ is not a closed embedding is characterized by either (1) or (2) in Theorem~\ref{main-emb}. 
Let us assume that $\mcF$ is indecomposable. 
In this case, in order to examine the morphism $\Phi_{\lvert \mcF \rvert}$, we consider the morphism $\Phi_{\lvert \xi \rvert} \colon Y=\P(\mcF) \to \P(H^{0}(\mcF))$, which is given by a tautological divisor $\xi$. 
Note that the crepant contraction $\psi \colon Y \to \ol{Y}$ is the first part of the Stein factorization of $\Phi_{\lvert \xi \rvert}$. 
Since $Y$ is a weak Fano fourfold such that $-K_{Y} \sim 2\xi$, we can check whether $\Phi_{\lvert \xi \rvert}$ gives a birational morphism onto its image in a similar way as the classification method for hyperelliptic Fano threefolds by Iskovskikh \cite{Iskovskikh77}. 
In particular, the degree of $\Phi_{\lvert \xi \rvert}$ is at most $2$. 
Let $Y'$ be the image of the canonical morphism $Y \to \Fl(H^{0}(\mcF);2,1)$. 
Then $\Phi_{\lvert \xi \rvert}$ factors as $Y \to Y' \to \P(H^{0}(\mcF))$. 
Moreover, $Y'$ is a $\P^{1}$-bundle over the image $X' \subset \Gr(H^{0}(\mcF),2)$ of $\Phi_{\lvert \mcF \rvert}$. 
Then, using the above discussion and Lemma~\ref{lem-singfibMella}, we see that $Y'$ is normal at the generic point of all fibers. 
From this, it follows that $X'$ is normal, and that $\Phi_{\lvert \mcF \rvert}$ is either a closed embedding or a double covering. 
If $\Phi_{\lvert \mcF \rvert}$ gives a double covering, then $\Phi_{\lvert \xi \rvert}$ is a degree $2$ morphism onto $\P^{4}$ and hence $Y' \to \P^{4}$ is a birational morphism. 
Therefore, the contraction $\psi$ contracts a divisor. 
Then by the classification result of \cite{JPR05}, it can be shown that $(X,\mcF)$ satisfies (3) of Theorem~\ref{main-emb}.

\subsection{Organization of this paper}

We devote Section~\ref{sec-conic} and Section~\ref{sec-RSGNL} for preliminaries.
In Section~\ref{sec-conic}, we discuss the connection between the weak Fano bundle and the Hilbert scheme of conics, and show Theorem~\ref{thm-conic}. 
This theorem is used in Sections~\ref{sec-Q3Num} and \ref{sec-MukaiEven}.
Section~\ref{sec-RSGNL} summarizes Brill--Noether theory and Noether--Lefschetz type theorems for weak Fano varieties. Results are used in Sections~\ref{sec-Q3Resol} and \ref{sec-MukaiOdd}. 
From Sections~\ref{sec-Q3Num} to \ref{sec-Q3Moduli}, we carry out the classification of weak Fano bundles $\mcE$ of rank $2$ on $\Q^{3}$. 
In Section~\ref{sec-Q3Num}, we restrict the possible values of $(c_{1}(\mcE),c_{2}(\mcE))$. 
In Section~\ref{sec-Q3Resol}, we prove Theorem~\ref{main-Q3} by showing the existence of $\mcE$ satisfying the condition for each possible value of $(c_{1},c_{2})$, and by providing an explicit decomposition of these bundles.
Using this explicit decomposition, we prove Theorem~\ref{main-moduli} in Section~\ref{sec-Q3Moduli}.
In Sections~\ref{sec-MukaiEven} and \ref{sec-MukaiOdd}, we carry out the classification on a Fano threefold $X$ of Picard rank $1$ and index $1$. 
In Section~\ref{sec-MukaiEven}, we treat the case where $c_{1}(\mcE) \not\equiv c_{1}(X) \pmod{2}$. 
This proves (1) of Theorem~\ref{main-index1}. 
Consequently, Corollary~\ref{maincor-parity} is also proved at this stage.
In Section~\ref{sec-MukaiOdd}, we deal with the case where 
$c_{1}(\mcE) \equiv c_{1}(X) \pmod{2}$. 
This proves the rest of Theorem~\ref{main-index1}.
Finally, in Section~\ref{sec-MukaiEmb}, we investigate the morphism associated with a weak Fano bundle $\mcF$ of rank 2 such that $c_{1}(\mcF)=c_{1}(X)$ and prove Theorem~\ref{main-emb}.

\begin{ACK}
The first author is grateful to Professor Hiromichi Takagi and Yoshinori Gongyo for their guidance during his doctoral studies. 
The second author was supported by World Premier International Research Center Initiative (WPI),
MEXT, Japan, and by JSPS KAKENHI Grant Number JP24K22829.
\end{ACK}

\begin{NaC}
Throughout this article, we will work over the complex number field $\C$. 
We also adopt the following conventions. 
\begin{itemize}
\item We regard vector bundles as locally free sheaves. 
For a locally free sheaf $\mcE$ on a smooth projective variety $X$, 
we define $\P(\mcE):=\Proj \Sym \mcE$. 
\item Let $\mcE$ be a locally free sheaf, $\pi \colon \P(\mcE) \to X$ the projection, $\xi$ a tautological divisor. 
In this paper, the $i$-th Chern class $c_{i}(\mcE)$ is defined to satisfy the Grothendieck relation
$\sum_{i=0}^{\rk \mcE} (-1)^{i} \pi^{\ast}c_{i}(\mcE) \xi^{\rk \mcE-i}=0$. 
The $i$-th Segre class is defined to be $s_{i}(\mcE):=\pi_{\ast}\xi^{\rk \mcE+i-1}$. 
\item Let $X$ be a smooth Fano $3$-fold of Picard rank $1$. 
We often identify $N^{i}(X)_{\Z} \simeq \Z$ by taking the effective generator class for each $i \in \{0,1,2,3\}$, 
where $N^{i}(X)_{\Z}$ is the numerical class group of the codimension $i$ cycles with $\Z$-coefficients. 
A positive generator $H_{X}$ in $\Pic(X)$ is called a \emph{fundamental divisor}. 
The $i$-th Chern class $c_{i}(\mcE)$ and the $i$-th Segre class $s_{i}(\mcE)$ are often denoted as $c_{i}$ and $s_{i}$ respectively if they are regarded as integers through the above identification. 
\item 
We also say that a rank $2$ vector bundle $\mcE$ on a Fano $3$-fold $X$ of Picard rank $1$ is \emph{normalized} if $c_{1}(\mcE) \in \{0,-1\}$. 
Note that if $\mcE$ is a weak Fano bundle, then so is $\mcE(n):=\mcE \otimes \mcO_{X}(nH_{X})$ for any $n \in \Z$.
Thus when one studies a given weak Fano bundle $\mcE$ of rank $2$, 
it can be assumed without loss of generality that $\mcE$ is normalized.
\item $\Q^{n}$ denotes the smooth quadric hypersurface of $\P^{n+1}$. 
\item The \emph{(negative) spinor bundle} on $\Q^{3}$ is denoted by $\mcS$ \cite{Ottaviani88}. 
In our notation, $\mcS$ is a slope stable bundle with $c_{1}(\mcS)=-1$, $c_{2}(\mcS)=1$, and $\rk \mcS=2$. 
\item A \emph{Cayley bundle} $\mcC$ on $\Q^{5}$ is defined to be a slope stable rank $2$ bundle with $c_{1}(\mcC)=-1$ and $c_{2}(\mcC)=1$ \cite{Ott90}. 
\item A \emph{null-correlation bundle} on $\P^{3}$ is denoted by $\mcN$ \cite{OSS80}. 
\item For the Grassmannian varieties, we employ the quotient notation as follows; for an $N$-dimensional vector space $V=\C^{N}$ and a positive integer $m<N$, 
we define the Grassmannian variety $\Gr(V,m)=\Gr(N,m)$ as the parameter space of $m$-dimensional linear quotient spaces of $V=\C^{N}$. 
In particular, if we regard $V$ as a locally free sheaf on $\Spec \C$, 
then $\P(V)$ is canonically isomorphic to $\Gr(V,1)$ in our notation. 
We also employ the quotient notation for the flag varieties.
\item Let $X$ be a smooth Fano $3$-fold of Picard rank $\rho(X)=1$ and $H_{X}$ its fundamental divisor. 
The \emph{Fano index} is the positive integer $i_{X}$ with $-K_{X} \sim i_{X} \cdot H_{X}$. 
If $\rho(X)=i_{X}=1$ and $-K_{X} \sim H_{X}$ is very ample, 
we call $X$  a \emph{prime Fano $3$-fold}. 
\end{itemize}
\end{NaC}

\section{Weak Fano bundles and Hilbert scheme of conics}\label{sec-conic}

\begin{defi}\label{def-dag}
Let $X$ be a smooth projective variety and $H_{X}$ a very ample divisor. 
We say a curve $C \subset X$ is a \emph{conic} (resp. \emph{line}) with respect to $H_{X}$ if the Hilbert polynomial $p_{C,H_{X}}(t):=\chi(C,\mcO(tH_{X})|_{C})$ is $2t+1$ (resp. $t+1$). 
It is well-known that every conic $\gamma$ is either a smooth rational curve, a union of two lines meeting at one point, or a double line on $\P^{2}$. 
In this paper moreover, we say $(X,H_{X})$ satisfies $(\dag)$ if the following condition is satisfied.
\begin{itemize}
\item[$(\dag)$] For every irreducible component $Z$ of the Hilbert scheme $\Hilb_{2t+1}(X,H_{X})$ of conics, 
if the open set $Z^{\sm}:=\{[\gamma] \mid \gamma \text{ is smooth }\}$ is non-empty, then $Z^{\sm}$ does not contain any proper curve.
\end{itemize}
\end{defi}

\begin{rem}[discriminant divisors]\label{rem-discriminant}
In the setting in Definition~\ref{def-dag}, 
let $V:=H^{0}(X,\mcO_{X}(H_{X}))$ and $i \colon X \hra \P(V)$ the embedding given by $\lvert H_{X} \rvert$. 
Then $\Hilb_{2t+1}(\P(V),\mcO_{\P(V)}(1))$ is isomorphic to $\P_{\Gr(V,3)}(\Sym^{2}\mcQ^{\vee})$, where $\mcQ$ is the universal quotient bundle on $\Gr(V,3)$ of rank $3$.  
Let $\pi \colon \P_{\Gr(V,3)}(\Sym^{2}\mcQ^{\vee}) \to \Gr(V,3)$ be the projection. 
The singular conics on $\P(V)$ are parametrized by the discriminant divisor $\Delta \subset \P_{\Gr(V,3)}(\Sym^{2}\mcQ^{\vee})$, 
which is the member of $\lvert \mcO_{\pi}(3) \otimes (\pi^{\ast}\det \mcQ)^{\otimes 2} \rvert$ given by the determinant of the natural map 
$\pi^{\ast}\mcQ^{\vee} \to \pi^{\ast}\mcQ \otimes \mcO_{\pi}(1)$
corresponding to $H^{0}(\mcO_{\pi}(1) \otimes \Sym^{2}\mcQ) \simeq \Hom(\Sym^{2}\mcQ,\Sym^{2}\mcQ) \ni \id$. 

For each irreducible component $Z \subset \Hilb_{2t+1}(X,H_{X})$ with the reduced structure, 
$Z$ is also a closed subscheme of $\Hilb_{2t+1}(\P(V),\mcO_{\P(V)}(1))$. 
The \emph{discriminant locus} $\Delta_{Z}$ is defined to be the scheme theoretic intersection $\Delta \cap Z$. 
By definition, $\Delta_{Z}$ is either an effective Cartier divisor on $Z$ or equal to $Z$ itself. 
In this viewpoint, the condition $(\dag)$ is equivalent to saying that $Z \setminus \Delta_{Z}$ does not contain any proper curve if $\Delta_{Z} \neq Z$. 
For example, if $\Delta_{Z}$ is ample, then the condition $(\dag)$ holds.
\end{rem}

The aim of Section~\ref{sec-conic} is to establish the following theorem. 
\begin{thm}\label{thm-conic}
Let $X$ be a smooth Fano variety of dimension $n \geq 3$ with $b_{2}(X)=b_{4}(X)=1$. 
Let $H_{X}$ be a fundamental divisor on $X$ and $i_{X}$ be the index of $X$, i.e., $-K_{X} \sim i_{X}H_{X}$. 
Suppose that $H_{X}$ is very ample, $X$ has a smooth conic with respect to $H_{X}$, and $(X,H_{X})$ satisfies $(\dag)$. 

Let $\mcE$ be a rank $2$ weak Fano bundle with $c_{1}(\mcE) = (i_{X}-1)H_{X}$. 
If $\mcE$ is $H_{X}$-slope semistable, then the following assertions hold. 
\begin{enumerate}
\item $\mcE|_{l}$ is nef for every line $l$ and $\mcE|_{C}$ is nef for every smooth conic $C$. 
\item If $i_{X}=1$, then $\mcE \simeq \mcO_{X}^{\oplus 2}$. 
\end{enumerate}
\end{thm}

\subsection{An observation for $(2,1)$-type crepant contractions from almost Fano fourfolds}

First, we present the following proposition to study crepant contractions from almost Fano $4$-folds whose fibers are at most $1$-dimensional. 
\begin{prop}\label{prop-semiAtiyah}
Let $Y$ be an almost Fano $n \geq 4$-fold, i.e. an $n$-dimensional smooth variety that is weak Fano but not Fano.
Let $\psi \colon Y \to \ol{Y} := \Proj R(Y,-K_{Y})$ be the contraction associated to $-K_{Y}$.
Suppose that $\dim \psi^{-1}(y) \leq 1$ for every closed point $y \in \ol{Y}$.

Then, given an irreducible curve $\Gamma_{0}$ contracted by $\psi$
and an ample divisor $A$ on $Y$, there exist
\begin{itemize}
\item a $\P^{1}$-bundle $p \colon S \to C$ over a smooth curve $C$,
\item a generically finite morphism $f \colon S \to Y$, and a finite morphism $g \colon C \to \ol{Y}$
\end{itemize}
satisfying the following properties. 
\begin{enumerate}
\item The following diagram commutes.
\[
\begin{tikzcd}
S \arrow[d, "p"'] \arrow[r, "f"] & Y \arrow[d, "\psi"] \\
C \arrow[r, "g"'] & \ol{Y}.
\end{tikzcd} 
\]

\item Any fiber $\gamma$ of $p$ satisfies $f^{\ast}A . \gamma \leq A . \Gamma_{0}$.
\end{enumerate}
\end{prop}

\begin{proof}
The proof of the proposition is divided into three steps.

\textit{Step 1.}
This step shows that any irreducible curve $\Gamma \subset Y$ contracted by $\psi$ is a smooth rational curve.
Indeed, since $R^{1}\psi_{\ast} \mcO_{Y}  = 0$ by the Kawamata--Viehweg vanishing 
and $R^{2}\psi_{\ast} \mcI_{\Gamma/Y} = 0$ by the dimension assumption,
it holds that $H^{1}(\mcO_{\Gamma}) \simeq R^{1}\psi_{\ast}\mcO_{\Gamma} = 0$, 
which implies that $\Gamma$ is a smooth rational curve.
In particular, the curve $\Gamma_{0}$  in the assumption is a smooth rational curve.

\textit{Step 2.}
In this step, we prove the following claim.
\begin{claim}\label{claim-moving}
With the same assumption as in Proposition~\ref{prop-semiAtiyah},
given an irreducible curve $\Gamma_{0}$ contracted by $\psi$ and an ample divisor $A$,
there exist 
\begin{itemize}
\item a normal projective surface $S$ and a pointed smooth curve $(C, x_{0})$ with a surjective morphism $p \colon S \to C$,  
\item a generically finite morphism $f \colon S \to Y$, and a finite morphism $g \colon C \to \ol{Y}$
\end{itemize} 
such that
\begin{enumerate}
\item the following diagram commutes;
\begin{equation}\label{dia-moving}
\begin{tikzcd}
S \arrow[d, "p"'] \arrow[r, "f"] & Y \arrow[d, "\psi"] \\
C \arrow[r, "g"'] & \ol{Y}.
\end{tikzcd} 
\end{equation}
\item the scheme-theoretic fiber $p^{-1}(x_{0})$ is a smooth rational curve, 
\item $f$ restricts to an isomorphism $f|_{p^{-1}(x_0)} \colon p^{-1}(x_0) \xrightarrow{\sim} \Gamma_{0}$, and 
\item $f^{\ast}A$ is $p$-ample. 
\end{enumerate}
\end{claim}

\begin{proof}
First, since $\Gamma_{0}$ is a smooth rational curve, 
\cite[Chap. II, Theorem 1.14]{Kollar} shows that $\dim_{[\Gamma_{0}]} \Hilb(Y) \geq n-3 \geq 1$. 
Thus there exists a proper irreducible curve $T \subset \Hilb(Y)$ with $[\Gamma_{0}] \in T$.
Let $C$ be the normalization of $T$, 
$x_{0} \in C$ a point over $[\Gamma_{0}] \in T$,
and $\pi_{C} \colon U_{C} \to C$ the pull-back of the universal family over $\Hilb(Y)$.
Note that $U_{C}$ is of dimension two. 
Put
\[ C^{0} := \{ t \in C \mid \text{$\pi^{-1}_{C} (t)$ is geometrically integral} \}. \]
Let us check that $\pi_{C^{0}} \colon U_{C^{0}} \to C^{0}$ is a $\P^{1}$-bundle, i.e., $U_{C^{0}}$ is a smooth irreducible surface and every fiber of $\pi_{C^{0}}$ is isomorphic to $\P^{1}$. 
Since $\pi_{C}$ is a flat and proper morphism and $\pi_{C}^{-1}(x_{0})$ is a smooth rational curve, 
the subset $C^{0} \subset C$ is open and contains $x_{0}$. 
Let $\pi_{C^{0}} \colon U_{C^{0}} \to C^{0}$ be the restriction of $\pi_{C}$.
Since every $\pi_{C^{0}}$-fiber and the base $C^{0}$ are integral, so is $U_{C^{0}}$.
Since the $\pi_{C^{0}}$-fiber of $x_{0}$ is isomorphic to $\P^{1}$, the cohomology and base change theorem shows that every $\pi_{C^{0}}$-fiber is also isomorphic to $\P^{1}$. 
Since every $\pi_{C^{0}}$-fiber is also an effective Cartier divisor on $U_{C^{0}}$, the surface $U_{C^{0}}$ is smooth and $\pi_{C^{0}} \colon U_{C^{0}} \to C^{0}$ is a $\P^{1}$-bundle. 

Let $\ol{U_{C^{0}}} \subset U_{C}$ be the closure of $U_{C^{0}}$ with the reduced scheme structure and $S \to \ol{U_{C^{0}}}$ be the normalization. 
The composition morphism $S \to \ol{U_{C^{0}}} \to U_{C} \to Y$ is denoted by $f \colon S \to Y$, 
and the composition $S \xrightarrow{\mu} \ol{U_{C^{0}}} \hra U_C \xrightarrow{\pi_{C}} C$ by $p \colon S \to C$. 
We now obtain the following diagram.
\[ \begin{tikzcd}
S \arrow[rd, "p"'] \arrow[r] \arrow[rr, bend left=40, "f"]& U_C \arrow[d, "\pi_C"] \arrow[r] & Y \arrow[d, "\psi"] \\
& C  & \ol{Y}.
\end{tikzcd} \]
Here we check that $f$ is generically finite and the condition (4). 
Since the morphism $S \to \ol{U_{C^{0}}}$ is finite, 
it suffices to show that $\ol{f} \colon \ol{U_{C^{0}}} \to Y$ is generically finite and the restriction of $\ol{f}$ to every fiber under $\ol{p} \colon \ol{U_{C^{0}}} \to C$ is a closed embedding. 
Since $\ol{p} \times \ol{f} \colon \ol{U_{C^{0}}} \to C \times Y$ is closed embedding, so is the restriction of $\ol{f}$ to each $\ol{p}$-fiber. 
Moreover, since the image of $\ol{f}$ containing $\Gamma_{0}$, 
if the dimension of the image of $\ol{U_{C^{0}}} \to Y$ is one, 
then $\ol{U_{C^{0}}} \simeq C \times \Gamma_{0}$, which contradicts our construction of $T$.

Note that $p^{-1}(x_{0})$ is a smooth rational curve by construction, 
and $f$ restricts to an isomorphism $f|_{p^{-1}(x_{0})} \colon p^{-1}(x_{0}) \xrightarrow{\sim} \Gamma_{0}$. Moreover, $\psi \circ f$ contracts $p^{-1}(x_{0})$, and hence applying the rigidity lemma \cite[Lemma~1.15]{Debarre} shows that 
there exists a morphism $g \colon C \to \ol{Y}$ such that the diagram (\ref{dia-moving}) commutes. 
Note that $g \colon C \to Y$ is finite; otherwise the image $y = g(C)$ is a point, and thus the fiber $\psi^{-1}(y)$ contains a surface $f(S)$, which contradicts that we assume every $\psi$-fiber is at most $1$-dimensional. 
This completes the proof of the claim.
\end{proof}

\textit{Step 3.}
This step completes the proof of this proposition.
First, applying the claim gives a commutative diagram as in (\ref{dia-moving}). 
If $p \colon S \to C$ is already a $\P^{1}$-bundle, the proposition holds.
Otherwise, there exists a point $x \in C$ such that the scheme-theoretic fiber $p^{-1}(x)$ is not integral.
Put $p^{-1}(x) = \sum_{i} a_{i} D_{i}$, where $D_{i}$ is a prime divisor on $S$ for all $i$.

Let $D_{i_{0}}$ be an arbitrary irreducible component of $p^{-1}(x)$. 
By Claim~\ref{claim-moving}~(4), $\Gamma_{1}:=f(D_{i_{0}})$ is an irreducible curve. 
Then $\Gamma_{1}$ is smooth rational curve since $\Gamma_{1}$ is contracted by $\psi$. 
Again by Claim~\ref{claim-moving}~(4), the inequality $A.\Gamma_{1} <  A.\Gamma_{0}$ holds since
\[ A.\Gamma_{1} \leq A. f_{\ast}D_{i_{0}} = f^{\ast}A.D_{i_{0}} < f^{\ast}A.p^{-1}(x_{0}) = A.\Gamma_{0}. \]
Replacing the fixed curve $\Gamma_{0}$ with $\Gamma_{1}$ and applying 
Claim~\ref{claim-moving} give
another $(S, C, x_{0}, f, g, p)$.
Repeating this procedure yields a decreasing sequence
\[ (0 \leq) \cdots < f^{\ast}A. \Gamma_{n} < \cdots <  f^{\ast}A. \Gamma_{1} < f^{\ast}A. \Gamma_{0}, \]
which should terminate at some point.
Thus the induction gives $(S, C, x_{0}, f, g, p)$ such that 
$p \colon S \to C$ is a $\P^{1}$-bundle and that $f^{\ast}A.p^{-1}(x) \leq A.\Gamma_{0}$ for all $x \in C$.
\end{proof}

\subsection{Positivity of weak Fano bundles along conics}

In this section, we prove Theorem~\ref{thm-conic}. 
First, we prove Theorem~\ref{thm-conic}~(1) by establishing the following slightly more general proposition.
\begin{prop}\label{prop-conicnef}
Let $X$ be a smooth weak Fano variety. 
Let $\mcE$ be a rank $2$ weak Fano bundle on $X$ whose anti-adjoint bundle $H_{X}:=-K_{X}-c_{1}(\mcE)$ is very ample. 
Suppose the following two conditions hold. 
\begin{enumerate}
\item $(X,H_{X})$ satisfies $(\dag)$ (see Definition~\ref{def-dag}). 
\item $(K_{X}^{2}+\Delta(\mcE)).T > 0$ for every irreducible surface $T$ on $X$, where $\Delta(\mcE):=4c_{2}(\mcE)-c_{1}(\mcE)^{2}$ is the discriminant class. 
\end{enumerate}
Then for every line $l$ and every smooth conic $C$ with respect to $H_{X}$, $\mcE|_{l}$ and $\mcE|_{C}$ are nef.
\end{prop}
\begin{proof}
Let $\pi \colon Y:=\P(\mcE) \to X$ be the projectivization, $\xi$ a tautological divisor, and $H:=\pi^{\ast}H_{X}$. 
From our assumption, $-K_{Y}=2\xi+H$ is nef. 
Hence $A:=\xi+H$ is ample and $\mcE(H_{X})$ is an ample vector bundle.
In particular,  $\mcE|_{l}$ is nef for every line $l$ on $X$. 
\begin{claim}[{\cite[Lemma~5.3]{MOS14}}]\label{claim-MOS}
Let $\psi \colon Y \to \ol{Y}$ be the contraction given by $-K_{Y}$. 
Then there is no irreducible surface contracted by $\psi$. 
\end{claim}
\begin{proof}
Assume the contrary. 
Then there is an irreducible surface $J \subset Y$ such that $\psi(J)$ is a point. 
Then $\pi|_{J} \colon J \to X$ is finite. 
Since $-K_{Y}|_{J} \sim 0$, it follows that
\begin{align*}
(\pi|_{J})^{\ast}K_{X}^{2}
&= (-K_{\pi}|_{J})^{2} \\
&=(2\xi|_{J}-(\pi|_{J})^{\ast}c_{1}(\mcE))^{2} \\
&=4(\xi|_{J})^{2}-4(\xi|_{J})(\pi|_{J})^{\ast}c_{1}(\mcE)+(\pi|_{J})^{\ast}c_{1}(\mcE)^{2} \\
&=4(\xi|_{J}(\pi|_{J})^{\ast}c_{1}(\mcE)-(\pi|_{J})^{\ast}c_{2}(\mcE))-4\xi|_{J}(\pi|_{J})^{\ast}c_{1}(\mcE)+(\pi|_{J})^{\ast}c_{1}(\mcE)^{2} \\
&=(\pi|_{J})^{\ast}(c_{1}(\mcE)^{2}-4c_{2}(\mcE)) = -(\pi|_{J})^{\ast} \Delta(\mcE).
\end{align*}
Thus $(\pi|_{J})^{\ast}(K_{X}^{2}+\Delta(\mcE))=0$, which contradicts our assumption (2). 
\end{proof}
Let $C$ be a smooth conic on $X$. 
%If $C$ is not smooth, then $C$ is a union of two lines or a double line. 
%Since $\mcE|_{l}$ is nef for every line $l$ on $X$, so is $\mcE|_{C}$. 
%Thus we may assume that $C$ is smooth. 
%Suppose that there is a smooth conic $C$ such that $\mcE|_{C}$ is not nef. 
Suppose that $\mcE|_{C}$ is not nef. 
Write $\mcE|_{C}= \mcO_{\P^{1}}(a) \oplus \mcO_{\P^{1}}(b)$ with $a \leq b$ and $a \leq -1$.  
For a minimal section $\wt{C} \subset \P(\mcE|_{C})$, we have $\xi.\wt{C}=a$ and $H.\wt{C}=2$. 
Since $-K_{\P(\mcE)}=2\xi+H$ is nef, we have $a \geq -1$. 
Then $a=-1$, which means $-K_{\P(\mcE)}.\wt{C}=0$. 
By Proposition~\ref{prop-semiAtiyah}, there is a smooth curve $B$ and a $\P^{1}$-bundle $p \colon S \to B$ such that every $p$-fiber $\gamma$ satisfies 
$0 < f^{\ast}A.\gamma \leq A.\wt{C}$, 
where $A=\xi+H$; 
\[
\begin{tikzcd}
S \arrow[r,"f"] \arrow[d,"p"']&Y=\P_{X}(\mcE) \arrow[r,"\pi"] \arrow[d,"\psi"]& X \\
B \arrow[r,"g"']&\ol{Y}.&
\end{tikzcd}
\]
Since $A.\wt{C}=1$, 
%Since $H.\wt{C}=2$ and $(2\xi+H).\wt{C}=0$, we have $\xi.\wt{C}=-1$ and hence $A.\wt{C}=1$. 
we have $f^{\ast}A.\gamma = 1$ and $f^{\ast}H.\gamma=2$ for every $p$-fiber $\gamma$. 
Define $e:=\pi \circ f \colon S \to X$. 
For a $p$-fiber $\gamma$, if $e(\gamma)$ is a line, then $\mcE|_{e(\gamma)}$ is nef and hence $-K_{Y}|_{\pi^{-1}(e(\gamma))}=(2\xi+H)|_{\pi^{-1}(e(\gamma))}$ is ample, which is a contradiction. 
Thus $H_{X}.e(\gamma)=2$, i.e., $e(\gamma)$ is a conic on $X$ and $e|_{\gamma} \colon \gamma \to e(\gamma)$ is an isomorphism. 
Hence the morphism $(p,e) \colon S \to B \times X$ is a closed embedding over $B$. 
Hence $p \colon S \to B$ is a non-trivial flat family of conics on $X$ and is the base change of $\Univ_{2t+1}(X,H_{X}) \to \Hilb_{2t+1}(X,H_{X})$ under the induced finite morphism $u \colon B \to \Hilb_{2t+1}(X,H_{X})$. 
%If $u$ is a point, then $e(S)$ is a conic, which is a contradiction. 
%Hence $u$ is finite and $u(B)$ is a curve. 
Let $Z$ be the irreducible component containing the image $u(B)$. 
Since every $p$-fiber is a smooth reduced rational curve, a closed point in $u(B)$ corresponds to a smooth conic on $X$. 
Hence the discriminant locus $\Delta_{Z}$ is a proper closed subscheme on $Z$ and $\Delta_{Z} \cap u(B) = \emp$. 
%Since $p$ is a $\P^{1}$-bundle, we have $\Delta_{Z} \cap u(B)=\emp$. 
This contradicts our assumption $(\dag)$ in Definition~\ref{def-dag} that $Z^{\sm}=Z \setminus \Delta_{Z}$ never contains a proper curve.
Therefore, $\mcE|_{C}$ is nef. 
\end{proof}

To show Theorem~\ref{thm-conic}~(2), we prepare the following lemma. 

\begin{lem}\label{lem-conictriv}
Let $X \subset \P^{n}$ be a projective variety and $\mcE$ a rank $2$ vector bundle on $X$. 
Suppose that 
$\mcE|_{l} \simeq \mcO_{l}^{\oplus 2}$ for every line $l$. 
Then $\mcE|_{C} \simeq \mcO_{C}^{\oplus 2}$ for every singular conic $C$. 
\end{lem}
\begin{proof}
Let $C$ be a singular conic on $X$. 
First, we treat the case when $C$ is not irreducible and reduced.
Let $l_{1}$ and $l_{2}$ be the irreducible components of $C$, and define $p:=l_{1} \cap l_{2}$. 
Then, we obtain the following exact sequence 
\begin{align}\label{ex-unionofline}
0 \to \mcO_{C} \to \mcO_{l_{1}} \oplus \mcO_{l_{2}} \to \Bbbk(p) \to 0.
\end{align}
Note that $\mcE|_{l} \simeq \mcO_{l}^{\oplus 2}$ for any line $l$. 
Tensoring with $\mcE$, we obtain 
\[0 \to \mcE|_{C} \to \mcO_{l_{1}}^{\oplus 2} \oplus \mcO_{l_{2}}^{\oplus 2} \xrightarrow{\alpha} \Bbbk(p)^{\oplus 2} \to 0.\]
Note that $H^{0}(\alpha)$ is given by $(s_{1},s_{2}) \mapsto s_{1}(p)-s_{2}(p)$ and hence $H^{0}(\alpha)$ is surjective. 
Thus $\RG(\mcE|_{C})=\C^{\oplus 2}$ and $H^{0}(\mcE|_{C})=\Ker H^{0}(\alpha)=\{(s_{1},s_{2}) \in H^{0}(\mcE|_{l_{1}} \oplus \mcE|_{l_{2}}) \mid s_{1}(p)-s_{2}(p)=0\}$. 
If $s \in H^{0}(\mcE|_{C})$ satisfies $s|_{l_{1}}=0$, then $s(p)=0$ and hence $s|_{l_{2}}(p)=0$, which implies $s=0$. 
Thus for a non-zero section $s \in H^{0}(\mcE|_{C})$, 
both $s|_{l_{1}}$ and $s|_{l_{2}}$ are non-zero.
Hence a non-zero global section $s$ of $\mcE|_{C}$ is nowhere vanishing. 
Thus the cokernel of $s \colon \mcO_{C} \to \mcE|_{C}$ is an invertible sheaf $\mcL_{C}$ on $C$. 
Since $\RG(\mcO_{C})=\C$ and $\RG(\mcE|_{C})=\C^{\oplus 2}$, 
we have $\RG(\mcL_{C})=\C$. 
From the exact sequence (\ref{ex-unionofline}), 
tensoring with $\mcL_{C}$, we obtain 
\[0 \to \mcL_{C} \to \mcL_{C}|_{l_{1}} \oplus \mcL_{C}|_{l_{2}} \to \Bbbk(p) \to 0.\]
For each $i \in \{1,2\}$, $\mcL_{C}|_{l_{i}}$ is a quotient of $\mcE|_{l_{i}}=\mcO_{l_{i}}^{\oplus 2}$, $\mcL_{C}|_{l_{i}}$ is nef. 
Since $\RG(\mcL_{C}|_{l_{1}} \oplus \mcL_{C}|_{l_{2}}) = \C^{\oplus 2}$, 
$\mcL_{C}|_{l_{i}} \simeq \mcO_{l_{i}}$. 
Thus $\mcL_{C} \simeq \mcO_{C}$. 
Since $H^{1}(\mcO_{C})=0$, we have $\mcE|_{C} \simeq \mcO_{C}^{\oplus 2}$. 

Finally, we treat the case $C$ is non-reduced. 
Let $l:=C_{\red}$. 
On the plane $\Pi=\gen{C} \simeq \P^{2}$, 
there exist $0 \to \mcO_{\P^{2}}(-2) \to \mcO_{\P^{2}} \to \mcO_{C} \to 0$ and $0 \to \mcO_{\P^{2}}(-1) \to \mcO_{\P^{2}} \to \mcO_{l} \to 0$. 
By the snake lemma, we have 
\[0 \to \mcO_{l}(-1) \to \mcO_{C} \to \mcO_{l} \to 0.\]
Then we have the following exact sequences on $C$: 
\[
\begin{tikzcd}
0 \arrow[r] & 0 \arrow[r] \arrow[d] & H^{0}(\mcE|_{C}) \otimes \mcO_{C} \arrow[r] \arrow[d] & H^{0}(\mcE|_{l}) \otimes \mcO_{C} \arrow[r] \arrow[d,"\ev"] & 0 \\
0 \arrow[r] & \mcE|_{l}(-1) \arrow[r] & \mcE|_{C} \arrow[r]  & \mcE|_{l} \arrow[r] & 0.
\end{tikzcd}
\]
Since $\mcE|_{l} \simeq \mcO_{l}^{\oplus 2}$, 
the restriction morphism $H^{0}(\mcE|_{C}) \to H^{0}(\mcE|_{l}) \simeq \C^{2}$ is isomorphic. 
Then by Nakayama's lemma, it suffices to show the canonical map $H^{0}(\mcE|_{C}) \otimes \mcO_{C} \to \mcE|_{C}$ is surjective. 
Note that the natural map $H^{0}(\mcE|_{l}) \otimes \mcO_{C} \to \mcE|_{l} \simeq \mcO_{l}^{\oplus 2}$ is surjective and its kernel is $\mcO_{l}(-1)^{\oplus 2}$. 
Hence we obtain a connecting morphism $\delta \colon \mcO_{l}(-1)^{\oplus 2} \to \mcO_{l}(-1)^{\oplus 2}$. 
Then there is $a \in \{0,1,2\}$ such that $\Cok(\delta)=\mcO_{l}(-1)^{\oplus a}$. 
Then there is a surjection $\mcE|_{C} \epm \mcO_{l}(-1)^{\oplus a}$. 
Restricting to $l$, we obtain a surjection $\mcE|_{l} \epm \mcO_{l}(-1)^{\oplus a}$, which implies $a=0$. 
Thus $\mcE|_{C} \simeq \mcO_{C}^{\oplus 2}$. 
\end{proof}

\noindent\emph{\textbf{Proof of Theorem~\ref{thm-conic}}}.
Theorem~\ref{thm-conic}~(1) is a direct corollary of Proposition~\ref{prop-conicnef}. 
Let us show Theorem~\ref{thm-conic}~(2). 
Let $X$ be a smooth Fano variety of dimension $n \geq 3$ with $b_{2}(X)=b_{4}(X)=1$. 
Suppose that $-K_{X}$ is very ample, $X$ has a smooth conic with respect to $-K_{X}$, and $(X,-K_{X})$ satisfies $(\dag)$. 
Let $\mcE$ be a $(-K_{X})$-slope semistable rank $2$ weak Fano bundle with $c_{1}(\mcE)=0$. 
The aim is to show $\mcE \simeq \mcO_{X}^{\oplus 2}$. 

Let $\Gamma$ be a smooth conic on $X$. 
Let $Z$ be an irreducible component of $\Hilb_{2t+1}(X,-K_{X})$ containing $[\Gamma]$. 
Denote the universal family by $p \colon U \to Z$ and the evaluation morphism by $e \colon U \to X$; 
%By the conditions, there exists irreducible component $Z \subset \Hilb_{2t+1}(X,-K_{X})$ such that the evaluation morphism $e \colon U \to X$ is surjective, where $p \colon U \to Z$ is the universal family:
\begin{equation}\label{dia-conic}
\begin{tikzcd}
U \arrow[r,"e"] \arrow[d,"p"']& X \\
Z. 
\end{tikzcd}
\end{equation}

Since $\mcE$ is $(-K_{X})$-slope semistable and $b_{4}(X)=1$, 
Bogomolov's inequality shows $\Delta(\mcE)=4c_{2}(\mcE)-c_{1}(\mcE)^{2}$ satisfies $\Delta(\mcE).T \geq 0$ for every irreducible surface $T$. 
Thus we can use Proposition~\ref{prop-conicnef} for $(X,\mcE)$ and hence $\mcE|_{l}$ is nef for every line $l$, and $\mcE|_{\gamma}$ is nef for every smooth conic $\gamma$. 
Since $c_{1}(\mcE)=0$, $\mcE|_{l} \simeq \mcO_{l}^{\oplus 2}$ for every line $l$ and hence by Lemma~\ref{lem-conictriv}, $\mcE|_{\gamma} \simeq \mcO_{\gamma}^{\oplus 2}$ for every (possibly singular) conic $\gamma$. 

Then Grauert's theorem shows $\mcG:=p_{\ast}e^{\ast}\mcE$ is a locally free sheaf of rank $2$ and the natural morphism $p^{\ast}\mcG \to e^{\ast}\mcE$ is isomorphic. 
In particular, $p^{\ast}c_{2}(\mcG) \equiv e^{\ast}c_{2}(\mcE)$. 
By \cite[Chap.~II, Theorem~1.14]{Kollar}, $\dim Z \geq n-1 \geq 2$. 
Then for every irreducible curve $C \subset Z$, we have 
\begin{align*}
0
&=c_{2}(p^{\ast}\mcG).p^{-1}(C) \\
&=c_{2}(e^{\ast}\mcE).p^{-1}(C) \\
&=a \cdot e^{\ast}(-K_{X})^{2}.p^{-1}(C),
\end{align*}
where $a \in \Q_{\geq 0}$. 
If $a>0$, then $p^{-1}(C).e^{\ast}(-K_{X})^{2}=0$, which implies that $\dim e(p^{-1}(C)) \leq 1$ for every irreducible curve $C$ on $Z$. 
This leads a contradiction, and hence $c_{2}(\mcE) = 0$ in $H^{4}(X,\Q)$. 
Then the Hirzebruch--Riemann--Roch formula implies $\chi(\mcE)=2$. 
Since $h^{\geq 2}(\mcE)=0$ follows from the Le Potier vanishing theorem \cite[Theorem~7.3.5]{laz2}, 
we have $h^{0}(\mcE) \geq 2 > 0$.
Take a non-zero section $s \colon \mcO \to \mcE$. 
Since $\mcE$ is slope semistable and $b_{2}(X)=1$, the cokernel of $s$ is isomorphic to the ideal sheaf $\mcI_{W}$ for a closed subscheme $W$ which is of purely codimension $2$ or empty. 
Since $c_{2}(\mcE) \equiv [W]$, $W$ is empty and hence $\mcE \simeq \mcO_{X}^{\oplus 2}$. 
This completes the proof. \qed

\begin{rem}
In the above proof, the condition $b_{4}(X)=1$ could be weakened by assuming that $c_{2}(\mcE)$ is nef in $N^{2}(X)_{\R}$ and the cylinder homomorphism $e_{\ast} \circ p^{\ast} \colon H_{2}(Z,\Q) \to H_{4}(X,\Q)$ is surjective for a family of conics as in (\ref{dia-conic}). 
\end{rem}

\section{Noether--Lefschetz theorems for weak Fano varieties}\label{sec-RSGNL}

As further preliminaries, we summarize here the Brill--Noether theory for K3 surfaces and the Grothendieck--Noether--Lefschetz-type theorem generalized by Ravindra--Srinivas, and review some results on linear systems of weak Fano $3$-folds derived from those theories.

\subsection{Preliminaries for K3 surfaces}\label{subsec-K3prelim}
First, we review some results from the Brill--Noether theory for K3 surfaces.

\begin{defi}[Brill--Noether general \cite{MukaiDev}]
A \emph{quasi-polarized K3 surface $(S,H)$ of genus $g$} is a pair of a smooth K3 surface $S$ and a nef big divisor $H$ on $S$ with $H^{2}=2g-2$. 
\begin{itemize}
\item A smooth curve $C$ of genus $g$ is said to be \emph{Brill--Noether general} if 
\[\forall \mcL \in \Pic(C), h^{0}(\mcL) \cdot h^{1}(\mcL) \leq g.\]
\item 
A quasi-polarized K3 surface $(S,H)$ is said to be \emph{Brill--Noether general} if 
\[\forall L \in \Pic(S) \setminus \{0,H\}, h^{0}(H-L) \cdot h^{0}(L) \leq g.\]
\end{itemize}
\end{defi}
Here we summarize important results about Brill--Noether theory of K3 surfaces. 
\begin{thm}[Lazarsfeld \cite{LazBNP}]
\label{thm-LazBN}
Let $(S,H)$ be a quasi-polarized K3 surface of genus $g$. 
Suppose that $\lvert H \rvert$ is base point free. %, $6 \leq g \leq 12$, and $g \neq 11$. 
\begin{enumerate}
\item If there is a Brill--Noether general curve $C \in \lvert H \rvert$, then $(S,H)$ is Brill--Noether general. 
\item If $S$ is of Picard rank $1$ and $H$ is a primitive element in $\Pic(S)$, then every smooth member $C \in \lvert H \rvert$ is Brill--Noether general. 
%If every member of $\lvert H \rvert$ is irreducible and reduced, then every smooth member $C \in \lvert H \rvert$ is Brill--Noether general. 
\item Let $X$ be a smooth Fano $3$-fold such that $\rho(X)=1$ and $-K_{X}$ is very ample. 
Let $S$ be a Du Val member of $\lvert -K_{X} \rvert$ and $\mu \colon \wt{S} \to S$ the minimal resolution. 
Then $(\wt{S},\mu^{\ast}(-K_{X}|_{S}))$ is Brill--Noether general. 
\end{enumerate}
\end{thm}
\begin{proof}
(1) Suppose there is a divisor $L$ on $S$ such that $h^{0}(L) \cdot h^{0}(H-L) \geq g+1$. 
For a Brill--Noether general member $C \in \lvert H \rvert$, we obtain an exact sequence
\[0 \to \mcO_{S}(L-H) \to \mcO_{S}(L) \to \mcO_{C}(L|_{C}) \to 0.\]
Since $H^{0}(S,\mcO_{S}(L-H))=0$ and $H^{2}(S,\mcO(L))=0$, 
we have $h^{0}(L) \leq h^{0}(L|_{C})$ and $h^{0}(H-L) \leq h^{1}(L|_{C})$. 
Hence $h^{0}(L|_{C}) \cdot h^{1}(L|_{C}) \geq g+1$, which violates the Brill--Noether generality of $C$. 

(2) follows from Lazarsfeld's result \cite[Corollary~1.4]{LazBNP}. 

(3) Pick a very general member $S' \in \lvert -K_{X} \rvert$ such that $\rho(S')=1$ \cite{Moishezon67,RS2} and $C:=S \cap S'$ is smooth. 
By (2), $C$ is Brill--Noether general. 
Hence by (1), the quasi-polarized K3 surface $(\wt{S},\mu^{\ast}(-K_{X}|_{S}))$ is Brill--Noether general. 
\end{proof}

\subsection{Ravindra--Srinivas--Grothendieck--Noether--Lefschetz theorem}

Let $X$ be a smooth Fano $3$-fold of Picard rank $1$. 
As we mentioned in the proof of Theorem~\ref{thm-LazBN}, 
Moishezon \cite{Moishezon67} proved the \emph{Noether--Lefschetz theorem} for a Fano $3$-fold of Picard rank $1$, which states that a very general anticanonical member $S \in \ls{-K_{X}}$ is of Picard rank $1$. 
On the other hand, when taking a very general member $S$ containing a given smooth curve $C \subset X$, it is generally nontrivial whether $S$ is smooth and whether $\rho(S)=2$. 
Ravindra--Srinivas's generalization of the Noether--Lefschetz theorem \cite{RS2} enables us to find sufficient conditions for this desired property for very general members in the linear system $\lvert \mcI_{C}(-K_{X}) \rvert$ (Corollary~\ref{cor-RSNL}). 
We will also need to determine the Picard group of a general member of $\lvert \mcO_{\P(\mcF)}(1) \rvert$ of the projectivization of a rank $2$ weak Fano bundle $\mcF$ such that $c_{1}(\mcF)=c_{1}(X)$. 
To see this, we use a variant of the \emph{Grothendieck--Lefschetz theorem} for weak Fano manifolds (Theorem~\ref{thm-RSGNL}), which can also be shown by using the generalization of the Grothendieck-Lefschetz theorem by Ravindra--Srinivas \cite{RS1}. 

\begin{thm}\label{thm-RSGNL}
Let $M$ be a smooth weak Fano variety of dimension $n \geq 3$ with $\rho(M)=2$. 
Let $H$ be a big base-point-free divisor with $-K_{M} \sim (n-2)H$. 
Let $\psi \colon M \to \ol{M}$ denote the contraction induced by $H$. 
When $n=3$, we additionally assume that $\psi$ coincides with the morphism given by the complete linear system $\lvert H \rvert$. 
Suppose that $\psi$ is small or divisorial with $\dim \psi(\Exc(\psi)) \geq 2$. 
Then for general $X \in \lvert H \rvert$, $\Pic(M) \to \Pic(X)$ is isomorphic. 
\end{thm}
\begin{proof}
Let $\ol{H}$ be the Cartier divisor on $\ol{M}$ such that $\psi^{\ast}\ol{H} = H$. 
Since $H$ is base point free, so is $\ol{H}$. 
If $n \geq 4$ (resp. $n=3$), 
then for (resp. very) general member $\ol{X} \in \lvert \ol{H} \rvert$, 
$\Cl(\ol{M}) \to \Cl(\ol{X})$ is isomorphic by \cite{RS1,RS2}. 
We may also assume that $X:=\psi^{-1}(\ol{X}) \in \lvert H \rvert$ is smooth since $\lvert H \rvert$ is base point free. 
Define $\psi_{X}:=\psi|_{X} \colon X \to \ol{X}$. 

First we suppose that $\psi$ is small. 
Then $\psi_{\ast} \colon \Pic(M)=\Cl(M) \to \Cl(\ol{M})$ is isomorphic. 
In this case, $\psi_{X}$ is also small or isomorphic and hence ${\psi_{X}}_{\ast} \colon \Pic(X)=\Cl(X) \to \Cl(\ol{X})$ is isomorphic. 
Hence the restriction $\Pic(M) \to \Pic(X)$ is an isomorphism. 

Let us consider the case $\psi$ is divisorial. 
Note that $\Pic(M)$ and $\Pic(X)$ are torsion-free since $M$ and $X$ are simply-connected. 
Thus $\Pic(M) \simeq \Z[\Exc(\psi)] \oplus \Cl(\ol{M})$. 
Since $\dim \psi(\Exc(\psi)) \geq 2$, $\psi_{X}$ is not an isomorphism but a divisorial crepant morphism such that 
the exceptional divisor $\Exc(\psi_{X})=\Exc(\psi) \cap X$ and 
the center $\psi_{X}(\Exc(\psi_{X}))=\psi(\Exc(\psi)) \cap \ol{X}$ are irreducible. 
Thus $\Pic(X) \simeq \Z[\Exc(\psi_{X})] \oplus \Cl(\ol{X})$. 
Hence $\Pic(M) \to \Pic(X)$ is isomorphic. 
\end{proof}

\begin{cor}\label{cor-RSNL}
Let $X$ be a Fano $3$-fold and $C=\bigsqcup_{i=1}^{m}C_{i} \subset X$ the disjoint union of smooth curves. 
Suppose the following three conditions. 
\begin{itemize}
\item $\wt{X}:=\Bl_{C}X$ is weak Fano. 
\item The contraction onto the anticanonical model $\wt{X} \to \ol{X}:=\Proj R(\wt{X},-K_{\wt{X}})$ is flopping. 
\item $-K_{\ol{X}}$ is very ample.
\end{itemize}
Then a very general member $S$ of $\lvert \mcI_{C}(-K_{X}) \rvert$ is smooth and the natural map $\Pic(X) \oplus \bigoplus_{i=1}^{m} \Z[C_{i}] \to \Pic(S)$ is isomorphic. 
\end{cor}
\begin{proof}
Let $\s \colon \wt{X}:=\Bl_{C}X \to X$ be the blowing-up. 
By Theorem~\ref{thm-RSGNL}, for a very general member $\wt{S} \in \lvert -K_{\wt{X}} \rvert$, the restriction morphism $\Pic(\wt{X}) \to \Pic(\wt{S})$ is isomorphic. 
It suffices to show that $\s|_{\wt{S}} \colon \wt{S} \to S:=\s(\wt{S})$ is isomorphic. 
As $H^{0}(\wt{X},\mcO(-K_{\wt{X}})) \simeq H^{0}(\ol{X},\mcO(-K_{\ol{X}}))$, we may assume that $\wt{S} \cap \Exc(\s)$ is smooth, which implies that $\s|_{\wt{S}}$ is bijective. 
Thus $\wt{S} \to S$ is the normalization. 
Since $S \in \lvert -K_{X} \rvert$ and $\wt{S} \in \lvert -K_{\wt{X}} \rvert$, it follows that $K_{S} \sim 0$ and $K_{\wt{S}} \sim 0$, which means that $\wt{S} \to S$ is an isomorphism (see e.g. \cite[Proposition~2.3]{Reid94}). 
\end{proof}

\section{Numerical Classification on $\Q^{3}$}\label{sec-Q3num}\label{sec-Q3Num}
Let $\mcE$ be a rank two vector bundle over $\Q^{3}$. 
Let $c_{i}$ be the integer corresponding to $c_{i}(\mcE)$ through the identification $N^{i}(\Q^{3})_{\Z} \simeq \Z$. 
The purpose of this section is to determine the range of possible values for $(c_{1},c_{2})$ when $\mcE$ is a weak Fano bundle. 
For this purpose, we may assume that $\mcE$ is normalized, i.e., $c_{1} \in \{-1,0\}$. 
With this terminology, the Hirzebruch--Riemann--Roch formula for $\mcE$ is given by
\begin{equation}\label{eq-RR}
\chi(\mcE) = \frac{1}{6}(2c_{1}^{3}  - 3 c_{1} c_{2}) + \frac{3}{2}(c_{1}^{2} - c_{2}) + \frac{13}{6}c_{1} + 2 .
\end{equation}
\begin{rem}\label{rem-c10c2even}
If $c_{1} = 0$, then the Hirzebruch--Riemann--Roch formula is simplified to  $\chi(\mcE) = - \frac{3}{2}c_{2} + 2$.
Thus, it follows that the second Chern class $c_{2} \in \Z$ is an even integer in this case.
\end{rem}

Let $\P(\mcE)$ be the projectivization with the projection $\pi \colon \P(\mcE) \to \Q^{3}$, $\xi$ a tautological divisor, and
$H$ the pull-back of a hyperplane section of $\Q^{3}$.
The Grothendieck relation $\sum_{i=0}^{2} (-1)^{i} \pi^{\ast}c_{i}(\mcE)\xi^{2-i}=0$ implies that $\xi^{4}=2c_{1}^{3}-2c_{1}c_{2}$, $\xi^{3}H=2c_{1}^{2}-c_{2}$, $\xi^{2}H^{2}=2c_{1}^{2}-c_{2}$, $\xi H^{3}=2$, and $H^{4}=0$. 
As the anticanonical divisor of $\P(\mcE)$ is given by $-K_{\P(\mcE)} \sim 2\xi + (3-c_{1})H$, it holds that 
\begin{align}
\label{eq-acself}
(-K_{\P(\mcE)})^{4} &= 48(c_{1}^{2}-2c_{2}+9) \text{ and } \\
\label{eq-h0van}
(-K_{\P(\mcE)})^{3}(\xi-aH) 
&= 2 (-2a(c_{1}^{2}-2c_{2}+27)+c_{1}^{3}-2c_{1}c_{2}+9(c_{1}^{2}+3c_{1}-2c_{2}+3)) \text{ for } a \in \Q.
\end{align}
\begin{lem} \label{lem-wFbQ3}
Let $\mcE$ be a rank two weak Fano bundle on $\Q^{3}$. 
\begin{enumerate}
\item If $\mcE$ is normalized, i.e., $c_{1} \in \{-1,0\}$, then $c_{2} \leq 4$.
\item $H^{i}(\mcE(n)) = 0$ for all $i \geq 2$ and $n \geq -1-c_{1}$.
\item $H^{i}(\mcE(n))=0$ for all $i \geq 1$ and $n \geq \frac{1}{2}(3-c_{1})$. 
\item For an integer $a$, $H^{0}(\mcE(-a))=0$ if $a > \frac{c_{1}^{3}-2c_{1}c_{2}+9(c_{1}^{2}+3c_{1}+3-2c_{2})}{2(c_{1}^{2}-2c_{2}+27)}$. 
In particular, $H^{0}(\mcE(-a))=0$ if $a > 
\begin{cases}
\frac{9}{2}-\frac{108}{27-2c_{2}} & \text{when $c_{1} = 0$,} \\
4-\frac{54}{14-c_{2}}& \text{when $c_{1} = -1$.}
\end{cases}
$
\end{enumerate}
\end{lem}

\begin{proof}
(1) Since $-K_{\P(\mcE)}$ is nef and big, it follows that $(-K_{\P(\mcE)})^{4} > 0$. 
Hence $c_{1}^{2}-2c_{2}+9 > 0$ follows from (\ref{eq-acself}). 
If $\mcE$ is normalized, then the inequality $c_{1}^{2}-2c_{2}+9 > 0$ is equivalent to $c_{2} \leq 4$, which proves (1)

(2) Since $-K_{\P(\mcE)} \sim 2\xi + (3-c_{1})H$ is nef, 
$\xi+(2-c_{1})H$ is ample. 
Hence $\mcE(n)$ is an ample vector bundle for all $n \geq 2-c_{1}$. 
Then the Le Potier vanishing theorem \cite[Theorem~7.3.5]{laz2} proves the statement.

(3) If $n \geq \frac{1}{2}(3-c_{1})$, then $3\xi + (3+n-c_{1})H$ is nef and big since so is $-K_{\P(\mcE)} \sim 2\xi + (3-c_{1})H$. 
Hence the Kawamata--Viehweg vanishing theorem shows
$H^{i}(\Q^{3},\mcE(n))=H^{i}(\P(\mcE),\xi+nH) = H^{i}(\P(\mcE),K_{\P(\mcE)}+(3\xi + (3+n-c_{1})H)) = 0$. 

(4) Since $-K_{\P(\mcE)}$ is nef and big, it follows that $(-K_{\P(\mcE)})^{4} > 0$. 
Hence $c_{1}^{2}-2c_{2}+9 > 0$ follows from (\ref{eq-acself}). 
In particular, $c_{1}^{2}-2c_{2}+27 > 0$ holds. 
Hence for a given integer $a$ satisfying the inequality in (4), 
the equality (\ref{eq-h0van}) shows $(-K_{\P(\mcE)})^{3}(\xi-aH) < 0$. 
Since $-K_{\P(\mcE)}$ is nef and big, this inequality yields that $0=H^{0}(\P(\mcE),\mcO(\xi-aH))=H^{0}(\mcE(-a))$. 
The remaining part follows from direct computation.
\end{proof}
Lemma~\ref{lem-wFbQ3} gives the upper bound for the second Chern class $c_{2}(\mcE)$ of rank $2$ weak Fano bundles $\mcE$ from the inequality $(-K_{\P(\mcE)})^{4} > 0$. 
The lower bound for $c_{2}(\mcE)$ is closely related to the slope stability of $\mcE$ as shown by the following proposition. 
Note that if $\mcE$ is a normalized rank two vector bundle, then $\mcE$ is slope stable if and only if $h^{0}(\mcE) = 0$ (c.f. \cite[Remark 2.2]{FHI2}).

\begin{prop}\label{prop-unstableQ3}
Let $\mcE$ be a normalized rank two weak Fano bundle on $\Q^{3}$.
Then the following conditions are equivalent.
\begin{enumerate}[label=(\alph*)]
\item $\mcE$ is not slope stable.
\item $\mcE$ is isomorphic to one of the following: $\mcO(-1) \oplus \mcO(1)$, $\mcO(-2) \oplus \mcO(1)$, $\mcO^{\oplus 2}$, or $\mcO(-1) \oplus \mcO$.
\item $c_{2} \leq 0$.
\end{enumerate}
\end{prop}

\begin{proof}
Let us show the implication (a) $\ra$ (b). 
Note that Lemma~\ref{lem-wFbQ3}~(2) and Serre duality show $h^{0}(\mcE(-2)) = h^{3}(\mcE^{\vee}(2) \otimes \omega_{\Q^{3}}) = h^{3}(\mcE(-c_{1}-1)) = 0$. 

If $h^{0}(\mcE(-1)) \neq 0$, then it follows from Lemma~\ref{lem-wFbQ3}~(4) that 
$\frac{9}{2}-\frac{108}{27-2c_{2}} \geq 1$ when $c_{1}=0$ % c_{2} \leq -2
and $4-\frac{54}{14-c_{2}} \geq 1$ when $c_{1} = -1$. % c_{2} \leq -4.
Hence $c_{2} \leq -2+2c_{1}$, which implies $c_{2}(\mcE(-1)) = c_{2}-2c_{1}+2 \leq 0$. 
%In both cases $c_{2}(\mcE) \leq - 2$, which also implies $c_{2}(\mcE(-1)) = c_{2}(\mcE)+c_{1}(\mcE)(-H)+H^{2}= \leq 0$.
Pick a non-zero section $s \in H^{0}(\mcE(-1))$ and let $Z$ be the zero locus of $s$. 
Since $H^{0}(\mcE(-2)) = 0$, $Z$ is purely two codimensional or empty. 
Since $[Z]=c_{2}(\mcE(-1)) \leq 0$ in $N^{2}(\Q^{3})_{\Z} \simeq \Z$, it follows that $Z$ is empty. 
Hence the non-zero section $s \colon \mcO \to \mcE(-1)$ is nowhere vanishing and hence its cokernel is an invertible sheaf $\mcO_{\Q^{3}}(c_{1}-2)$. 
Since $H^{1}(\mcO_{\Q^{3}}(n)) = 0$ for every $n$, the exact sequence $0 \to \mcO \to \mcE(-1) \to \mcO_{\Q^{3}}(c_{1}-2) \to 0$ splits. 
Hence $\mcE=\mcO(1) \oplus \mcO(c_{1}-1)$.
Note that this $\mcE$ is a Fano bundle if $c_{1}=0$ and not a Fano bundle but a weak Fano bundle if $c_{1}=-1$. 

Assume $H^{0}(\mcE(-1)) = 0$.
Since $\mcE$ is not slope stable, it holds that $h^{0}(\mcE) \neq 0$.
Again by Lemma~\ref{lem-wFbQ3}~(4), we have 
$\frac{9}{2}-\frac{108}{27-2c_{2}} \geq 0$ when $c_{1}=0$ % c_{2} \leq 1
and $4-\frac{54}{14-c_{2}} \geq 0$ when $c_{1} = -1$. % c_{2} \leq 0.
Recall that, if $c_{1} = 0$, then $c_{2}$ is even by Remark~\ref{rem-c10c2even}.
Thus, in both cases, it follows that $c_{2} \leq 0$.
Then a similar argument to the above shows that 
$\mcE \simeq \mcO(c_{1}) \oplus \mcO$. 
Whether $c_{1}=0$ or $-1$, $\mcE$ is a Fano bundle in this case. 
Thus (a) implies (b).

It is clear that (b) implies (c). Now it remains to show that (c) implies (a).
By Lemma~\ref{lem-wFbQ3}~(2) and (\ref{eq-RR}), 
\[ h^{0}(\mcE) \geq h^{0}(\mcE) - h^1(\mcE) = \chi(\mcE) =
\begin{cases}
- \frac{3}{2}c_{2} + 2 & \text{if $c_{1} = 0$, and} \\
-c_{2} + 1 & \text{if $c_{1} = -1$.}
\end{cases} \]
In both cases, $c_{2} \leq 0$ implies $h^{0}(\mcE) > 0$.
\end{proof}

\begin{rem}
In contrast to del Pezzo threefolds $X$ of degree $\deg X \in \{3,4,5\}$ \cite{Ishikawa16, FHI1, FHI2},
there is no indecomposable weak Fano bundle $\mcE$ over $\Q^{3}$ that is not slope stable.
\end{rem}

We now have a numerical classification of rank $2$ weak Fano bundles on $\Q^{3}$, as stated in the following proposition. 

\begin{prop}\label{prop-nonexistQ3}
Let $\mcE$ be an indecomposfable normalized rank two weak Fano bundle on $\Q^{3}$.
Then $\mcE$ is slope stable and 
\begin{align}\label{eq-candidates}
(c_{1},c_{2}) \in \{(0,2),(-1,1),(-1,2),(-1,3),(-1,4)\}.
\end{align}
In particular, the case $(c_{1},c_{2})=(0,4)$ cannot occur.
\end{prop}
\begin{proof}
Suppose that $\mcE$ is not a direct sum of line bundles. 
By Proposition~\ref{prop-unstableQ3}, 
$\mcE$ is slope stable and $c_{2} > 0$. 
By Remark~\ref{rem-c10c2even} and Lemma~\ref{lem-wFbQ3}, 
we have $c_{2} \leq 4$ and $c_{2}$ is even if $c_{1}=0$. 
Hence it suffices to show $(c_{1},c_{2}) \neq (0,4)$. 

Suppose that there is a rank $2$ weak Fano bundle $\mcE$ with $c_{1}=0$ and $c_{2}=4$. 
By Proposition~\ref{prop-unstableQ3}, $\mcE$ is slope stable. 
For such a slope stable bundle $\mcE$, Sols--Szurek--Wi\'sniewski proved in \cite[Proof~of~Proposition~2]{ssw} that there is a smooth conic $C$ such that $\mcE|_{C} \simeq \mcO_{\P^{1}}(-d) \oplus \mcO_{\P^{1}}(d)$ for $d \geq 3$. 
In particular, there is a conic $C$ such that $\mcE(1)|_{C}$ is not nef. 
However, since $(\Q^{3},\mcO_{\Q^{3}}(1))$ satisfies $(\dag)$ by \cite[Lemma~4.3, Remark~4.4]{BH18} and $(K_{\Q^{3}}^{2}+\Delta(\mcE(1))).\mcO_{\Q^{3}}(1) =22 > 0$, Proposition~\ref{prop-conicnef} shows that $\mcE(1)|_{C}$ is nef for every conic $C$. 
This is a contradiction and hence $\mcE$ is not a weak Fano bundle. 
We complete the proof. 
\end{proof}

\begin{rem}
In \cite{ssw}, Sols--Szurek--Wi\'sniewski showed that every rank $2$ stable bundle $\mcE$ with $(c_{1},c_{2})=(0,4)$ is never a Fano bundle. 
Proposition~\ref{prop-nonexistQ3} shows that every such an $\mcE$ also is not a weak Fano bundle. 
\end{rem}
For the rest of this paper,
we study a rank two normalized weak Fano bundle $\mcE$ on $\Q^{3}$ that is slope stable. 
Of the possible values of $(c_{1},c_{2})$ given in \ref{eq-candidates}, when $c_{2} \leq 2$, geometric details of $\mcE$ have been obtained as introduced in the following theorem, and in particular, the existence of such an $\mcE$ is also known.
\begin{thm}[{\cite{ssw, SW90, Ott90, OS94}}] \label{thm-FanoClsf}
Let $\mcE$ be a normalized rank two slope stable bundle on $\Q^{3}$.
Then the following holds.
\begin{enumerate}
\item If $(c_{1}, c_{2}) = (-1,1)$, then $\mcE$ is isomorphic to the (negative) Spinor bundle. This is a Fano bundle. 
\item If $(c_{1}, c_{2}) = (0,2)$, then $\mcE$ is isomorphic to the pull-back of the null-correlation bundle on $\P^3$ via a double covering $\Q^{3} \to \P^{3}$. This is a Fano bundle. 
\item If $(c_{1},c_{2})=(-1,2)$, then $\mcE$ is isomorphic to the restriction of a Cayley bundle on $\Q^{5}$ via a linear embedding $\Q^{3} \hra \Q^{5}$. This is not a Fano bundle but a weak Fano bundle. 
\end{enumerate}
\end{thm}
\begin{proof}
For the proofs of (1) and (2), we refer to \cite[Proposition~3.2 and Corollary]{SW90}. 
(3) is originally stated in \cite[Remark~4.7]{OS94} and we add a proof here as a supplement. 

Let $\mcE$ be a rank two slope stable bundle on $\Q^{3}$ with $c_{1}(\mcE) = -1$ and $c_{2}(\mcE) = 2$.
Then $\RG(\mcE(1)) \simeq \C^{\oplus 2}$ by \cite[Proposition~4.2]{OS94}. 
The slope stability of $\mcE$ shows that the evaluation map $\ev \colon \mcO^{\oplus 2} \to \mcE(1)$ is injective.
Since $c_{1}(\mcE(1)) = 1$, the support of $\Cok(\ev)$ is a hyperplane section $H \subset \Q^{3}$.
Then it follows from \cite[Proposition~4.4]{OS94} that $H \simeq (\P^{1})^{2}$ is a smooth hyperplane section on $\Q^{3}$ and $\Cok(\ev) \simeq \mcO(1,-1)$.
Moreover, the zero set of a non-zero section of $H^{0}(\mcE(1)) \simeq \C^{\oplus 2}$ is a member of $\lvert \mcO_{H}(2,0) \rvert$ and two linearly independent sections $s,t \in H^{0}(\mcE(1))$ cut a system $g^{1}_{2}$ without base points. 
In particular, for a general non-zero section $s \in H^{0}(\mcE(1))$, the zero locus $Z$ of $s$ is a disjoint union of two lines $l_{1} \sqcup l_{2}$: 
\[0 \to \mcO \to \mcE(1) \to \mcI_{l_{1} \sqcup l_{2}}(1) \to 0.\]
Fix an embedding $\Q^{3} \hra \Q^{4}$. 
Let $P_{1}, P_{2} \subset \Q^{4}$ be disjoint $2$-planes such that $(\Q^{3}, l_{1}, l_{2})$ is the linear section of $(\Q^{4}, P_{1}, P_{2})$. 
Then we obtain a vector bundle $\mcC_{\Q^{4}}$ by applying the Hartshorne--Serre construction along $P_{1} \sqcup P_{2}$ as
$0 \to \mcO_{\Q^{4}}(-1) \to \mcC_{\Q^{4}} \to \mcI_{P_{1} \sqcup P_{2}/\Q^{4}} \to 0$. 
By construction, $\mcC_{\Q^{4}}|_{\Q^{3}} \simeq \mcE$. 
Then \cite[Remark~3.4]{Ott90} shows (3). 
\end{proof}

\section{Resolutions of rank $2$ weak Fano bundles on $\Q^{3}$}\label{sec-Q3Resol}

From the above discussion, every indecomposable weak Fano bundle is slope stable, and furthermore, its isomorphism class is determined when $c_{2} \leq 2$.
The remaining cases are when $c_{2} \geq 3$, and more specifically, when $(c_{1},c_{2}) \in \{(-1,3),(-1,4)\}$ by Proposition~\ref{prop-nonexistQ3}. 
To obtain geometric expressions for weak Fano bundles in these cases, we will provide resolutions of them. 
In this section, we adopt the terminology of \cite[Section~2]{FHI2}.

The following theorem by Kapranov provides an example of a full exceptional collection.

\begin{thm}[\cite{Kap88}]
The derived category $\Db(\Q^{3})$ of the quadric threefold $\Q^{3}$ admits a full exceptional collection
\[ \Db(\Q^{3}) = \langle \mcS(-2), \mcO(-2), \mcO(-1), \mcO  \rangle. \]
\end{thm}

Mutations applied to Kapranov's collection produce another exceptional collection; for example
\[ \Db(\Q^{3}) = \langle \mcO(-2), \mcO(-1), \mcS, \mcO  \rangle. \]

\subsection{When $(c_{1}, c_{2}) = (-1,2)$}

\begin{thm} \label{thm-Q3CaylayResol}
Let $\mcE$ be a rank two vector bundle on $\Q^{3}$ with $(c_{1},c_{2}) = (-1, 2)$.
Then the following conditions are equivalent.
\begin{enumerate}[label=(\alph*)]
\item $\mcE$ is slope stable.
\item $\mcE$ is a weak Fano bundle. 
\item There exists an exact sequence
\begin{equation} \label{ex-Q3CayleyResol}
0 \to \mcO(-2) \to \mcO(-1)^{\oplus 5} \to \mcO^{\oplus 2} \oplus \mcS^{\oplus 2} \to \mcE(1) \to 0.
\end{equation}
\end{enumerate}
\end{thm}
\begin{proof}
The implication (b) $\ra$ (a) follows from Proposition~\ref{prop-unstableQ3}. 
The implication (c) $\ra$ (b) also holds since if $\mcE$ fits into an exact sequence (\ref{ex-Q3CayleyResol}),
then $\mcE(2)$ is globally generated with $c_{1}(\mcE(2))=3$ and $s_{3}(\mcE(2))=18$. 
Hence it suffices to prove the implication (a) $\ra$ (c). 

Assume that $\mcE$ is slope stable.
Then $\RG(\mcE(1)) \simeq \C^{\oplus 2}$ by \cite[Proposition~4.2]{OS94}. 
Thus $\LL_{\mcO}(\mcE(1)) \simeq \mcL_{H}$, where $\mcL_{H}$ is the cokernel of the injective evaluation map $\ev \colon \mcO^{\oplus 2} \to \mcE(1)$ that appeared in
the proof of Theorem~\ref{thm-FanoClsf}~(3). 
This sheaf $\mcL_{H}$ is a torsion sheaf with $c_{1}(\mcL_{H}) = 1$.

Put $\mcV := \RR_{\mcO(-2)}(\mcL_{H})$.
Since $\RHom(\mcL_{H}, \mcO(-2)) \simeq \RHom(\mcE(1), \mcO(-2)) \simeq \RG(\mcE(-2)) \simeq \C[-2]$ by \cite[Proposition~4.2]{OS94},
the complex $\mcV$ fits in an exact triangle
\[ \mcV \to \mcL_{H} \to \mcO(-2)[2] \xrightarrow{+1} , \]
and hence
\begin{align*}
\mcH_{\mathrm{coh}}^{i}(\mcV) \simeq \begin{cases}
\mcO(-2) & \text{if $i = -1$}, \\
\mcL_{H} & \text{if $i = 0$, and} \\
 0 & \text{if $i \neq 0, -1$.}
\end{cases}
\end{align*}
Note that $\mcV \in \langle \mcO \rangle^{\bot} \cap {}^{\bot}\langle \mcO(-2) \rangle = \langle \mcO(-1), \mcS \rangle$.
Thus applying \cite[Lemma 2.7]{FHI2} gives an exact triangle
\[ \RHom(\mcS, \mcV) \otimes \mcS \to \mcV \to \RHom(\mcO(-1), \mcV)^{\vee} \otimes \mcO(-1) \xrightarrow{+1}. \]
As a part of the cohomology long exact sequence from the above triangle, there is an exact sequence
\[ \cdots \to \mcS^{\oplus a} \xrightarrow{\phi} \mcO(-2) \to \mcO(-1)^{\oplus x} \to \mcS^{\oplus y} \to \mcL_{H} \xrightarrow{\psi} \mcO(-1)^{\oplus b} \to \cdots, \]
for some $a,b,x,y \in \Z_{\geq 0}$.
The vanishing $\Hom(\mcS, \mcO(-2)) = 0$ implies that the morphism $\phi$ is zero, 
and since $\mcL_{H}$ is torsion, $\psi$ is also zero.
Thus there is an exact sequence
\[ 0 \to \mcO(-2) \to \mcO(-1)^{\oplus x} \to \mcS^{\oplus y} \xrightarrow{\alpha} \mcL_{H} \to 0. \]
Comparing the rank and $c_{1}$ of these sheaves gives $1 - x + 2y = 0$ and $-2 + x - y - 1 = 0$, which implies $x = 5$ and $y = 2$.
Put $\mcK := \Ker(\alpha)$.
Then pulling back the exact sequence $0 \to \mcO^{\oplus 2} \to \mcE(1) \to \mcL_{H} \to 0$ by $\alpha$ yields a commutative diagram of exact sequences
\[ \begin{tikzcd}
 & & 0 \arrow[d] & 0 \arrow[d]  & \\
 & & \mcK \arrow[d]  \arrow[r ,equal] & \mcK \arrow[d]  & \\
0 \arrow[r] & \mcO^{\oplus 2} \arrow[d, equal]  \arrow[r] & \mcF \arrow[d]  \arrow[r] & \mcS^{\oplus 2} \arrow[d, "\alpha"] \arrow[r]  & 0 \\
0 \arrow[r] & \mcO^{\oplus 2} \arrow[r] & \mcE(1) \arrow[d] \arrow[r] & \mcL_{H}  \arrow[d] \arrow[r] & 0 \\
&& 0 & 0 &
\end{tikzcd} \]
Since $\Ext^{1}(\mcS, \mcO) = 0$, $\mcF \simeq \mcO^{\oplus 2} \oplus \mcS^{\oplus 2}$.
Connecting two exact sequences $0 \to \mcK \to \mcO^{\oplus 2} \oplus \mcS^{\oplus 2} \to \mcE(1) \to 0$ and $0 \to \mcO(-2) \to \mcO(-1)^{\oplus 5} \to \mcK \to 0$ gives the required exact sequence.
\end{proof}

\subsection{When $(c_{1}, c_{2}) = (-1,3)$}

\begin{lem} \label{lem-orth3}
Let $\mcE$ be a rank two vector bundle on $\Q^{3}$ with $(c_{1}(\mcE), c_{2}(\mcE)) = (-1, 3)$.
Then the following holds.
\begin{enumerate}[label=(\arabic*)]
\item If $\mcE$ is slope stable, then $\RG(\mcE(-1)) = 0$ holds.
\item The following conditions are equivalent.
\begin{enumerate}[label=(\alph*)]
\item $\mcE$ is weak Fano.
\item $h^{0}(\mcE(1)) = 0$.
\item $\mcE$ is $2$-regular.
\end{enumerate}
If $\mcE$ satisfies one of (and hence all of) the above equivalent conditions, $\RG(\mcE(1)) = 0$ holds.
\end{enumerate}
\end{lem}

\begin{proof}
(1) Note that $h^{i}(\mcE(-1)) = h^{3-i}(\mcE(-1))$ by Serre duality.
Then $h^{0}(\mcE(-1))=0$ holds since $\mcE$ is slope stable, and $h^{1}(\mcE(-1))=0$ by \cite[Corollary 2.4]{ES84}. This proves (1).

(2) First we show the implication (a) $\ra$ (b).
Suppose that $\mcE$ is weak Fano. 
Let $\mcF:=\mcE(2)$, $\pi \colon Y:=\P_{\Q^{3}}(\mcF) \to \Q^{3}$ the projectivization, $\xi_{\mcF}$ the tautological divisor, and $H$ the pull-back of a hyperplane section of $\Q^{3}$. 
Note that $-K_{Y} \sim 2\xi$ and hence $\xi$ is nef and big. 
Then $\mcF$ is globally generated by \cite[Theorem~1.7]{FHI1}. 
In particular, $\xi$ is base point free. 
Let $\wt{X} \in \ls{\xi}$ be a general member and $s \in H^{0}(\mcF)$ a corresponding section. 
Then $\wt{X}$ is the blowing-up of $\Q^{3}$ along the zero-locus of $s$, denote by $C$. 
Since $c_{1}(\mcF)=c_{1}(\Q^{3})$ and $c_{2}(\mcF)=7$, each connected component of $C$ is a smooth elliptic curve and $\deg C=7$. 
If $C$ is not connected, then there is a connected component  $C'$ of $C$ such that $\deg C' = 3$, which contradicts that $\Q^{3}$ does not contains a plane cubic. 
Hence $C$ is connected and hence an elliptic curve of degree $7$. 

Since $-K_{\wt{X}} \sim \xi|_{\wt{X}}$ is nef and big, 
$\wt{X}=\Bl_{C}\Q^{3}$ is a weak Fano $3$-fold. 
Let $\psi_{\wt{X}} \colon \wt{X} \to \ol{X}$ be the crepant contraction given by $-K_{\wt{X}}$. 
By \cite{JPR05}, $\psi_{\wt{X}}$ is not divisorial. 
Hence $\psi_{\wt{X}}$ is a flopping contraction. 
By \cite[Proposition~2.7]{JPR11}, $-K_{\ol{X}}$ is very ample since $-K_{\wt{X}}^{3}=\xi^{4}=s_{3}(\mcF)=12$. 
Therefore, it follows from Corollary~\ref{cor-RSNL} that a very general member $S \in \ls{\mcI_{C/\Q^{3}}(3)}$ is smooth and satisfies $\Pic(S) \simeq \Z[H_{S}] \oplus \Z[C]$, where $H_{S}$ is the restriction of a hyperplane section on $\Q^{3}$. 
We also regard $S$ as an anticanonical member of $\wt{X}=\Bl_{C}\Q^{3}$. 
Let $\xi_{S}:=\xi|_{S}$. Then $\xi_{S} \sim 3H_{S}-C$. 

Suppose $H^{0}(\mcE(1)) \neq 0$. 
This is equivalent to saying that $\ls{\xi-H}$ is non-empty. 
Since the morphisms $H^{0}(Y,\mcO(\xi-H)) \to H^{0}(\wt{X},\mcO((\xi-H)|_{\wt{X}})) \to H^{0}(S,\mcO(\xi_{S}-H_{S}))$ are isomorphisms, $D:=\xi_{S}-H_{S} \sim 2H_{S}-C$ is an effective divisor on $S$. 
Since $D^{2}=-4$, $D$ must contain a $(-2)$-curve $\Gamma$. 
On the other hand, since $\Pic(S)$ is generated by $H_{S}$ and $C$, we can pick $a,b \in \Z$ such that $\Gamma \sim aH_{S}+bC$. 
Then $-2=\Gamma^{2} = 6a^{2}+14ab$. 
Thus $a=\frac{-7b \pm \sqrt{(7b)^{2}-12}}{6}$, which cannot be an integer. 
This is a contradiction and hence $h^{0}(\mcE(1)) = 0$. 

(b) $\ra$ (c): Assume that $h^{0}(\mcE(1)) = 0$.
Then $\mcE$ is stable, and thus applying \cite[Corollary 2.4]{ES84} together with Serre duality shows that $h^{i}(\mcE(1)) = 0$ for $i = 2,3$ and $h^{2}(\mcE) = 0$.
Now the Hirzebruch--Riemann--Roch formula gives that $h^{1}(\mcE(1)) = - \chi(\mcE(1)) = 0$.
Therefore $\mcE$ is $2$-regular.

(c) $\ra$ (a): If $\mcE$ is $2$-regular, then $\mcE(2)$ is globally generated with $c_{1}(\mcE(2))=3$ and $s_{3}(\mcE(2))=12$. 
Hence $\mcE$ is a weak Fano bundle. 
\end{proof}

\begin{thm}\label{thm-Q3Resol7}
Let $\mcE$ be a rank two vector bundle with $(c_{1}(\mcE), c_{2}(\mcE)) = (-1, 3)$.
Then the following conditions are equivalent.
\begin{enumerate}[label=(\alph*)]
\item $\mcE$ is weak Fano.
\item $\mcE$ lies in the exact sequence
\[ 0 \to \mcO_{X}(-2)^{\oplus 2} \to \mcO_{X}(-1)^{\oplus 10} \to \mcS^{\oplus 5} \to \mcE(1) \to 0. \]
\end{enumerate}
\end{thm}

\begin{proof}
Assume that $\mcE$ is a weak Fano bundle. 
By Proposition~\ref{prop-unstableQ3}, $\mcE$ is slope stable.
Then Lemma~\ref{lem-orth3} implies that 
$\RHom(\mcE,\mcO_{\Q^{3}}) = \RG(\mcE^{\vee}) = \RG(\mcE(1)) = 0$
and
$\RHom(\mcO_{\Q^{3}}(1),\mcE)=\RG(\mcE(-1)) = 0$. 
Consider the exceptional collection
\[ \Db(\Q^{3}) = \langle \mcO_{\Q^{3}}, \mcS(1), \mcO_{\Q^{3}}(1), \mcO_{\Q^{3}}(2) \rangle. \]
The vector bundle $\Omega_{\P^{4}}(2)|_{\Q^{3}}$ fits into
\begin{align}\label{ex-EulerP4}
0 \to \Omega_{\P^{4}}(2)|_{\Q^{3}} \to \mcO_{\Q^{3}}(1)^{\oplus 5} \to \mcO_{\Q^{3}}(2) \to 0
\end{align}
and hence is also an exceptional vector bundle on $\Q^{3}$ giving a new exceptional collection
%The vector bundle $\Omega_{\P^{4}}(2)|_{\Q^{3}}$ given as the kernel of the natural surjection $\mcO_{\Q^{3}}(1)^{\oplus 5} \twoheadrightarrow \mcO_{\Q^{3}}(2)$ is also an exceptional vector bundle on $X$, which gives a new exceptional collection
\[ \Db(\Q^{3}) = \langle \mcO_{\Q^{3}}, \mcS(1), \Omega_{\P^{4}}(2)|_{\Q^{3}}, \mcO_{\Q^{3}}(1) \rangle. \]
This exceptional collection implies that $\mcE \in \langle \mcS(1), \Omega_{\P^{4}}(2)|_{\Q^{3}} \rangle$, and hence 
\cite[Lemma 2.7]{FHI2} gives an exact sequence
\begin{align} \label{ex-seq1c27}
0 \to \mcS(1)^{\oplus c} \to \Omega_{\P^{4}}(2)|_{\Q^{3}}^{\oplus d} \to \mcE \to \mcS(1)^{\oplus a} \to \Omega_{\P^{4}}(2)|_{\Q^{3}}^{\oplus b} \to 0 
\end{align} 
for some $a, b, c, d \in \Z_{\geq 0}$.
Applying the functor $\RHom(-, \mcE)$ to the exact sequence (\ref{ex-EulerP4}) gives
\[ 0=\Hom(\mcO_{\Q^{3}}(1)^{\oplus 5},\mcE) 
\to \Hom(\Omega_{\P^{4}}(2)|_{\Q^{3}},\mcE) 
\to \Ext^{1}(\mcO_{\Q^{3}}(2), \mcE) 
\to \Ext^{1}(\mcO_{\Q^{3}}(1)^{\oplus 5},\mcE)=0,
\]
where the first term and the last terms vanish since $\RG(\mcE(-1))=0$. 
$\Ext^{1}(\mcO_{\Q^{3}}(2), \mcE) = H^{1}(\mcE(-2)) = H^{2}(\mcE)^{\vee} = 0$ also follows from Serre duality and Lemma~\ref{lem-wFbQ3}~(2). 
Thus, the homomorphism $\Omega_{\P^{4}}(2)|_{\Q^{3}}^{\oplus d} \to \mcE$ in (\ref{ex-seq1c27}) is zero, which yields an exact sequence
\[ 0 \to \mcE \to \mcS(1)^{\oplus a} \to \Omega_{\P^{4}}(2)|_{\Q^{3}}^{\oplus b} \to 0. \]
Comparing the rank and $c_{1}$ in the exact sequence above gives equations $2a = 4b+1$ and $a = 3b -1$ respectively.
Thus $a =5$ and $b = 2$. 
Combining this with the exact sequence (\ref{ex-EulerP4}) gives the exact sequence
\[ 0 \to \mcE \to \mcS(1)^{\oplus 5} \to \mcO_{\Q^{3}}(1)^{\oplus 10} \to \mcO_{\Q^{3}}(2)^{\oplus 2} \to 0, \]
and dualizing this yields the desired exact sequence.

Conversely, the exact sequence in (b) implies that $\mcE(2)$ is globally generated since $\mcS(1)$ is also globally generated. 
As $c_{1}(\mcE(2))=3$ and $s_{3}(\mcE(2))=12$, $\mcE$ is weak Fano. 
\end{proof}

\begin{rem}
Let $\mcE$ be a rank two weak Fano bundle with $(c_{1}(\mcE), c_{2}(\mcE)) = (-1, 3)$.
In the proof of \cite[Theorem 5.2]{OS94}, it is shown that $\mcE$ fits an exact sequence
\[ 0 \to \mcS^{\oplus 3} \to \left(\Omega_{\P^4}(1)|_{\Q^{3}} \right)^{\oplus 2} \to \mcE(1) \to 0, \]
and this exact sequence was used to analyze the structure of the moduli space.
\end{rem}

\subsection{When $(c_{1}, c_{2}) = (-1,4)$}

\begin{lem} \label{lem-c2=8}
Let $\mcE$ be a weak Fano bundle on $\Q^{3}$ with $c_{1}(\mcE) = -1$ and $c_{2}(\mcE) = 4$.
Then the following holds.
\begin{enumerate}
\item $H^{0}(\mcE(a)) = 0$ for all $a \leq 1$.
\item $\Hom(\mcS(1), \mcE(2)) = 0$.
\item $\RHom(\mcS(1), \mcE(2)) = 0$.
\end{enumerate}
\end{lem}

\begin{proof}
(1) This statement directly follows from Lemma~\ref{lem-wFbQ3}~(4).

(2) Put $\mcF = \mcE(2)$, and assume that there exists a non-zero morphism $\alpha \colon \mcS(1) \to \mcF$ towards a contradiction. 
Since $H^{0}(\mcF(-1)) = 0$ by (1), the morphism $\alpha$ is injective.
On the other hand, $\Hom(\mcS(a+1), \mcF) = 0$ for all $a > 0$.
Indeed, the surjection $\mcO_{\Q^{3}}^{\oplus 4} \to \mcS(1)$ gives an injection $\Hom(\mcS(a+1), \mcF) \subset H^{0}(\mcF(-a))^{\oplus 4}$, and the right hand side is zero for all $a > 0$ by (1).
Applying \cite[Lemma~4.1]{FHI2} to $\mcS(1) = \mcS^{\vee} \subset \mcF$, we have
\begin{align*}
0 &\leq \left(c_{1}(\mcF)^{3} - 2c_{1}(\mcF)c_{2}(\mcF) \right) - \left(c_{1}(\mcF)^{2}- c_{2}(\mcF)\right) c_{1}(\mcS^{\vee}) + c_{1}(\mcF)c_{2}(\mcS^{\vee}) \\
&= \left( (3H)^{3} - 2 \cdot 3H \cdot 8l \right) 
- \left( (3H)^{2} - 8l \right) \cdot H
+ (3H) \cdot l \\
&= \left( 54 - 48 \right) 
- \left( 18 - 8 \right)
+ 3
= -1
\end{align*}
which leads to a contradiction.

(3) The injective morphism $\mcS \subset \mcO_{\Q^{3}}^{\oplus 4}$ gives a surjection 
$H^{3}(\mcE(2-a))^{\oplus 4} \to \Ext^{3}(\mcS(a), \mcE(2))$ for all $a \in \Z$.
Thus Lemma~\ref{lem-wFbQ3}~(2) implies that $\Ext^{3}(\mcS(a), \mcE(2)) = 0$ for all $a \leq 2$.
Next, the exact sequence
\[ 0 \to \mcS(1) \to \mcO_{\Q^{3}}(1)^{\oplus 4} \to \mcS(2) \to 0 \]
together with Lemma~\ref{lem-wFbQ3}~(2) shows that $\Ext^{2}(\mcS(1), \mcE(2)) \simeq \Ext^{3}(\mcS(2), \mcE(2)) = 0$.
The Hirzebruch--Riemann--Roch formula gives $\chi(\mcS(1), \mcE(2)) = 0$, which shows $\Ext^{1}(\mcS(1), \mcE(2)) = 0$. 
Hence the statement holds. 
\end{proof}

\begin{thm} \label{thm-Q3Resol8}
Let $\mcE$ be a rank two bundle with $(c_{1}(\mcE), c_{2}(\mcE)) = (-1, 4)$.
Then the following conditions are equivalent.
\begin{enumerate}[label=(\alph*)]
\item $\mcE$ is weak Fano.
\item $\mcE$ lies in the exact sequence
\[ 0 \to \mcO_{\Q^{3}}(-2)^{\oplus 2} \to \mcO_{\Q^{3}}(-1)^{\oplus 7} \to \mcO_{\Q^{3}}^{\oplus 7} \to \mcE(2) \to 0. \]
\end{enumerate}
\end{thm}

\begin{proof}
Assume that $\mcE$ is weak Fano.
Then $\mcE(2)$ is globally generated by \cite[Theorem~1.7]{FHI1}, and $H^{0}(\mcE(2)) \simeq \C^{\oplus 7}$ by the Hirzebruch--Riemann--Roch formula.
Thus there exists a surjective morphism $\mcO_{\Q^{3}}^{\oplus 7} \to \mcE(2)$.
Let $\mcK$ be the kernel of the morphism.
Since $\RHom(\mcK, \mcO(-2)) \simeq \RHom(\mcE(2), \mcO_{\Q^{3}}(-2))[1] \simeq \RG(\mcE(-3))[1] \simeq \C^{\oplus 2}[-1]$, 
there exists a canonical non-trivial extension
\[ 0 \to \mcO_{\Q^{3}}(-2)^{\oplus 2} \to \mcV \to \mcK \to 0.\]
This locally free sheaf $\mcV$ has rank $7$ and admits an exact sequence
\[ 0 \to \mcO_{\Q^{3}}(-2)^{\oplus 2} \to \mcV \to \mcO_{\Q^{3}}^{\oplus 7} \to \mcE(2) \to 0. \]
Consider the semiorthogonal decomposition $\Db(\Q^{3})  = \langle \mcO_{\Q^{3}}(-2), \mcO_{\Q^{3}}(-1), \mcO_{\Q^{3}},\mcS(1) \rangle$.
Then this construction implies $\mcV \simeq \RR_{\mcO_{\Q^{3}}(-2)}\LL_{\mcO_{\Q^{3}}}(\mcE(2))[-1]$.
Since $\RHom(\mcS(1), \mcE(2)) = 0$ by Lemma~\ref{lem-c2=8}~(3), the semiorthogonal decomposition above yields $\mcV \in \langle \mcO_{\Q^{3}}(-1) \rangle$.
Therefore $\mcV \simeq \mcO_{\Q^{3}}(-1)^{\oplus 7}$, which confirms the existence of the desired exact sequence.
The implication (b) $\ra$ (a) follows from the observation that $\mcE(2)$ is globally generated, $c_{1}(\mcE(2))=3$, and $s_{3}(\mcE(2))=6$. 
\end{proof}

\subsection{Proof~of~Theorem~\ref{main-Q3}}\label{subsec-proof-main-Q3}
%\noindent\emph{\textbf{Proof~of~Theorem~\ref{main-Q3}.}}
By Theorems~\ref{thm-FanoClsf}, \ref{thm-Q3CaylayResol}, \ref{thm-Q3Resol7}, and \ref{thm-Q3Resol8},  it is easy to see that the vector bundle appearing in Theorem~\ref{main-Q3} is weak Fano. 
Conversely, let $\mcE$ be an arbitrary normalized rank $2$ weak Fano bundle. 
By Proposition~\ref{prop-unstableQ3}, 
$\mcE$ is slope stable or isomorphic to one of the direct sums of line bundles in Theorem~\ref{main-Q3}~(i). 
Suppose that $\mcE$ is slope stable. 
Then the pair of Chern classes $(c_{1},c_{2})$ of $\mcE$ satisfies (\ref{eq-candidates}). 
If $(c_{1},c_{2})=(0,2)$, $(-1,1)$, $(-1,2)$, $(-1,3)$, or $(-1,4)$, then by Theorems~\ref{thm-FanoClsf}, \ref{thm-Q3CaylayResol}, \ref{thm-Q3Resol7}, and \ref{thm-Q3Resol8}, 
$\mcE$ is isomorphic to (ii), (iii), (iv), (v), or (vi) in Theorem~\ref{main-Q3}, respectively. 
Finally, we show the existence of an example for each of (i) -- (vi) in Theorem~\ref{main-Q3}. 
This statement is trivial for the cases (i) -- (iv). 
Hence it suffices to show that, for a given $c \in \{3,4\}$, there exists weak Fano bundle $\mcE$ such that $c_{1}(\mcE)=-1$ and $c_{2}(\mcE)=c$. 
By the same argument as in \cite[Section~5.1]{FHI2}, 
this is reduced to the existence of elliptic curves $C$ with $-K_{\Q^{3}}.C=4+c$ for which $\Bl_{C}\Q^{3}$ is weak Fano. 
This existence directly follows from \cite[Theorem~3.2]{ACM17}.
This completes the proof of Theorem~\ref{main-Q3}. 

\begin{rem}
The existence for the case (viii) can also be proved by using a result in \cite{OS94}.
Indeed, [ibid, Theorem~5.2] showed that there exists a rank two slope stable bundle $\mcE$ on $\Q^{3}$ with $(c_{1}(\mcE),c_{2}(\mcE)) = (-1,3)$ and $h^{0}(\mcE(1)) = 0$.
Then Lemma~\ref{lem-orth3}~(2) implies that this $\mcE$ must be weak Fano.
\end{rem}

\section{The moduli space $M^{\wF}_{-1,4}$}\label{sec-Q3Moduli}

Following \cite[Section~5]{FHI2}, let $M_{c_{1},c_{2}}^{\wF}$ denote the moduli space of rank $2$ weak Fano bundles on $\Q^{3}$ with Chern classes $c_{1}$ and $c_{2}$.
For the precise definition of $M_{c_{1},c_{2}}^{\wF}$, see \cite[Definition 5.1]{FHI2}. 
The main ingredient of this section is to study the moduli space $M_{-1,4}^{\wF}$ using Theorem~\ref{main-Q3} and prove Theorem~\ref{main-moduli}. 
To see this, let us prepare the following notations, which are related to quiver representations.

\begin{nota}
\begin{itemize}
\item Put $\mcT := \mcO_{\Q^{3}}(-1) \oplus T_{\P^4}(-2)|_{\Q^{3}}$.
The endomorphism algebra $\End(\mcT)$ is isomorphic to the path algebra $\C$ of the $5$-Kronecker quiver.

\item Let $v_{0}$ (resp. $v_{1}$) be the vertex of $Q$ corresponding to $\mcO_{\Q^{3}}(-1)$ (resp. $T_{\P^4}|_{\Q^{3}}(-2)$). 
For a representation $M$ of $Q$ over $\C$, 
the dimension vector $\underline{\dim}(M)$ is defined as $(\dim_{\C}M_{0},\dim_{\C}M_{1})$, 
where we identify $M$ as a right $\C Q$-module and $M_{i}:=M \cdot e_{v_{i}}$, where $e_{v_{i}}$ is the corresponding idempotent. 

\item Define a stability function $\Theta \colon \Z^{\oplus 2} \to \Z$ as $\Theta(a,b) := 7b-2a$. 

\item For the stability condition $\Theta$, let $M^{\Theta\text{-st}}_{\mathbf{v}}(Q)$ denote the moduli space of stable $Q$-representations of dimension vector $\mathbf{v} \in \Z^{\oplus 2}$ with $\Theta(\mathbf{v}) = 0$
\cite[Proposition 5.2]{King94}.
If $\mathbf{v} = (7,2)$, since it is primitive, $M^{\Theta\text{-st}}_{(7,2)}(Q)$ is a smooth projective variety of dimension $18$
\cite[Section 3.5]{Reineke08}.
\item The bundle $\mcT$ associates an equivalence of triangulated categories
\[ \Phi \colon \langle \mcO_{\Q^{3}}(-1), T_{\P^4}|_{\Q^{3}}(-2) \rangle \ni \mcF \to \RHom(\mcT,\mcF) \in \Db(\mod{\mhyphen}\C Q). \]
Let $\mcA := \Phi^{-1}(\mod{\mhyphen}\C Q)$ be the pull-back of the standard heart under $\Phi$.
Then $\underline{\dim}$ (resp.~$\Theta$) defines the dimension vector (resp.~a stability function) on the abelian category $\mcA$.
\item Let $S_0$ and $S_1$ be the full collection of simple $Q$-representations corresponding to vertices $v_0$ and $v_1$, respectively.
Explicitly, they are given as $S_0 = \Phi(\mcO_{\Q^{3}}(-1))$ and $S_1 = \Phi(\mcO_{\Q^{3}}(-2)[1])$. 
\end{itemize}
\end{nota}

\begin{prop} \label{stab of K}
Let $\mcE$ be a rank two weak Fano bundle on $\Q^{3}$ with $c_{1} = -1$ and $c_{2} = 4$.
Put $\mcF := \mcE(2)$ and $\mcK_{\mcF} := \Ker\left(H^{0}(\mcF) \otimes \mcO_{\Q^{3}} \twoheadrightarrow \mcF\right)$.
Then $\mcK_{\mcF} \in \mcA$ with dimension vector $\underline{\dim}(\mcK_{\mcF}) = (7,2)$, and $\mcK_{\mcF}$ is $\Theta$-stable.
\end{prop}

\begin{proof}
By Theorem~\ref{thm-Q3Resol8}, $\Phi(\mcK_{\mcF})$ fits in an distinguished triangle
\[ S_0^{\oplus 7} \to \Phi(\mcK_{\mcF}) \to S_1^{\oplus 2} \xrightarrow{+1} \]
which is a short exact sequence in $\mod{\mhyphen}\C Q$.
This proves $\mcK_{\mcF} \in \mcA$ with $\underline{\dim}(\mcK_{\mcF}) = (7,2)$.

In order to show the $\Theta$-stability, let $M \subset \mcK_{\mcF}$ be a destabilizing subobject in $\mcA$.
Put $\underline{\dim}(M) := (a,b)$, then the integers $a$ and $b$ satisfies 
$0 \leq a \leq 7$, $0 \leq b \leq 2$, and $7b - 2a \geq 0$.

First, note that there is no non-zero morphism $\mcK_{\mcF} \to \mcO_{\Q^{3}}(-1)$.
Indeed, applying $\Hom(-,\mcO_{\Q^{3}}(-1))$ to the sequence $0 \to \mcK_{\mcF} \to \mcO_{\Q^{3}}^{\oplus 7} \to \mcF \to 0$ and using $\Ext^{1}(\mcF, \mcO_{\Q^{3}}(-1)) \simeq H^{1}(\mcE(-2)) = 0$ show that $\Hom(\mcK_{\mcF}, \mcO_{\Q^{3}}(-1)) = 0$.
Similarly, there is no non-zero morphism $\mcO_{\Q^{3}}(-2)[1] \to \mcK_{\mcF}$ since $\mcK_{\mcF}$ is a coherent sheaf.

Let us prove the non-existence of $M$ that destabilizes $\mcK_{\mcF}$.
First, the inequality $7b - 2a \geq 0$ shows $b \in \{1, 2 \}$.
If $b = 2$, then the quotient $N := \mcK_{\mcF}/M$ (in $\mcA$) has dimension vector $(7-a,0)$, 
and thus $N \simeq \mcO_{\Q^{3}}(-1)^{\oplus (7-a)}$, which is a contradiction.

Now the only remaining possibility is $b =1$, 
and in this case $\underline{\dim}(M) = (a,1)$ and $\underline{\dim}(N) = (7-a, 1)$ with $0 \leq a \leq 3$.
Note that $M$ fits in an exact triangle
\[ \mcO_{\Q^{3}}(-2) \xrightarrow{\alpha} \mcO_{\Q^{3}}(-1)^{\oplus a} \to M \to \mcO_{\Q^{3}}(-2)[1]. \]
If $\alpha = 0$, then $M$ contains $\mcO_{\Q^{3}}(-2)[1]$ as a summand, which is a contradiction since it gives a non-zero morphism $\mcO_{\Q^{3}}(-2)[1] \to \mcK_{\mcF}$.
Thus $\alpha$ is non-zero, and hence is an injective morphism of sheaves.
This implies that $M$ is (isomorphic to) a coherent sheaf. 
Since $a \leq 3$, the locus $(\alpha=0)$ is non-empty, and hence $M$ is not locally free. 
A similar argument proves $N$ is also a coherent sheaf that fits in an exact sequence
\begin{equation} \label{ex-forN}
0 \to \mcO_{\Q^{3}}(-2) \to \mcO_{\Q^{3}}(-1)^{\oplus (7- a)} \to N \to 0.
\end{equation}
In particular, the exact sequence
\begin{equation} \label{ex-forK}
0 \to M \to \mcK_{\mcF} \to N \to 0 
\end{equation}
in $\mcA$ is also a short exact sequence of coherent sheaves.

The exact sequence (\ref{ex-forN}) shows that 
$\mcE{xt}^{i}(N,\mcO_{\Q^{3}})=0$ for $i \geq 2$. 
%a syzygy of $N$ by a locally free sheaf is locally free.
Thus the other exact sequence (\ref{ex-forK}) shows that 
$\mcE{xt}^{i}(M,\mcO_{\Q^{3}})=0$ for $i \geq 1$. 
Hence $M$ is locally free, which is a contradiction. 
%since the morphism $\alpha$ cannot be a bundle injection for any $a \leq 3$.
Therefore, $\mcK_{\mcF}$ cannot be destabilized, which means that it is $\Theta$-stable.
\end{proof}

\subsection{Proof of Theorem~\ref{main-moduli}}
%\begin{proof}[\textbf{Proof of Theorem~\ref{main-moduli}}]
First, we check that $M_{-1,4}^{\wF}$ is smooth of dimension $18$. 
Note that the Hirzebruch--Riemann--Roch formula together with Lemma~\ref{lem-wFbQ3} and \ref{lem-c2=8} implies that
\begin{center}
$\RG(\mcE(-1)) = 0$, $\RG(\mcE(-2)) \simeq \C^{\oplus 21}[-2]$, and $\RG(\mcE(-3)) \simeq \C^{\oplus 4}[-2]$.
\end{center}
Applying the functor $\RHom(\mcE(2), -)$ to the sequence in Theorem~\ref{thm-Q3Resol8} gives an exact sequence
\begin{align*}
\cdots \to H^{1}(\mcE(-2)) &\to \Hom(\mcE, \mcE) \to H^{2}(\mcE(-3)) \to H^{2}(\mcE(-2)) \to \Ext^{1}(\mcE, \mcE) \\
 &\to H^{3}(\mcE(-3)) \to H^{3}(\mcE(-2)) \to \Ext^{2}(\mcE, \mcE) \to 0.
\end{align*}
This shows $\Ext^{1}(\mcE, \mcE) \simeq \C^{18}$ and $\Ext^{2}(\mcE, \mcE) = 0$. 
Hence $M_{-1,4}^{\wF}$ is smooth of dimension $18$. 
With Proposition~\ref{stab of K} already established, 
the proof of \cite[Propositions~6.6 and 6.7]{FHI2} can apply almost verbatim to show 
the existence of an open immersion 
$M_{-1,4}^{\wF} \ni \mcE \mapsto \Phi(\mcK_{\mcE(2)}) \in M^{\Theta\text{-st}}_{(7,2)}(Q)$. 
Since the dimension vector $(7,2)$ is primitive, $M^{\Theta\text{-st}}_{(7,2)}(Q)$ has the universal family, and hence so does $M_{-1,4}^{\wF}$.
In particular, the moduli space $M_{-1,4}^{\wF}$ is both irreducible and fine.
\qed
%\end{proof}

\begin{rem}
Let $M_{-1,4}^{\mathrm{ss}}$ denote the moduli space of rank two semistable bundles on $\Q^{3}$ with $c_{1} = -1$ and $c_{2} = 4$.
Then $M_{-1,4}^{\wF}$ is an open subscheme of $M_{-1,4}^{\mathrm{ss}}$. 
The property that $M_{-1,4}^{\wF}$ is fine can be derived from the fact that $M_{-1,4}^{\mathrm{ss}}$ is fine, which Ottaviani-Szurek showed in \cite[Corollary~of~Proposition~(2.2)]{OS94}. 
However, it seems that the irreducibility of $M_{-1,4}^{\mathrm{ss}}$ is still open.
\end{rem}

\section{Classification on Fano threefolds of index one: $c_{1}(\mcF)$ is even}\label{sec-MukaiEven}

As in \cite{FHI1,FHI2} and the previous discussions in this paper, 
we classified rank $2$ weak Fano bundles on a Fano $3$-fold $X$ of Picard rank $\rho(X)=1$ provided that the Fano index $i_{X}$ is greater than $1$. 
In this section, we give their classification when $i_{X}=1$. 
For a Fano $3$-fold $X$ with $\rho(X)=i_{X}=1$, when $-K_{X}$ is (resp. is not) very ample, $X$ is called a \emph{prime} (resp. \emph{hyperelliptic}) Fano $3$-fold. 

\subsection{Triviality when the case $c_{1}$ is even}

Let $X$ be a Fano $3$-fold with $\rho(X)=i_{X}=1$ and $\mcF$ a rank $2$ weak Fano bundle on $X$. 
Since $\Pic(X) = \Z \cdot c_{1}(X)$, we may assume that $c_{1}(\mcF)=0$ or $c_{1}(\mcF) = c_{1}(X)$. 
First, we treat the case $c_{1}(\mcF)=0$. 
The result is the following theorem.
\begin{thm}\label{thm-split}
Let $X$ be a Fano $3$-fold with $\rho(X)=i_{X}=1$ and $\mcF$ a rank $2$ weak Fano bundle with $c_{1}(\mcF)=0$. 
Then $\mcF=\mcO_{X}^{\oplus 2}$. 
\end{thm}

Note that Corollary~\ref{maincor-parity} now directly follows from this result and the previous works \cite{yas, Ishikawa16, FHI1, FHI2}. 

Our proof of Theorem~\ref{thm-split} uses Theorem~\ref{thm-conic}, which we also used for classification on $\Q^{3}$. 
To use Theorem~\ref{thm-conic}, we examine the discriminant locus of the Hilbert scheme of conics of prime Fano threefolds with genus $g \in \{10,12\}$, as in the following proposition.

\begin{prop}\label{prop-discample}
Let $X$ be a prime Fano $3$-fold of genus $g \in \{10,12\}$. 
Then $-K_{X}$ is very ample, $X$ has a smooth conic with respect to $-K_{X}$,  and $(X,-K_{X})$ satisfies $(\dag)$ (see Definition~\ref{def-dag}). 
\end{prop}
\begin{proof}
We first note that $-K_{X}$ is very ample 
(c.f. \cite[Theorem~1.10]{MukaiDev}) 
and $X$ has a smooth conic 
(c.f. \cite[Theorem~2.1]{Takeuchi89})
with respect to the embedding $X \hra \P^{g+1}$ given by $-K_{X}$. 
Let $S$ be the Hilbert scheme of the conics on $X$. 
Then $S$ is an abelian surface when $g=10$ and $S \simeq \P^{2}$ when $g=12$ \cite[Theorem~1.1.1]{KPS18}. 
In particular, the discriminant locus $\Delta$ is an effective divisor on $S$ by Remark~\ref{rem-discriminant}.
We show the ampleness of $\Delta$ for each $g \in \{10,12\}$. 

Suppose $g=12$. 
In this case, since $S=\P^{2}$, it suffices to show $\Delta>0$. 
When $X$ is the Mukai-Umemura $3$-fold, then \cite[Remark~5.2.15]{Fanobook} shows that $\Delta \subset \P^{2}$ is quadratic and parametrizes the double lines on $X$. 
When $X$ is not the Mukai-Umemura $3$-fold, then 
for each line $l_{1} \subset X$, there exists a different line $l_{2} \subset X$ meeting transversally $l_{1}$ \cite[Remark~4.3.6]{Fanobook}. 
Then $l_{1} \cup l_{2}$ is a degenerated conic, which implies $\Delta >0$. 

Suppose $g=10$. 
Then $X$ is a codimension $2$ linear section of $K(G_{2})$. 
Hence there is a smooth Mukai $4$-fold $F$ of genus $10$ containing $X$. 
We use the explicit description of the Hilbert scheme $H_{F}$ on the conics on $F$ given by Kapustka and Ranestad in \cite[Proposition~3.13]{KR13} as follows. 
They showed that $H_{F}$ is isomorphic to the blowing-up $\Bl_{V_{1}}\P^{5}_{1}$ of the $5$-dimensional projective space $\P^{5}_{1}$ along a Veronese surface $V_{1} \subset \P^{5}_{1}$. 
Thus $H_{F}$ is the graph of the Cremona transformation $\P^{5}_{1} \dra \P^{5}_{2}$ defined by the linear system of quadrics containing $V_{1}$. 
Let $p_{1} \colon H_{F} \to \P^{5}_{1}$ be the blowing-up along $V_{1}$ and $p_{2} \colon H_{F} \to \P^{5}_{2}$ be the restriction of the second projection. 
Then $p_{2}$ is also the blowing-up along a Veronese surface $V_{2} \subset \P^{5}_{2}$. 
Moreover, each $V_{i}$ corresponds to a component of the Hilbert scheme of the cubic surface scrolls $\P_{\P^{1}}(\mcO(1) \oplus \mcO(2))$ in $F$. 
For each point $x \in V_{i}$, the fiber $p_{i}^{-1}(x) \simeq \P^{2}$ corresponds to the Hilbert scheme of conics on the cubic scroll corresponding to $x$. 

Let $H_{i}$ be the pull-back of a hyperplane section on $\P^{5}_{i}$ under $p_{i}$ and $E_{i}$ the exceptional divisor of $p_{i}$. 
Then $H_{2} \sim 2H_{1}-E_{1}$ and $E_{2} \sim 3H_{1}-2E_{2}$. 
Let $\Delta_{F} \subset H_{F}$ be the discriminant divisor on this Hilbert scheme of conics. 
Pick $i \in \{1,2\}$ and an arbitrary point $x \in V_{i}$ and let $T \subset F$ be the cubic scroll corresponding to $x$. 
Let $h$ be a tautological divisor on $T \simeq \F_{1}:=\P_{\P^{1}}(\mcO \oplus \mcO(1))$ and $f$ a ruling of this scroll. 
Then the fiber $p_{i}^{-1}(x)$ is naturally identified with the linear system $\lvert h \rvert$. 
On this linear system, the degenerated conics on $T$ parametrized by $C_{0} + \lvert f \rvert$, where $C_{0}$ is the unique member of $\lvert h-f \rvert$. 
Therefore, the intersection $p_{i}^{-1}(x) \cap \Delta_{F}$ is a line on $\P^{2} \simeq p_{i}^{-1}(x)$. 

Hence for each $i \in \{1,2\}$, $\Delta_{F}|_{E_{i}}$ is a tautological divisor of this $\P^{2}$-bundle $E_{i} \to V_{i} \simeq \P^{2}$. 
Putting $i':=3-i$, we have $\Delta_{F}|_{E_{i}} \sim_{V_{i}} H_{i'}|_{E_{i}}$. 
Hence $\Delta_{F} \sim H_{1} + H_{2}$ holds as $\Pic(H_{F}) = \Z \cdot [H_{1}] \oplus \Z \cdot [H_{2}]$.
In particular, $\Delta_{F}$ is ample, and so is $\Delta$ on the Hilbert scheme $S$ of conics on $X$. 
We complete the proof. 
\end{proof}

\begin{proof}[Proof of Theorem~\ref{thm-split}]
Let $X$ be a Fano $3$-fold with $\rho(X)=i_{X}=1$ 
and $\mcE$ a rank $2$ weak Fano bundle with $c_{1}(\mcE)=0$. 
Let $\pi \colon \P(\mcE) \to X$ be the projectivization. 
Let $\xi$ be a tautological divisor, $H_{X}=-K_{X}$, $g:=\frac{1}{2}H_{X}^{3}+1$, and $H=\pi^{\ast}H_{X}$. 
It is known that $g$ is a positive integer with $2 \leq g \leq 12$ and $g \neq 11$ (see e.g. \cite{MukaiDev}). 
By taking a line on $X$ \cite{Sho79}, we fix the identification $A^{2}(X)_{\Z} \simeq \Z$ and let $c_{2} \in \Z$ be the integer corresponding to $c_{2}(\mcE) \in A^{2}(X)_{\Z}$. 
Note that $-K_{\P(\mcE)}=2\xi+H$ is nef and big by definition and hence $\mcE(1)$ is ample. 
\begin{claim}\label{num-claim}
Let $\beta:=\min\{b \in \Z \mid h^{0}(\mcE(b))>0\}$. 
\begin{enumerate}
\item $\beta \geq 0$ and $d\beta^{2}+c_{2} \geq 0$. In particular, $\mcE$ is slope semistable. 
\item $\chi(\mcE)=2-\frac{1}{2}c_{2}$. In particular, $c_{2} \equiv 0 \pmod{2}$. 
\item $H^{p}(X,\mcE) = 0$ for every $p \geq 2$. 
\item $(-K_{\P(\mcE)})^{4}=8((2g-2)-4c_{2})$. 
\item $(-K_{\P(\mcE)})^{3} \xi = (2g-2)-12c_{2}$. 
\end{enumerate}
\end{claim}
\begin{proof}
(1) There is an exact sequence 
\begin{align}\label{ex-beta}
0 \to \mcO(-\beta) \to \mcE \to \mcI_{Z}(\beta) \to 0,
\end{align}
where $Z$ is a purely $2$-codimensional closed subscheme such that $[Z] \sim c_{2}(\mcE(\beta))=d\beta^{2}+c_{2}(\mcE) \geq 0$. 
In particular, $\mcE(-\beta)$ has $\mcI_{Z}$ as a quotient. 
Since $\mcE(1)$ is ample, we obtain $\beta \geq 0$. 
(2) follows from the Hirzebruch--Riemann--Roch theorem. 
(3) follows from Le Potier's vanishing theorem \cite[Theorem~7.3.5]{laz2}. 
(4) and (5) follow from direct computations. 
\end{proof}

\begin{claim}\label{claim-numsplit}
If $c_{2} < 4$, then $\mcE \simeq \mcO_{X}^{\oplus 2}$. 
\end{claim}
\begin{proof}
By Claim~\ref{num-claim}~(2) and (3), we have $c_{2} \leq 2$ and $h^{0}(\mcE) \geq \chi(\mcE)>0$. 
Hence $\beta=0$ and $\xi$ is effective. 
By Claim~\ref{num-claim}~(1) and (2), we have $c_{2} \in \{0,2\}$. 
Then Claim~\ref{num-claim}~(5) implies that $c_{2}=0$ since $g \leq 12$. 
Thus the closed subscheme $Z$ in (\ref{ex-beta}) is empty and the exact sequence splits. 
Hence $\mcE = \mcO_{X}^{\oplus 2}$. 
\end{proof}

Suppose that $\mcE \not\simeq \mcO^{\oplus 2}$. 
Note that $\mcE$ is semistable by Claim~\ref{num-claim}~(1). 
By Claim~\ref{claim-numsplit}, we have $4 \leq c_{2}$. 
By Claim~\ref{num-claim}~(4), we have $c_{2} < (g-1)/2$. 
Hence we have $g \in \{10,12\}$. 
By Proposition~\ref{prop-discample}, 
$(X,H_{X})$ satisfies the condition $(\dag)$ in Definition~\ref{def-dag}. 
Thus we can apply Theorem~\ref{thm-conic}~(2) for $X$ and $\mcE$, which implies $\mcE \simeq \mcO_{X}^{\oplus 2}$, a contradiction. 
This completes the proof of Theorem~\ref{thm-split}. 
\end{proof}

\section{Classification on Fano threefolds of index one: $c_{1}(\mcF)$ is odd}\label{sec-MukaiOdd}

Let $X$ be a Fano $3$-fold with $\rho(X)=i_{X}=1$ and $(-K_{X})^{3}=2g-2$. 
The remaining classification problem is now the classification of rank $2$ weak Fano bundle $\mcF$ with $c_{1}(\mcF)=c_{1}(X)$. 
As we will see, if $\mcF$ is indecomposable, then $g \geq 6$ and the range of values that the $2$nd Chern class $c_{2}(\mcF)$ can take is $\lfloor \frac{g+3}{2} \rfloor \leq c_{2}(\mcF) \leq g-2$. 
The most difficult part of our classification is whether such an $\mcF$ actually exists on an arbitrary prime Fano $3$-fold of genus $g$ that satisfies $c_{2}(\mcF)=d$ for each $d$ with $\lfloor \frac{g+3}{2} \rfloor \leq d \leq g-2$. 

In the case $X$ is $\P^{3}$, $\Q^{3}$, or a del Pezzo $3$-fold of degree $5$, since $X$ itself has no moduli, we were able to prove the existence of such an $\mcF$ by 
applying Arap-Cutrone-Marshburn's construction of an elliptic curve with the desired properties \cite{ACM17}
as done in \cite{FHI2} and Section~\ref{subsec-proof-main-Q3} of this paper. 
By similar methodology, for fixed integers $g$ and $d$ satisfying inequality $\lfloor \frac{g+3}{2} \rfloor \leq d \leq g-2$, 
it is possible to show that such a pair $(X,\mcF)$ with $c_{2}(\mcF)=d$ exists. 
However, when $X$ is a prime Fano $3$-fold, we need to show that such an $\mcF$ for an \emph{arbitrary} $X$, because $X$ has non-trivial moduli.
To prove the existence of such an $\mcF$ for any $X$, we use the results for ACM bundles obtained in \cite{BF11,CFK}. 

\subsection{Numerical preparation}\label{subsec-prep}

In Section~\ref{subsec-prep}, we employ the following notation. 
\begin{nota}\label{nota-c1Fc1X}
\begin{itemize}
\item Let $X$ be a Fano $3$-fold of Picard rank $1$ whose Fano index is $i_{X}$ and denote a fundamental divisor by $H_{X}$. 
\item Let $\mcF$ be a weak Fano vector bundle with $c_{1}(\mcF)=c_{1}(X)$ and $\rk \mcF=2$. 
\item Let $\pi \colon Y:=\P_{X}(\mcF) \to X$ be the projectivization of $\mcF$ and $\xi$ a tautological divisor. 
Since $-K_{Y} \sim 2\xi$, $\xi$ is semi-ample by the Kawamata--Shokurov base point free theorem \cite[Theorem~3.3]{KM98}. 
\item Let $\psi \colon Y \to \ol{Y}$ be the contraction induced by $\xi$ and $\ol{\xi}$ an ample Cartier divisor on $\ol{Y}$ such that $\xi = \psi^{\ast}\ol{\xi}$. Note that $-K_{\ol{Y}} \sim 2\ol{\xi}$. 
\item We take a smooth ladder from $\lvert \xi \rvert$, say 
\begin{align}\label{eq-ladder}
S \subset \wt{X} \subset Y, 
\end{align}
which exists according to \cite[Theorem~4.1]{FHI1}. 
\item Let $\ol{S} \subset \ol{X} \subset \ol{Y}$ denotes the corresponding ladder on $\ol{Y}$. 
\item By the same argument as in \cite[Section~4.3.1 and Claim~4.4]{FHI1}, 
the weak Fano $3$-fold $\wt{X}$ is the blowing-up along a (possibly disconnected or empty) smooth curve $C$ on $X$. 
Moreover, each connected component $C_{i}$ of $C$ is an elliptic curve. 
Here we have an exact sequence
\begin{align}\label{ex-wtX}
0 \to \mcO_{X} \to \mcF \to \mcI_{C}(-K_{X}) \to 0 \text{ with }c_{2}(\mcF) \equiv C.
\end{align}
When $\mcF$ is globally generated moreover, we may assume that $S \to \pi(S)$ is isomorphic. 
In this case, we often regard $S$ as a smooth anticanonical member of $X$ containing $C$ through this identification.
\end{itemize}
\end{nota}

\begin{lem}\label{lem-c20}
The following conditions are equivalent: 
\begin{enumerate}
\item[(a)] $c_{2}(\mcF)=0$. 
\item[(b)] $\mcF \simeq \mcO_{X} \oplus \mcO(-K_{X})$. 
\item[(c)] $\psi$ is divisorial and $\dim \psi(\Exc(\psi)) \leq 1$. 
\item[(d)] $\psi$ has a fiber whose dimension is greater than $1$. 
\end{enumerate}
\end{lem}
\begin{proof}
The implication (a) $\ra$ (b) follows from the exact sequence (\ref{ex-wtX}).
The implications (b) $\ra$ (c) and (c) $\ra$ (d) are clear. 
Let us show (d) $\ra$ (a). 
Suppose there is a surface $J$ such that $\psi(J)$ is a point. 
Then $\pi|_{J} \colon J \to X$ is finite. 
Since $\xi|_{J} \sim 0$, the Grothendieck relation gives 
$0=\pi|_{J}^{\ast}c_{2}(\mcF)$. 
Hence $c_{2}(\mcF)=0$. 
We complete the proof. 
\end{proof}

From now on, we mainly treat the case that $C \neq \emp$, which is equivalent to saying that $c_{2}>0$, and $\xi$ is not ample. 
The main aim of this section is to prove the following proposition: 
\begin{prop}\label{prop-BNineq}
Suppose that $c_{2}(\mcF) \neq 0$ and $\xi$ is not ample. 
\begin{enumerate}
\item $C$ is connected, which is equivalent to $H^{1}(\mcF^{\vee})=H^{2}(\mcF^{\vee})=0$. 
\item $H_{X}$ is very ample. 
\item If $i_{X} = 1$, then $c_{2} \geq 4$ and $c_{1}(X)^{3} \geq 10$. 
If $i_{X}=2$, then $c_{2} \geq H_{X}^{3}+1$. 
If $i_{X}=3,4$, then $c_{2} \geq 6$. 
%In particular, if $i_{X}=1$ and $X$ is hyperelliptic, then $\mcF=\mcO_{X} \oplus \mcO(-K_{X})$.
\item If $i_{X} = 1$, put $g:=\frac{c_{1}(X)^{3}}{2}+1$. 
Then the following inequality hold:
\[\lfloor \frac{g+3}{2} \rfloor \leq c_{1}(\mcF)c_{2}(\mcF) \leq g-2.\]
In particular, we have $g \geq 6$. 
%\item For every elliptic curve $C$ on $X$, if $\mcI_{C}(-K_{X})$ is globally generated, then 
%\[\lfloor \frac{g+3}{2} \rfloor \leq -K_{X}.C \leq g-2.\]
\end{enumerate}
\end{prop}
\begin{proof}
Since we assume that $\xi$ is not ample, $\mcF$ is globally generated by \cite[Theorem~1.7]{FHI1}. 

(1) Since $\mcF^{\vee} \simeq \mcF(K_{X})$, we have $h^{1}(\mcF^{\vee})=h^{2}(\mcF^{\vee})$. 
Moreover, (\ref{ex-wtX}) shows $h^{1}(\mcF^{\vee})=h^{1}(\mcF(K_{X}))=h^{1}(\mcI_{C})$. 
Hence the condition that $C$ is connected is equivalent to $h^{1}(\mcF^{\vee})=0$. 
This is also equivalent to $\rho(\wt{X})=2$ for a general member $\wt{X} \in \lvert \xi \rvert$. 
By Lemma~\ref{lem-c20}, if the contraction $\psi$ is divisorial, then $\dim \psi(\Exc(\psi)) = 2$. 
Then it follows from Theorem~\ref{thm-RSGNL} that $\rho(\wt{X})=2$ for general $\wt{X}$. 

(2) 
Since $-K_{X}|_{C}$ is globally generated, we have $-K_{X}.C \geq 2$. 
Note that
\begin{align}\label{eq-s3}
s_{3}(\mcF)=c_{1}(\mcF)^{3}-2c_{1}(\mcF)c_{2}(\mcF)>0. 
\end{align}
Hence $c_{1}(X)^{3} > 4$. 
Thus Fujita-Iskovskikh classification of del Pezzo $3$-fold shows that if $H_{X}$ is not very ample, then $X$ is a del Pezzo $3$-fold of degree $d \leq 2$ (see e.g. \cite[Proposition~3.2.4]{Fanobook}). 
By \cite[Theorem~1.5]{FHI2}, $\mcF$ is a direct sum of line bundles, which contradicts our assumptions that $c_{2} \neq 0$ and $\xi$ is not ample. 
Therefore, $H_{X}$ is very ample. 

(3) 
Suppose $i_{X}=1$. 
Then $c_{1}(\mcF)c_{2}(\mcF) \geq 3$ holds by (2) and (\ref{eq-s3}) implies $c_{1}(X)^{3} > 6$. 
Hence $X$ is defined by quadratic equations in $\P(H^{0}(-K_{X}))$. 
If $ -K_{X}.C = 3$ holds, then the degree-genus inequality shows $C$ is a plane cubic curve, which is a contradiction since $X$ must contain the $2$-plane $\gen{C}$ spanned by $C$. Hence $c_{2} \geq 4$, and then (\ref{eq-s3}) implies $c_{1}(X)^{3} \geq 10$. 
The remaining inequalities in (2) for the case $i_{X} \geq 2$ follow from the classification results including \cite{yas}, \cite{FHI2} and Theorem~\ref{main-Q3}.

(4) 
The inequality $c_{2} \leq g-2$ is just a rephrasing of (\ref{eq-s3}). 
By (2), we may assume that $-K_{X}$ is very ample. 
Thus the Brill--Noether generality of $S$ deduces the second inequality as follows. 
Since $\mcF$ is globally generated, we may further assume that $S \to \pi(S)$ is isomorphic and thus $S \simeq \pi(S) \subset X$ contains $C$. 
Write $H_{S}:=-K_{X}|_{S}$. 
Then $h^{0}(S,\mcO_{S}(C))=2$ and $h^{0}(S,\mcO(H_{S}-C))=\frac{s_{3}(\mcF)}{2}+2$. 
%\begin{align*}
%h^{0}(S,\mcO_{S}(C))
%&= \chi(\mcO_{S}(C)) = 2, \text{ and } 
%h^{0}(S,\mcO(H_{S}-C))
%&= \chi(\mcO_{S}(H_{S}-C)) \\
%&= \chi(\mcO_{S}(H_{S}))-\chi(\mcO_{C}(H_{S}|_{C})) \\
%&= \frac{c_{1}(X)^{3}}{2}+2-(H_{S}.C) \\
%&=\frac{s_{3}(\mcF)}{2}+2.
%\end{align*}
By Theorem~\ref{thm-LazBN}~(3), this polarized K3 surface $(S,H_{S})$ of genus $g$ is Brill--Noether general, which implies that
$
g \geq h^{0}(\mcO_{S}(C)) \cdot h^{0}(\mcO_{S}(H_{S}-C)) \geq s_{3}(\mcF)+4
$. 
This is equivalent to the inequality $\lfloor \frac{g+3}{2} \rfloor \leq c_{2}$. Hence (4) holds. 
\end{proof}

\begin{rem}
On a prime Fano $3$-fold $X$, the lower bound $\lfloor \frac{g+3}{2} \rfloor \leq -K_{X}.c_{2}(\mcF)$ is originally proved by \cite{Madonna01} to classify ACM bundles on $X$ and also proved in \cite[Corollary~3.9]{CFK} to compute the dimension of the moduli of ACM bundles. 
Their proofs are slightly different from our way. 
\end{rem}

\subsection{Elliptic normal curves on prime Fano threefolds}\label{subsec-exprelim}

Let $X$ be a prime Fano $3$-fold of genus $g \geq 6$. 
Let $d$ be an integer such that $\lfloor \frac{g+3}{2} \rfloor \leq d \leq g-2$. 
By \cite[Theorem~1.1]{CFK} and its proof, there exists an elliptic normal curve $C$ on $X$ such that $-K_{X}.C=d$. 
Let 
\begin{align}\label{eq-defofHd}
\mathfrak{H}_{d}(X)' \subset \Hilb_{dt}(X)
\end{align}
be an irreducible locally closed subset that parametrizes elliptic normal curves on $X$ of degree $d$. 
By examining elliptic normal curves in more detail, we can show the following proposition, which strengthens [ibid, Proposition~3.8]. 
\begin{prop}{\cite[Proposition~3.8~(v)]{CFK}}\label{prop-CFK}
Let $X$ be a prime Fano $3$-fold of genus $g$.
Let $d$ be an integer that satisfies $\lfloor \frac{g+3}{2} \rfloor \leq d \leq g-2$. 
The following assertions hold for a general member $E_{d} \in \mathfrak{H}_{d}(X)'$. 
\begin{enumerate}
\item A general member $S \in \lvert \mcI_{E_{d}}(-K_{X}) \rvert$ is smooth or $\Sing(S)=\{p\} \subset E_{d}$. Moreover, if $S$ is singular, then $p$ is a rational double point of $S$. 
\item For every line $l$ on $X$, $l$ is away from $E_{d}$ or intersects $E_{d}$ transversally at one point. 
\item When $g \geq 9$, then for every conic $\gamma$ on $X$, $\gamma$ is away from $E_{d}$ or $\lg(\mcO_{\gamma \cap E_{d}}) \leq 2$. 
\end{enumerate}
\end{prop}
The statements (1) and (2) of Proposition~\ref{prop-CFK} are included in \cite[Proposition~3.8~(v)]{CFK} and we recall them to use these properties explicitly in the next section. 
For our sake, the additional property (3) is necessary. 
The aim of Section~\ref{subsec-exprelim} is to prove (3). 

\begin{lem}\label{lem-RSNLconic}
Let $X$ be a prime Fano $3$-fold of genus $g \geq 9$ and $\gamma \subset X$ a smooth conic. 
Then for a very general member $S \in \lvert \mcI_{\gamma/X}(1) \rvert$, $S$ is a smooth K3 surface and its Picard group $\Pic(S)$ is equal to $\Z[-K_{X}|_{S}] \oplus \Z[\gamma]$. 
\end{lem}
\begin{proof}
Let $\s \colon \wt{X}:=\Bl_{\gamma}X \to X$ be the blowing-up. 
Since $X$ is defined by quadratic equations, $-K_{\wt{X}}$ is base point free. 
Let $\psi \colon \wt{X} \to \ol{X}$ be the contraction given as the Stein factorization of the morphism defined by $\lvert -K_{\wt{X}} \rvert$. 
Then it was shown in \cite[Corollary~4.4.3]{Fanobook} that $\psi$ is a flopping contraction. 
Since we assume $g \geq 9$, $-K_{\ol{X}}$ is very ample by \cite[Proposition~2.7]{JPR11}. 
Hence Corollary~\ref{cor-RSNL} shows the result. 
\end{proof}

\begin{lem}\label{lem-locusconic}
Let $X$ be a prime Fano $3$-fold of genus $g \geq 9$. 
Then the locus
\[\mathfrak{S}_{\mathrm{conic}}:=\{S \in \lvert \mcO_{X}(1) \rvert \mid S \text{ contains a conic } \}\]
is a prime divisor of $\lvert \mcO_{X}(1) \rvert \simeq \P^{g+1}$ and a very general member $S \in \mathfrak{S}_{\mathrm{conic}}$ is smooth and has no effective divisors $E$ with $E^{2}=0$ and $H.E \leq g-2$. 
\end{lem}
\begin{proof}
%Note that $\dim \lvert \mcO_{X}(1) \rvert=g+1$ and $\mathfrak{S}_{d}(X)$ is a $g$-dimensional locally closed subset of $\lvert \mcO_{X}(1) \rvert$. 
For each conic $\gamma \subset X$, 
the hyperplane sections containing $\gamma$ are parametrized by $\lvert \mcI_{\gamma} \otimes \mcO_{X}(1) \rvert$. 
Since $X$ is defined by quadratics equations, the sheaf $\mcI_{\gamma} \otimes \mcO_{X}(1)$ is globally generated, and hence $\dim \lvert \mcI_{\gamma} \otimes \mcO_{X}(1) \rvert = g-2$. 
As the Hilbert scheme of conics on $X$ is an irreducible smooth surface by \cite{KPS18}, $\mathfrak{S}_{\mathrm{conic}}$ is an irreducible hypersurface on $\lvert \mcO_{X}(1) \rvert=\P^{g+1}$. 

Moreover, Lemma~\ref{lem-RSNLconic} shows a very general member $S \in \mathfrak{S}_{\mathrm{conic}}$ contains a smooth conic $\gamma$ and $\Pic(S) \simeq \Z[H] \oplus \Z[\gamma]$, where $H:=-K_{X}|_{S}$. 
Suppose that $S$ has an effective divisor $E$ with $E^{2}=0$ and $d:=H.E \leq g-2$. 
Pick integers $a,b$ such that $E \sim aH+b\gamma$ in $\Pic(S)$. 
Then
\begin{align*}
0 &= E^{2} = a^{2}(2g-2)+4ab - 2b^{2} \text{ and } \\
d &= H.E = a(2g-2)+2b.
\end{align*}
Hence 
$a=\frac{d}{2(g-1)}(1 \pm \frac{1}{\sqrt{g}})$ and $b=\mp \frac{d}{2\sqrt{g}}$. 
Since $9 \leq g \leq 12$, we have $g=9$ and hence 
$(a,b)=(d/12,-d/6)$. 
Hence $d$ is divisible by $12$, which contradicts $d \leq g-2=7$. 
\end{proof}
\begin{proof}[Proof~of~Proposition~\ref{prop-CFK}]
(1) and (2) was shown in \cite[Proposition~3.8]{CFK}. 
To prove the remaining statement (3), suppose $g \geq 9$. 
Recall the locally closed subset $\mathfrak{S}_{d}(X)  \subset \lvert \mcO_{X}(1) \rvert$ which is given as follows in [page 7, ibid.]. 
Let $\ol{\mathfrak{H}_{d}}(X)$ be the Hilbert scheme parametrizing a locally complete intersection $1$-dimensional subscheme $E_{d}$ whose dualizing sheaf $\omega_{E_{d}}$ is trivial and $-K_{X}.E_{d}=d$. 
For each $E_{d} \in \ol{\mathfrak{H}_{d}}(X)$, 
we obtain a unique rank $2$ vector bundle $\mcF_{E_{d}}$ fitting into 
$0 \to \mcO \to \mcF_{E_{d}} \to \mcI_{E_{d}}(-K_{X}) \to 0$. 
For each injection $s=(s_{1},s_{2}) \colon \mcO^{\oplus 2} \to \mcF$, an anticanonical member is given by taking the determinant $\det s \colon \mcO \to \det \mcF=\mcO(-K_{X})$. 
This gives a regular map
\[w_{E_{d}} \colon \Gr(2,H^{0}(\mcF_{E_{d}})) \to \lvert \mcO_{X}(1) \rvert.\]
The locus $\mathfrak{S}_{d}(X)$ is then defined by 
\[\mathfrak{S}_{d}(X):=\bigcup_{E_{d} \in \ol{\mathfrak{H}_{d}}(X)} \Im(w_{E_{d}}).\]
From this construction, a member $S \in \Im(w_{E_{d}})$ contains $E_{d}$ as a closed subscheme. 

We return to the proof. By [ibid, Proposition~3.8], each irreducible component $\mathfrak{S}'_{d}$ of $\mathfrak{S}_{d}(X)$ has dimension $g$. 
Moreover, $\mathfrak{S}'_{d}$ can be written by $\bigcup_{E_{d} \in \ol{\mathfrak{H}}_{d}'(X)} \Im (w_{E_{d}})$, where $\ol{\mathfrak{H}}_{d}'(X)$ is an irreducible component of $\ol{\mathfrak{H}}_{d}(X)$. 
Therefore, if $\mathfrak{S}_{\mathrm{conic}}(X)$ contains $\mathfrak{S}'_{d}$, 
then a general member of $\mathfrak{S}_{\mathrm{conic}}(X)$ contains 
a locally complete intersection $1$-dimensional subscheme $E$ with $\omega_{E} \simeq \mcO_{E}$ and $-K_{X}.E \leq g-2$, which contradicts Lemma~\ref{lem-locusconic}. 

Therefore, for general $E_{d} \in \mathfrak{H}_{d}(X)'$ and a general member $S \in \Im(w_{E_{d}})$, $S$ does not contain any conics.
In particular, every conic $\gamma$ satisfies $\lg(\mcO_{S \cap \gamma})=2$ as $S$ is a hyperplane section, and hence $\lg(\mcO_{E_{d} \cap \gamma}) \leq 2$ as $E_{d} \subset S$. 
We complete the proof. 
\end{proof}

\subsection{Existence Theorem on prime Fano threefolds}

Finally, we show the existence for each possibility of the second Chern class $c_{2}(\mcF)$. 

\begin{thm}\label{thm-exist}
Let $X$ be a prime Fano $3$-fold of genus $g \geq 6$. 
Let $d$ be an integer satisfies that $\lfloor \frac{g+3}{2} \rfloor \leq d \leq g-2$. 

For an elliptic normal curve $C \in \mathfrak{H}_{d}(X)'$ of degree $d$, 
let $\mcF$ be a rank $2$ ACM vector bundle given by the Hartshorne--Serre correspondence fitting into 
\begin{align}\label{seq-3fold}
0 \to \mcO \xrightarrow{s} \mcF \to \mcI_{C}(-K_{X}) \to 0,
\end{align}
as in \cite[Theorem~1.2]{CFK}. 
For general $C \in \mathfrak{H}_{d}(X)'$, $\mcF$ is globally generated. 
\end{thm}

By \cite[Theorem~1.7]{FHI1}, this theorem is equivalent to saying that the ACM bundle $\mcF$ is nef if we choose $C$ sufficiently general. 
We will see that the nefness of $\mcF$ comes down to the nefness of a certain Cartier divisor $D$ on a Du Val K3 surface $S$ with $H^{>0}(\mcO(D))=0$. 
To see the nefness of such a $D$, we prepare the following lemma. 

\begin{lem}\label{lem-K3dcp}
Let $T$ be a smooth K3 surface and $D$ an effective divisor with $H^{>0}(\mcO(D))=0$. 
Then there exists a unique decomposition
\[D \sim P_{D}+\Gamma\]
where 
\begin{itemize}
\item $P_{D}$ is a nef divisor with $H^{p}(\mcO(P_{D})) \xrightarrow{\sim} H^{p}(\mcO(D))$ for every $p$, and 
\item $\Gamma=0$ or a reduced union of $(-2)$-curves contained in the fixed locus of $D$. 
\end{itemize}
If $\Gamma \neq 0$ moreover, $\Gamma$ is the sum of $(-2)$-curves
\[ \Gamma=\sum_{i=1}^{n} \Gamma_{i} = \sum_{i=1}^{n} \bigsqcup_{j=1}^{m_{i}} \Gamma_{i,j}\]
satisfying the following properties. 
\begin{enumerate}
\item For each $i \in \{1,\ldots,n\}$, the $(-2)$-curves $\Gamma_{i,1},\ldots,\Gamma_{i,m_{i}}$ are mutually disjoint. 
\item Define $D_{0}=D$ and $D_{k}:=D-\sum_{i=1}^{k}\Gamma_{i}$ for $k \in \{1,\ldots,n\}$. 
Then $\Gamma_{i,j}$ is a $(-2)$-curve with $D_{i}.\Gamma_{i+1,j}=-1$ for every $j$ and the set $\{\Gamma_{i+1,j} \mid 1 \leq j \leq m_{i}\}$ coincides with $\{C \subset T \mid C\text{ is an irreducible curve with } D_{i}.C<0\}$. 
\end{enumerate}
In particular, $\Gamma_{1,j}$ with $j \in \{1,\ldots,m_{1}\}$ satisfies $D.\Gamma_{1,j}=-1$. 
%In particular, $\Gamma_{n,j}$ with $j \in \{1,\ldots,m_{n}\}$ satisfies $M_{D}.\Gamma_{n,j}=1$. 
\end{lem}
\begin{proof}
There is nothing to prove when $D$ is nef. 
Suppose that $D$ is not nef. 
Let $\{\Gamma_{1,i}\}_{i=1,\ldots,m_{1}}$ be the set of curves with $D.\Gamma_{1,i}<0$. 
Fix an arbitrary $i$. 
Then $\Gamma_{1,i}$ is also an irreducible component of the fixed part of $\lvert D \rvert$. 
Hence $\Gamma_{1,i}$ is a $(-2)$-curve and $D-\Gamma_{1,i}$ is a non-zero effective divisor. 
Consider an exact sequence
\[0 \to \mcO(D-\Gamma_{1,i}) \to \mcO(D) \to \mcO(D)|_{\Gamma_{1,i}} \to 0.\]
Since $H^{1}(\mcO(D))=0$ and $H^{2}(\mcO(D-\Gamma_{1,i}))=H^{0}(\mcO(\Gamma_{1,i}-D))^{\vee}=0$, $H^{1}(\mcO(D)|_{\Gamma_{1,i}})=0$. 
Since $D.\Gamma_{1,i}<0$, we have $D.\Gamma_{1,i}=-1$. 
In particular, the map $H^{p}(\mcO(D-\Gamma_{1,i})) \to H^{p}(\mcO(D))$ is isomorphic for every $p$. 
For $i' \neq i$, it holds that $(D-\Gamma_{1,i}).\Gamma_{1,i'}=-1-\Gamma_{1,i}.\Gamma_{1,i'} \leq -1$. 
Since $H^{>0}(\mcO(D-\Gamma_{1,i}))=0$ hold, 
one has $D-\Gamma_{1,i} \neq \Gamma_{1,i'}$,
and hence $H^{2}(\mcO(D-\Gamma_{1,i}-\Gamma_{1,i'}) = 0$.
This shows that $H^{1}(\mcO_{\Gamma_{1,i'}}(D-\Gamma_{1,i})) = 0$
and thus $(D-\Gamma_{1,i}).\Gamma_{1,i'}=-1$, which gives $\Gamma_{1,i}.\Gamma_{1,i'}=0$. 
Hence the curves $\Gamma_{1,1},\ldots,\Gamma_{1,m_{1}}$ are mutually disjoint, and it holds that $D.\Gamma_{1,i}=-1$ for every $i$. 

Put $\Gamma_{1}:=\sum_{i=1}^{m}\Gamma_{1,i}$ and $D_{1}:=D-\Gamma_{1}$. 
Note that $H^{p}(\mcO(D_{1})) \to H^{p}(\mcO(D))$ is isomorphic for every $p$. 
If $D_{1}$ is still not nef, then let $\{\Gamma_{2,j}\}_{j=1,\ldots,m_{2}}$ be the set of curves such that $D_{1}.\Gamma_{2,j}<0$. 
Then by the exactly same argument, the curves $\Gamma_{2,1},\ldots,\Gamma_{2,m_{2}}$ are mutually disjoint and it holds that $D_{1}.\Gamma_{2,j}=-1$ for every $j$. 
Let $\Gamma_{2}:=\sum_{j}\Gamma_{2,j}$ and $D_{2}:=D_{1}-\Gamma_{2}$. 
Since the maps $H^{p}(\mcO(D_{2})) \to H^{p}(\mcO(D_{1})) \to H^{p}(\mcO(D))$ are isomorphic for every $p$, 
the divisor $\sum_{i=1}^{m_{1}} \Gamma_{1,i} + \sum_{j=1}^{m_{2}} \Gamma_{2,j}$ lies in the fixed locus of $\lvert D \rvert$. 

By repeating this process, we obtain the desired decomposition
\[
D=P + \sum_{i=1}^{n} \Gamma_{i} \text{ with } \Gamma_{i}=\bigsqcup_{j=1}^{m_{i}} \Gamma_{i,j},
\]
where $P$ is a nef divisor and $\Gamma_{i}$ is a disjoint union of $(-2)$-curves $\Gamma_{i,j}$ with $j \in \{1,\ldots,m_{i}\}$. 
Let $D_{0}:=D$ and $D_{k}:=D-\sum_{i=1}^{k} \Gamma_{i}$ for $k \in \{1,\ldots,n\}$. 
Note that $D_{n}=P$. 
Then $D_{k-1}.\Gamma_{k,j}=-1$ for all $j$ and $\{\Gamma_{k,j} \mid 1 \leq j \leq m_{k}\}$ coincides with the set of curves which negatively intersect with $D_{k-1}$. 
From our construction, the morphisms $H^{p}(\mcO(P))=H^{p}(\mcO(D_{n})) \to H^{p}(\mcO(D_{n-1})) \to \cdots \to H^{p}(\mcO(D_{1})) \to H^{p}(\mcO(D))$ are isomorphic for every $p$. 
In particular, $\sum_{i}\Gamma_{i}$ is contained in the fixed locus of $\lvert D \rvert$. 
\end{proof}

\begin{proof}[Proof~of~Theorem~\ref{thm-exist}]
Pick a general elliptic normal curve $C$ and a general member $S \in \lvert \mcI_{C}(-K_{X}) \rvert$ as in Proposition~\ref{prop-CFK}. 
We proceed with $4$ steps.

\emph{Step 1.} In this step, we check that we can reduce our problem to the nefness of a divisor on a certain Du Val K3 surface. 

By \cite[Theorem~1.7]{FHI1}, it suffices to show $\mcF$ is nef. 
Let $\pi \colon Y:=\P_{X}(\mcF) \to X$ be the projectivization and $\xi$ a tautological divisor. 
By (\ref{seq-3fold}), the member $X' \in \lvert \xi \rvert$ corresponding to the section $s$ is the blowing-up of $X$ along $C$. 
%Note that $-K_{X'} \sim \xi|_{X'}$ since $-K_{Y} \sim 2\xi$. 
Let $S':=\Bl_{C}S$ be the proper transform of $S$. 
Then $S' \in \lvert -K_{X'} \rvert=\lvert \xi|_{X'} \rvert$. 
Let $\s:=\pi|_{X'}$, $f:=\pi|_{S'}$, $E:=\Exc(\s)$, $E_{S'}:=E|_{S'}$, and $H_{S'}:=f^{\ast}(-K_{X}|_{S})$. 
When $S$ is singular, let $p$ denote the unique singular point of $S$ and $F:=\pi^{-1}(p) \simeq \P^{1}$. 
Then the effective Cartier divisor $E_{S'}$ has a unique irreducible component $C'$ which is a section of $\s|_{E} \colon E \to C$. 
Then it suffices to show the Cartier divisor $\xi|_{S'} \sim H_{S'}-E_{S'}$ is nef.

\emph{Step 2.} 
To treat the case where $S'$ is singular, we prepare the following claim. 

\begin{claim}\label{claim-K3exc}
Suppose that $S'$ is singular.
\begin{enumerate}
\item The $\pi$-fiber $F$ is contained in $S'$. 
\item $S'$ has only rational double points. Moreover, the singular locus of $S'$ lies on $F$. 
\item There exist a positive integer $a > 0$ and an equality of Weil divisors $E_{S'}=C'+aF$. 
\end{enumerate}
\end{claim}
\begin{proof}
(1) Note that every fiber of $S' \to S$ is connected and $S$ is normal by Proposition~\ref{prop-CFK}~(1). 
If $F \not\subset S'$, then $S' \to S$ is bijective and hence an isomorphism by Zariski main theorem. 
Since $S'=\Bl_{C}S$, it follows that $C$ is a Cartier divisor on $S$, which contradicts that $\{p\}=\Sing S$ is contained in $C$. 
Hence $F \subset S'$. 

(2) Since $\Sing S=\{p\}$, we have $\Sing S' \subset F$. 
First we show $S'$ is normal. 
Assume the contrary; then $S'$ is singular along $F$. 
Since $K_{S'} \sim 0$ by the adjunction formula, $S'$ is in particular Gorenstein. 
Take a normalization $\nu \colon \ol{S'} \to S'$. 
Then, $K_{\ol{S'}} + \Delta_{\ol{S'}} \sim \nu^{\ast}K_{S'} \sim 0$, where $\Delta_{\ol{S'}} > 0$ denotes the conductor. 
For an irreducible component $D$ of $\Delta_{\ol{S'}}$, 
since $D$ dominates $F$, 
$D$ is an exceptional divisor over $S$ whose center is $p$ such that the discrepancy $a(D,K_{S})$ satisfies $a(D,K_{S})<0$. 
This is a contradiction since $p$ is a canonical singularity.
Hence $S'$ is normal. 
Since $K_{S'} \sim 0$, $K_{S} \sim 0$, and $S$ has only rational double points, so does $S'$. 

(3) We can write $E_{S'}=C'+aF$ with $a \in \Z_{\geq 0}$ as an equation of Weil divisors. 
If $a=0$, then $E \cap S'$ does not contain $F$, which contradicts (1). 
\end{proof}

\emph{Step 3.} 
Our aim is to prove the nefness of $H_{S'}-E_{S'}$ on the Du Val K3 surface $S'$. 
Let $\mu \colon \wt{S} \to S'$ be the minimal resolution of singularity. 
Then $\wt{S}$ is a smooth K3 surface. 
Thus, it suffices to show $\mu^{\ast}(H_{S'}-E_{S'})$ is nef. 
Put $g:=f \circ \mu \colon \wt{S} \to S$: 
\[
\begin{tikzcd}
g := f \circ \mu&\wt{S} \arrow[d,"\mu"] && \\
F\arrow[d] \arrow[r, phantom, "\subset"] &S' \arrow[r, hook] \arrow[d,"f"]& X' \ar[r,hook] \arrow[d]& \P(\mcF) \arrow[d]\\
p \arrow[r, phantom,"\in"] &S \ar[r, hook]&X\ar[r, equal]&X.
\end{tikzcd}
\]
Let $\wt{C} \subset \wt{S}$ be the proper transform of $C \subset S$. 
When $S$ is singular, let $\wt{F} \subset \wt{S}$ be the proper transform of $F \subset S'$. 
\begin{claim}
\begin{enumerate}
\item $(H_{S'}-E_{S'})^{2}>0$ and $H_{S'}-E_{S'}$ is effective. 
\item $H^{1}(S',\mcO(H_{S'}-E_{S'}))=0$. 
%
%\item Let $\wt{C} \subset \wt{S}$ be the proper transform on $C \subset S$. 
%Then 
%$\mu^{\ast}E_{S'}=\wt{C}$ or 
%there is a connected union of $(-2)$-curves $\Gamma$ with 
%$\mu^{\ast}E_{S'}=\wt{C}+\Gamma$ such that 
%$\wt{C}.\Gamma=1$, $\Gamma^{2}=-2$, and $\wt{F} \subset \Gamma$. 
%The latter case only occurs when $S$ is singular. 
%
\end{enumerate}
\end{claim}
\begin{proof}
(1) Note that $\mu^{\ast}H_{S'}$ is nef big and 
\begin{align*}
\mu^{\ast}(H_{S'}-E_{S'})^{2} 
&= (H_{S'}-E_{S'})^{2} \\
&= (-K_{\wt{X}})^{3} \\
&= \xi^{4} = 2g-2-2d \geq 2, \\
\mu^{\ast}H_{S'}.\mu^{\ast}(H_{S'}-E_{S'})
&=(H_{S'}-E_{S'}).H_{S'} \\
&=(-K_{\wt{X}})^{2}.(\s^{\ast}(-K_{X})) \\
&=\xi^{3}.\pi^{\ast}(-K_{X}) \\
&=(c_{1}^{2}-c_{2})(\mcF).(-K_{X}) \\
&=2g-2-d > 0.
\end{align*}
In particular, $\mu^{\ast}(H_{S'}-E_{S'})$ is effective on $\wt{S}$. 

(2) Note that $\mcF$ is an ACM bundle. 
Since $\RG(\P(\mcF),\mcO(\xi))=\RG(X,\mcF)=\C^{\oplus 3+g-d}$, the sequences 
\begin{align*}
&0 \to \mcO_{\P(\mcF)}(-\xi) \to \mcO_{\P(\mcF)}^{\oplus 2} \to \mcI_{S'}(\xi) \to 0 \text{ and } 0 \to \mcI_{S'}(\xi) \to \mcO_{\P(\mcF)}(\xi) \to \mcO_{\P(\mcF)}(\xi)|_{S'} \to 0
\end{align*}
show that $\RG(\mcI_{S'}(\xi))=\C^{\oplus 2}$, which shows $H^{>0}(S',\mcO(H_{S'}-E_{S'}))=0$. 
%
\begin{comment}
(3) Let $E_{\wt{S}}:=\mu^{\ast}E_{S'}$. Note that $E_{\wt{S}}$ contains $\wt{C}$ and $E_{\wt{S}}-\wt{C}$ is exceptional over $S$. 

If $E_{\wt{S}}$ is nef, then there are a smooth elliptic curve $D$ and $a>0$ with $E_{\wt{S}}=aD$. 
Let $f \colon S \to \P^{1}$ be the elliptic fibration given by the pencil $\lvert D \rvert$. 
Since $E_{\wt{S}}-\wt{C}$ is effective, we have 
$
0 \leq (E_{\wt{S}}-\wt{C}).D = (aD-\wt{C}).D = -\wt{C}.D \leq 0
$.
Hence $\wt{C}.D=0$, which means the smooth elliptic curve $\wt{C}$ is contained in an $f$-fiber. 
Thus $\wt{C}$ is an $f$-fiber, which means $\wt{C} \in \lvert D \rvert$. 
Since $E_{\wt{S}}-\wt{C}$ is exceptional over $S$, we have $a=1$. 
Hence $E_{\wt{S}}=\wt{C}$. 

Suppose $E_{\wt{S}}$ is not nef. 
Let $E_{\wt{S}}=M+\sum_{i=1}^{n} \bigsqcup_{j=1}^{m_{i}}\Gamma_{i,j} $ be the decomposition as in Lemma~\ref{lem-K3dcp}. 
Since some member of $\lvert \wt{E} \rvert$ contains $\wt{C}$, so does some member of $\lvert M \rvert$. 
Thus $M-\wt{C}$ is effective. 
By the same argument as above, we can see that $M \sim \wt{C}$. 

As $g_{\ast}E_{\wt{S}}=f_{\ast}E_{S'}=C$, every $\Gamma_{i,j}$ is contracted by $g \colon \wt{S} \to S$. 
Note that $\Gamma_{1,1},\ldots,\Gamma_{1,m_{1}}$ are $(-2)$-curves with $E_{\wt{S}}.\Gamma_{1,j}=-1$. 
Since $E_{\wt{S}}=\mu^{\ast}E_{S'}$, $\mu(\Gamma_{1,j})$ is not a point. 
Hence $\mu_{\ast}\Gamma_{1,j}=F$ holds for every $j$. 
As $\mu$ is the minimal resolution, $m_{1}=1$ and $\Gamma_{1,1}$ is the proper transform $\wt{F}$ of $F$. 
We complete the proof.
\end{comment}
\end{proof}
\emph{Step 4.}
Finally we show $\mu^{\ast}(H_{S'}-E_{S'})$ is nef. 
Let 
\begin{align}\label{eq-nefdcp}
\mu^{\ast}(H_{S'}-E_{S'}) = M+N = M+\sum_{i=1}^{n} \bigsqcup_{j=1}^{m_{i}}N_{i,j}
\end{align}
be the decomposition as in Lemma~\ref{lem-K3dcp}. 
Since $\mu^{\ast}(H_{S'}-E_{S'}).N_{1,j}=-1$, $N_{1,j}$ is not contracted by $\mu$. 
Put $N_{1,j}':=\mu_{\ast}N_{1,j}$. 
If $N_{1,j}'=F$, then 
\begin{align*}
-1
&=\mu^{\ast}(H_{S'}-E_{S'}).N_{1,j} \\
&=(H_{S'}-E_{S'}).N_{1,j}' &\text{ (since $H_{S'}$ and $E_{S'}$ are Cartier) }\\
&=(H_{S'}-E_{S'}).F =1,
\end{align*}
which is a contradiction. 
Therefore, for every $j$, the morphism $g \colon \wt{S} \to S$ does not contract $N_{1,j}$ to a point, which is equivalent to saying 
\[\mu^{\ast}H_{S'}.N_{1,j} > 0 \text{ for all } i.\]
\begin{claim}\label{claim-excludeline}
$\mu^{\ast}H_{S'}.N_{1,j} \geq 2$ for each $j$.
\end{claim}
\begin{proof}
Suppose that there exists a $j$ with $\mu^{\ast}H_{S'}.N_{1,j}=1$. 
Then $N_{1,j}$ is the proper transform $\wt{l}$ of a line $l$ on $S$. 
By (\ref{eq-nefdcp}), we have $\mu^{\ast}(H_{S'}-E_{S'}).N_{1,j}=-1$, which implies $\mu^{\ast}E_{S'}.N_{1,j}=2$. 
Thus we obtain $E_{S'}.N_{1,j}'=2$ on $S'$ where $N_{1,j}':=\mu_{\ast}N_{1,j}$. 

When $S$ is smooth, the morphism $S' \to S$ is isomorphic and hence we have $C.l = 2$. 
This is a contradiction since we pick a general $C$ as in Proposition~\ref{prop-CFK} and hence $C$ meets every line transversally at at most one point. 

Suppose that $S$ is not smooth. 
In this case, $N_{1,j}' \subset S'$ is a reduced Weil divisor containing the proper transform $l'$ of the line $l$. 
Hence there is $b \in \Z_{\geq 0}$ with $N_{1,j}'=l'+bF$. 
Thus we obtain $E_{S'}.l'=E_{S'}.(N_{1,j}'-bF)=2+b \geq 2$. 
Since $E_{S'}=f^{-1}(C)$ as schemes and $\lg(\mcO_{E_{S'} \cap l'}) \geq 2$, 
we have $\lg(\mcO_{C \cap l}) \geq 2$, which is also a contradiction. 
\end{proof}
Thus we obtain
\[2 \leq \mu^{\ast}H_{S'}.N_{1,j} 
= \mu^{\ast}(H_{S'}-E_{S'}).N_{1,j} + \mu^{\ast}E_{S'}.N_{1,j}
= -1+ \mu^{\ast}E_{S'}.N_{1,j},\]
which means $\mu^{\ast}E_{S'}.N_{1,j} \geq 3$ for all $j \in \{1,\ldots,m_{1}\}$. 
Put $L:=\mu^{\ast}E_{S'}+N_{1}$. Then 
\begin{align*}
L^{2}
&=\mu^{\ast}E_{S'}^{2} + 2\mu^{\ast}E_{S'}.N_{1} + N_{1}^{2} \\
&=2 \sum_{j=1}^{m_{1}} \mu^{\ast}E_{S'}.N_{1,j} - 2m_{1} \\
&\geq 4m_{1}. 
\end{align*}
Therefore, $h^{0}(L) \geq \chi(L) = 2m_{1}+2$. 
Since $h^{0}(M+\sum_{i \geq 2} N_{i}) = h^{0}(M)=h^{0}(H_{S'}-E_{S'}) \geq (1/2)(H_{S'}-E_{S'})^{2}+2=1+g-d$, we have 
\[(2m_{1}+2)(1+g-d) \leq h^{0}(L)h^{0}\left( M+\sum_{i=2}^{n} N_{i} \right) < g+1\]
by Theorem~\ref{thm-LazBN}. 
If $m_{1} \geq 1$, then $4(1+g-d) < g+1$. 
Hence $3(1+g) < 4d$. 
Since $d \leq g-2$, we have $3+3g < 4g-8$. 
Hence $g = 12$. 
Then $d=H_{S'}.E_{S'}=10$, $m_{1}=1$, $\mu^{\ast}E_{S'}.N_{1,1}=3$ and $\mu^{\ast}H_{S'}.N_{1,1}=2$. 
Hence $N_{1,1}$ is the proper transform $\wt{\gamma}$ of a conic $\gamma$ on $S$. 
Thus we obtain $E_{S'}.N_{1,1}'=3$ on $S'$ where $N_{i}'=\mu_{\ast}N_{i}$. 
By the same argument as in the proof of Claim~\ref{claim-excludeline}, the curve $\gamma:=g(N_{1,1})$ is a conic satisfying $\lg(\mcO_{\gamma \cap C}) \geq 3$, which contradicts Proposition~\ref{prop-CFK}~(3). 
We complete the proof. 
\end{proof}

\subsection{Proof~of~Theorem~\ref{main-index1}}
%\begin{proof}[\textbf{Proof~of~Theorem~\ref{main-index1}}]
If $c_{1}(\mcE) = 0$, then Theorem~\ref{thm-split} shows that $\mcE$ is the case of Theorem~\ref{main-index1}~(1).
Consider the case when $c_{1}(\mcE) = c_{1}(X)$.
If $\mcE$ is decomposable, then since $\mcE$ is nef, it is the case of Theorem~\ref{main-index1}~(2).
Recall that, thanks to \cite{mos2} or \cite{Wis89}, it is known that there is no indecomposable rank $2$ Fano bundle over $X$.
Thus if $\mcE$ is indecomposable, then $c_{2}(\mcE) \neq 0$ by Lemma~\ref{lem-c20} and the tautological divisor $\xi$ of $\P(\mcE)$ is not ample.
Therefore Proposition~\ref{prop-BNineq} can be applied to $\mcE$ to prove that it is the case of Theorem~\ref{main-index1}~(3). 
The remaining part follows from Theorem~\ref{thm-exist} and the non-emptyness of $\mathfrak{H}_{d}'(X)$ defined in (\ref{eq-defofHd}), which is established by \cite{CFK}. \qed
%\end{proof}

\section{Embedding theorem}\label{sec-MukaiEmb}
To conclude this paper, we prove Theorem~\ref{main-emb}. 
Until the end of this paper, we follow the terminology introduced in Notation~\ref{nota-c1Fc1X}; let $X$ denote a Fano $3$-fold with $\rho(X)=i_{X}=1$, and $\mcF$ a rank $2$ weak Fano bundle with $c_{1}(\mcF)=c_{1}(X)$. 

Let $\xi$ be a tautological divisor on $\P_{X}(\mcF)$, and let 
\[\Phi_{\lvert \xi \rvert} \colon Y:=\P_{X}(\mcF) \to \P(H^{0}(\mcF))\]
be the morphism given by the complete linear system $\lvert \xi \rvert$. 
Note that $\lvert \xi \rvert$ is base point free by \cite[Theorem~1.7]{FHI1}. 
Let $\psi \colon Y \to \ol{Y}$ be the contraction given by the semi-ample divisor $\xi$. 
Then, $\psi$ is the first part of the Stein factorization of $\Phi_{\lvert \xi \rvert}$.  
We also note that $\psi$ is not an isomorphism, which is equivalent to saying that $\xi$ is not ample, by Wi\'{s}niewski's classification result \cite{Wis89} of rank $2$ Fano bundles on Fano $3$-folds of Picard rank $1$. 
The Hirzebruch--Riemann--Roch and the Kawamata--Viehweg vanishing theorem give
\begin{align}\label{eq-h0F}
h^{0}(\mcF)=\chi(X,\mcF)=\frac{(-K_{X})^{3}}{2}+4+K_{X}.c_{2}(\mcF)=\frac{s_{3}(\mcF)}{2}+4.
\end{align}

\subsection{Hyper-elliptic case}\label{subsec-hypell}

By the equation (\ref{eq-h0F}), $s_{3}(\mcF)=2$ if and only if $h^{0}(\mcF)=5$. 
In this case, the morphism $\Phi_{\lvert \xi \rvert}$ is of degree $2$. 
In other words, the condition $s_{3}(\mcF)=2$ is a sufficient condition for $\deg \Phi_{\lvert \xi \rvert}=2$. 
In the following proposition, we characterize the pair $(X,\mcF)$ such that $\Phi_{\lvert \xi \rvert}$ is not a birational morphism onto the image.

\begin{prop}\label{prop-nonhypell}
The image $Z$ of $\Phi_{\lvert \xi \rvert}$ is normal and $\deg \Phi_{\lvert \xi \rvert} \in \{1,2\}$. 
Moreover, $\Phi_{\lvert \xi \rvert}$ is not a birational morphism onto $Z$ if and only if $s_{3}(\mcF)=2$ or $(X,\mcF)$ satisfies (2) in Theorem~\ref{main-emb}. 
\end{prop}

Before proving the proposition above, we show the following proposition. 
\begin{prop}[\cite{JPR05}]\label{prop-divisorial}
The contraction $\psi \colon Y \to \ol{Y}$ contracts a divisor if and only if $X$ and $\mcF$ satisfies one of the following.  
\begin{enumerate}
\item $\mcF \simeq \mcO \oplus \mcO(-K_{X})$. 
\item $X$ is a double cover of a smooth del Pezzo $3$-fold $V$ of degree $5$ ramified along a smooth K3 surface, and $\mcF$ is isomorphic to the pull-back of the restriction of the rank $2$ quotient bundle $\mcQ_{\Gr(5,2)}$ under the embedding $V_{5} \hra \Gr(5,2)$. In this case, $s_{3}(\mcF)=2$. 
%In this case, $h^{0}(\mcF)=5$ and 
%the morphism $\Phi_{\lvert \mcF \rvert} \colon X \to \Gr(5,2)$ induced by $\mcF$ is a closed embedding or a double cover onto a del Pezzo $3$-fold of degree $5$. 
\end{enumerate}
\end{prop}
\begin{proof}
First, we see that $\psi$ is divisorial when $(X,\mcF)$ satisfies (1) or (2). 
If $\mcF \simeq \mcO \oplus \mcO(-K_{X})$, $\ol{Y}$ is the cone of the anticanonical model of $X$ and $\psi$ is the blowing-up at the vertex. 
If $(X,\mcF)$ satisfies (2), then there is a double covering $\Phi \colon X \to V$ onto a del Pezzo $3$-fold of degree $5$. 
Let $\mcQ_{V}$ denote the restriction of the rank $2$ quotient bundle $\mcQ_{\Gr(5,2)}$ under the embedding $V \hra \Gr(5,2)$. 
Then $\Phi^{\ast}\mcQ_{V} \simeq \mcF$. 
The morphism $\P(\mcF) \to \P^{4}$ factors $\P_{V}(\mcQ_{V})$ and $\P(\mcQ_{V}) \to \P^{4}$ is a divisorial contraction. 
Since $\psi$ is obtained as the Stein factorization of $\P(\mcF) \to \P^{4}$, $\psi$ is divisorial. 

To show the converse direction, let us suppose $\psi$ is a divisorial contraction. 
Let $\wt{X},\ol{X},C$ be as in Notation~\ref{nota-c1Fc1X}. 
By Proposition~\ref{prop-BNineq}~(1), $C$ is connected or empty. 
By Lemma~\ref{lem-c20}, $C$ is empty if and only if $\mcF \simeq \mcO_{X} \oplus \mcO_{X}(-K_{X})$. 
Assume that $C$ is not empty. 
Then $\wt{X}=\Bl_{C}X$ is a weak Fano $3$-fold having a divisorial crepant contraction. 
By the classification result of \cite{JPR05}, 
$X$ is a prime Fano $3$-fold of genus $6$ and $-K_{X}.C=4$. 
Since $c_{2}(\mcF) \equiv C$, we have $c_{2}(\mcF)=4$. 
In this case, $h^{0}(\mcF)=5$ holds by (\ref{eq-h0F}). 
It is known by \cite{Gushel-genus6} that the morphism $\Phi_{\lvert \mcF \rvert} \colon X \to \Gr(5,2)$ induced by $\mcF$ is a closed embedding or a double cover onto a smooth del Pezzo $3$-fold of degree $5$. 

Now we show that $\Exc(\psi) \sim 3\xi+\pi^{\ast}K_{X}$. 
Take $n,a,b \in \Z_{\geq 0}$ such that $\Exc(\psi) \sim n(a\xi+b\pi^{\ast}K_{X})$ and $a,b$ are coprime. 
Let $D:=a\xi+b\pi^{\ast}K_{X}$. 
Since $\xi^{3}D=0$, we have 
$0=a\xi^{4}+b\xi^{3}\pi^{\ast}K_{X}=2a-6b$, 
which implies $a=3$ and $b=1$. 
Then $nD \sim \Exc(\psi)$, and hence $\psi_{\ast}D \sim 0$. 
Since $-D$ is $\psi$-nef, the negativity lemma shows $D$ is effective. 
Hence $D=\Exc(\psi)$. 

Now it suffices to show that $\Phi_{\lvert \mcF \rvert}$ is not a closed embedding. 
If $\Phi_{\lvert \mcF \rvert}$ is closed embedding, then the image $V$ of $\Phi_{\lvert \mcF \rvert}$ is a complete intersection of two members of $\lvert \mcO_{
\Gr(5,2)}(1) \rvert$ and a member of $\lvert \mcO_{\Gr(5,2)}(2) \rvert$ \cite{Gushel-genus6}. 
Consider the partial flag variety $\Fl(5;2,1)$ with the projections $\pr^{21}_{1} \colon \Fl(5;2,1) \to \P^{4}$ and $\pr^{21}_{2} \colon \Fl(5;2,1) \to \Gr(5,2)$.
Put 
$L_{1}:={\pr^{21}_{1}}^{\ast}\mcO_{\P^{4}}(1)$ and $L_{2}:={\pr^{21}_{2}}^{\ast}\mcO_{\Gr(5,2)}(1)$.
Then the ideal of $Y=\P_{X}(\mcF) \simeq \P_{V}(\mcQ|_{V})$ in $\Fl(5;2,1)$ has the following Koszul resolution:
\begin{align*}
0 \to \mcO(-4L_{2}) \to \mcO(-2L_{2}) \oplus \mcO(-3L_{2})^{\oplus 2} \xrightarrow{\alpha} \mcO(-L_{2})^{\oplus 2} \oplus \mcO(-2L_{2}) \to \mcI_{Y/\Fl(5;2,1)} \to 0.
\end{align*}
This Koszul resolution above yields
\begin{align}\label{eq-idealresol}
&0 \to \mcO(3L_{1}-5L_{2}) \to \mcO(3L_{1}-3L_{2}) \oplus \mcO(3L_{1}-4L_{2})^{\oplus 2} \to \Im(\alpha)(3L_{1}-L_{2}) \to 0  \\
&0 \to \Im(\alpha)(3L_{1}-L_{2}) \to \mcO(3L_{1}-2L_{2})^{\oplus 2} \oplus \mcO(3L_{1}-3L_{2}) \to \mcI_{Y/\Fl(5;2,1)}(3L_{1}-L_{2})  \to 0 \notag \\
&0 \to \mcI_{Y/\Fl(5;2,1)}(3L_{1}-L_{2})  \to \mcO_{\Fl(5;2,1)}(3L_{1}-L_{2}) \to \mcO_{Y}(D) \to 0 \notag
\end{align}
We note that $-K_{\Fl(5;2,1)} \sim 2L_{1}+4L_{2}$ and $\pr^{21}_{1} \colon \Fl(5;2,1) \to \P(V)$ is the projectivization of $\Omega_{\P^{4}}(2)$. 
For $1 \leq a \leq 5$, we have 
\begin{align*}
H^{i}(\Fl(5;2,1),3L_{1}-aL_{2})
%&=H^{i}(\Fl(5;2,1),-K_{\Fl}+K_{\Fl}+3L_{1}-aL_{2}) \\
%&=H^{i}(\Fl(5;2,1),K_{\Fl}+5L_{1}+(4-a)L_{2}) \\
&=H^{7-i}(\Fl(5;2,1),-5L_{1}+(a-4)L_{2})^{\vee} \\
&=\left\{ 
\begin{array}{ll}
0 & \text{ if } a \in \{1,2,3\} \\
H^{7-i}(\P^{4},\mcO_{\P^{4}}(-5)) & \text{ if } a = 4 \\
H^{7-i}(\P^{4},\Omega_{\P^{4}}(-3)) & \text{ if } a = 5 
\end{array} 
\right.\\
&=\left\{ 
\begin{array}{ll}
\C & \text{ if } a=4 \text{ and } i=3 \\
0 & \text{ otherwise. }
\end{array} 
\right.
\end{align*}
Hence $\RG(Y,\mcO_{Y}(D))=\RG(\mcI_{Y/\Fl(5;2,1)}(3L_{1}-L_{2}))[1] = \RG(\Im(\alpha)(3L_{1}-L_{2}))[2] = \RG(3L_{1}-4L_{2})[2]^{\oplus 2} = \C[-1]^{\oplus 2}$. 
In particular, we have $h^{0}(Y,\mcO_{Y}(D))=0$, which is a contradiction. Therefore, $\Phi_{\lvert \mcF \rvert}$ is not a closed embedding. 
This completes the proof. 
\end{proof}

\begin{proof}[Proof~of~Proposition~\ref{prop-nonhypell}]
Let $Y=\P_{X}(\mcF) \xrightarrow{\psi} \ol{Y} \xrightarrow{p} Z \subset \P(V)=\P(H^{0}(\mcF))$ be the Stein factorization of $\Phi_{\lvert \xi \rvert}$. 
Let $H_{Z}$ denote a hyperplane section of $Z$ such that $\Phi_{\lvert \xi \rvert}^{\ast}H_{Z}=\xi$ and let $\ol{\xi}:=p^{\ast}H_{Z}$. 
Put $g_{\mcF}:=\frac{s_{3}(\mcF)}{2}+1$. 
Then 
$\deg Z = \frac{2g_{\mcF}-2}{\deg p}$
and
$\codim Z=g_{\mcF}-1$. 
Then \cite{Fujita75} shows 
$0 \leq \Delta(Z):=\deg Z-(\codim Z+1) =  \frac{2g_{\mcF}-2}{\deg p} -(g_{\mcF}-1)$. 
If $\deg p \neq 1$, then $\deg p=2$ and $\Delta(Z)=0$, which implies that $Z$ is normal. 
If $\deg p=1$, then the map $\P(\mcF) \to Z$ is birational. 
As 
$H^{0}(\mcO_{\P(V)}(m)) \to H^{0}(\mcO_{Y}(m\xi))$
is surjective for every $m \geq 0$ by Noether-Enriques-Petri theorem, $Z$ is normal. 
%By taking a ladder of $\lvert \xi \rvert$ on $Y=\P(\mcF)$, this problem can be reduced to the projectively normality of a canonical curve, which follows from Noether-Enriques-Petri theorem. 

Now it is enough to show $\deg p \neq 2$ if $s_{3}(\mcF) \neq 2$. 
If $\deg p=2$, then $Z$ is one of the following \cite{Fujita75}.
\begin{enumerate}
\item[(i)] $Z=\P^{4}$. 
\item[(ii)] $Z=\Q^{4}$. 
\item[(iii)] $Z$ is a cone over $\Q^{3}$.
\item[(iv)] $Z$ is the image of the morphism $\s \colon \F = \F(d_{1},d_{2},d_{3},d_{4}) \to Z$, where 
\[\F(d_{1},d_{2},d_{3},d_{4}):=\P_{\P^{1}}\left( \bigoplus_{i=1}^{4} \mcO_{\P^{1}}(d_{i}) \right) \text{ with } 0 \leq d_{1} \leq \cdots \leq d_{4} \text{ and } \sum_{i=1}^{4} d_{i}=\frac{s_{3}(\mcF)}{2}\]
and $\s$ is given by the complete linear system of a tautological bundle $\eta$. 
\item[(v)] $Z$ is the join of a line and a Veronese surface. 
\end{enumerate}

Suppose that $s_{3}(\mcF) \geq 4$.
Then the case (i) does not occur. 
Since $\xi$ is primitive in $\Pic(Y)$, the case (v) also does not occur. 
If the case (iv) occurs, then we denote a fiber of $\F(d_{1},d_{2},d_{3},d_{4}) \to \P^{1}$ by $F$. 
If $\sigma \colon \F \to Z$ is divisorial, which is equivalent to the condition that $d_{1}=d_{2}=d_{3}=0$, 
then a unique member $E \in \lvert \eta-d_{4}F \rvert$ is the $\s$-exceptional divisor. 
Putting $\ol{F}:=\s_{\ast}F$, we have $H_{Z} \sim d_{4}\ol{F}$ as Weil divisors. 
Thus $\xi$ is divisible by $d_{4}$, which implies that $d_{4}=1$, a contradiction. 
Therefore, $\dim \Sing(Z) \leq 1$. 
In summary, only the cases (ii), (iii), or (iv) with $\dim \Sing(Z) \leq 1$ occur.

Suppose $\psi \colon Y \to \ol{Y}$ is divisorial. 
Since we assumed $s_{3}(\mcF) \geq 4$, 
Proposition~\ref{prop-divisorial} shows $\mcF=\mcO_{X} \oplus \mcO_{X}(-K_{X})$ with $(-K_{X})^{3} \geq 4$. 
Then the morphism $\Phi_{\lvert \xi \rvert}$ coincides with the blowing-up of the cone of the image of $\Phi_{\lvert -K_{X} \rvert}$ at the vertex. 
In particular, $-K_{X}$ is not very ample. 
Hence $(X,\mcF)$ is in the case (1) and (2) of Theorem~\ref{main-emb} by \cite{Iskovskikh77}. 

Suppose that $\psi$ is a small contraction. 
Take a general ladder $S \subset \wt{X} \subset Y=\P(\mcF)$ as in (\ref{eq-ladder}). 
Denote the restriction $\Phi_{\lvert \xi \rvert}|_{S}$ by $\Phi_{S}$. 
Then $\Phi_{S}$ is a double covering $\Phi|_{S} \colon S \to W$, where $W$ is a rational normal scroll (including the quadric). 
Moreover, the pull-back of a hyperplane section $H_{W}$ of $W$ is linearly equivalent to $L:=H_{S}-C$, where $H_{S}:=-K_{X}|_{S}$ and $C$ is an elliptic curve with $c_{2}(\mcF) \equiv C$. 
Hence $L^{2}=s_{3}(\mcF)$. 
Take integers $k \geq 0$ and $a>0$ so that $W \simeq \P(\mcO(a) \oplus \mcO(a+k))$ and $H_{W}$ is a tautological divisor of this projectivization. 
Let $f$ be the ruling on $W$ and $h:=H_{W}-af$. 
Let $D:={\Phi_{S}}_{\ast}C$. Then
\begin{align}\label{eq-degD}
H_{W}.D
=H_{W}.{\Phi_{S}}_{\ast}C
=\Phi_{S}^{\ast}H_{W}.C 
=(H_{S}-C).C 
=H_{S}.C 
=c_{1}(\mcF)c_{2}(\mcF)
\end{align}
holds. 
For general member $C' \in \lvert C \rvert$, we have $C \cap C'=\emp$ and hence $D':={\Phi_{S}}_{\ast}C'$ does not meet $D$. 
Hence $D^{2}=0$ on the scroll $W$. 
Thus $D_{\red}$ is a smooth rational curve and $C \to D_{\red}$ is a double cover. 
Thus $D \sim 2f$ or $k=0$ and $D \sim 2h$. 
If $D \sim 2f$, then $c_{1}(\mcF)c_{2}(\mcF)=H_{W}.D=2$, which contradicts the inequality in Proposition~\ref{prop-BNineq}~(3). 
Hence the case $k=0$ and $D \sim 2h$ only occur. 
In this case, the following can be calculated for the Chern classes of $\mcF$; 
$c_{1}(\mcF)c_{2}(\mcF)=H_{W}.D=2h(h+af)=2a$, 
$s_{3}(\mcF)=L^{2}=2H_{W}^{2}=4a$. 
Hence it holds that $c_{1}(X)^{3}=c_{1}(\mcF)^{3}=s_{3}(\mcF)+2c_{1}(\mcF)c_{2}(\mcF)=8a$. 
Moreover, $h^{0}(S,L)=2a+2$ and $h^{0}(S,C)=2$ holds. 
Since $(S,-K_{X}|_{S})=(S,L+C)$ is Brill--Noether general by Proposition~\ref{prop-BNineq}~(2) and Theorem~\ref{thm-LazBN}~(3), we have 
\[\frac{c_{1}(X)^{3}}{2}+1 = h^{0}(S,-K_{X}|_{S}) \geq h^{0}(S,L) \cdot h^{0}(S,C)=4(a+1)=\frac{c_{1}(X)^{3}}{2}+4,\]
a contradiction. 
We complete the proof. 
\end{proof}
Let 
\[\Phi_{\lvert \mcF \rvert} \colon X \to \Gr(H^{0}(\mcF),2)\]
be the morphism given by the globally generated rank $2$ vector bundle $\mcF$. 
The relationship between $\Phi_{\lvert \xi \rvert}$ and $\Phi_{\lvert \mcF \rvert}$ may be organized as follows.
Let $\Fl(H^{0}(\mcF);2,1) \simeq \P(\mcQ_{\Gr(H^{0}(\mcF),2)})$ denote the Flag variety and $\pr^{21}_{i} \colon \Fl(H^{0}(\mcF);2,1) \to \Gr(H^{0}(\mcF),i)$ the projection. 
Then there is a morphism $\Psi \colon Y=\P(\mcF) \to \Fl(H^{0}(\mcF);2,1)$ such that the following diagram
\begin{equation}\label{dia-Flag}
\begin{tikzcd}
Y=\P(\mcF) \arrow[d,"\pi"'] \arrow[r,"\Psi"] \arrow[rr, bend left, "\Phi_{\lvert \xi \rvert}"] \arrow[rd, phantom, "\Box"]& \Fl(H^{0}(\mcF);2,1) \arrow[r, "\pr^{21}_{1}"] \arrow[d,"\pr^{21}_{2}"']& \P(H^{0}(\mcF)) \\
X \arrow[r, "\Phi_{\lvert \mcF \rvert}"'] & \Gr(H^{0}(\mcF),2)&
\end{tikzcd}
\end{equation}
commutes. 
Note that $\Phi_{\lvert \mcF \rvert}$ is finite since $\mcF|_{C} \not\simeq \mcO_{C}^{\oplus 2}$ for every curve $C$.  
Hence $\Psi$ is also finite. 
Moreover, if $\Phi_{\lvert \xi \rvert}$ is birational onto its image, then so is $\Phi_{\lvert \mcF \rvert}$. 

We also prepare another lemma in order to prove Theorem~\ref{main-emb}. 

\begin{lem}\label{lem-Sveedim}
Let $x \in X$ be a point and $l:=\pi^{-1}(x)$ a fiber. 
Put $\ol{l}:=\Phi_{\lvert \xi \rvert}(l) \subset \P(H^{0}(\mcF))$. 
If $\mcF$ does not split into line bundles, then $\dim \pi(\Phi_{\lvert \xi \rvert}^{-1}(\ol{l})) \leq 1$. 
\end{lem}
\begin{proof}
Let $V:=H^{0}(\mcF)$ and let $m:=\dim V$, which is $\frac{s_{3}(\mcF)}{2}+4$ by (\ref{eq-h0F}). 
Note that the maps $l \to \Psi(l) \to \ol{l}$ are isomorphisms. 
Put $\wt{W}:=(\pr^{21}_{1})^{-1}(\ol{l})$ and $W:=\pr^{21}_{2}(\wt{W})$.
Then $W$ is a cone over $\P^{1} \times \P^{m-4}$ embedded in $\Gr(V,2)$: 
\[
\begin{tikzcd}
\Psi^{-1}(\wt{W}) \arrow[r] \arrow[d] \arrow[rd, phantom, "\Box"]&\wt{W} \arrow[r] \arrow[d] \arrow[rd, phantom, "\Box"]&\ol{l} \arrow[d]\\
\P(\mcF) \arrow[d,"\pi"'] \arrow[r,"\Psi"] \arrow[rd, phantom, "\Box"]& \Fl(V;2,1) \arrow[r, "\pr^{21}_{1}"] \arrow[d, "\pr^{21}_{2}"'] & \P(H^{0}(\mcF)) \simeq \P^{m-1} \\
X \arrow[r,"\Phi_{\lvert \mcF \rvert}"] & \Gr(H^{0}(\mcF),2) \simeq \Gr(m,2). 
\end{tikzcd}
\]
The line $\ol{l} \subset \P(V)$ corresponds to a $2$ dimensional quotient $V \epm \C^{\oplus 2}$. 
By taking the dual, we have a two dimensional subspace $\C^{\oplus 2}  \hra V^{\vee}=H^{0}(\Gr(V,2),\mcS^{\vee})$. 
Let $s \colon \mcO_{\Gr(V,2)}^{\oplus 2} \to \mcS^{\vee}$ be the corresponding morphism. 
Then, the degeneracy locus $D(s):=\{u \in \Gr(V,2) \mid \rk s(u) \leq 1\}$ is isomorphic to $W$. 
Hence $\Phi_{\lvert \mcF \rvert}^{-1}(W)$ is the degeneracy locus $D(s_{X})$ of the morphism 
\[s_{X}:=\Phi_{\lvert \mcF \rvert}^{\ast}s \colon \mcO_{X}^{\oplus 2} \to \mcF^{\perp}.\]
If $D(s_{X})=X$, then there is a finite morphism $X \to W$ and hence a rational map $X \dra \P^{1}$ whose indeterminacy locus is contained in a finite set. 
Hence the rational map $X \dra \P^{1}$ defines a morphism $X \to \P^{1}$. 
As $X$ is of Picard rank $1$, this morphism must contracts $X$ to a point. 
Hence the image of $\Phi_{\lvert \mcF \rvert}$ is contained in a linear space $P \simeq \P^{m-3}$, which is the image of a $\pr_{1}^{21}$-fiber $\wt{P}$. 
Since $\mcQ|_{P} \simeq \mcO_{P} \oplus \mcO_{P}(1)$, 
it holds that $\mcF \simeq \mcO_{X} \oplus \mcO_{X}(1)$, which contradicts the assumption. 

Hence $D(s_{X}) \neq X$, which implies $\rk \Im (s_{X})=2$ and $s_{X}$ is injective. 
Let $\mcK$ be the reflexive hull of $s_{X}$ and $\mcC:=\mcF^{\perp}/\mcK$, which is torsion-free: 
\[
\begin{tikzcd}
0 \arrow[r] &\mcO_{X}^{\oplus 2} \arrow[r,"s_{X}"] \arrow[d]&\mcF^{\perp} \arrow[r] \arrow[d, equal] &\Cok(s_{X}) \arrow[r] \arrow[d]& 0 \\
0 \arrow[r] & \mcK \arrow[r] & \mcF^{\perp} \arrow[r] & \mcC \arrow[r] & 0.
\end{tikzcd}
\]
Put $\mcT:=\Ker(\Cok(s_{X}) \to \mcC) \simeq \Cok(\mcO_{X}^{\oplus 2} \to \mcK)$. 
Since $\mcT$ is a torsion sheaf, $c_{1}(\mcT) \geq 0$ and hence $c_{1}(\mcK) \geq 0$. 
On the other hand, since $\mcF^{\perp}$ is globally generated, so is $\mcC$, which implies $c_{1}(\mcC) \geq 0$. 
If $c_{1}(\mcC)=0$, then $\mcC \simeq \mcO^{\oplus \rk \mcC}$, which contradicts $\Hom(\mcF^{\perp},\mcO_{X})=H^{0}(\Ker(H^{0}(\mcF) \otimes \mcO_{X} \to \mcF))=0$. 
Hence $c_{1}(\mcC)>0$. 
Since $c_{1}(\mcF^{\perp})=c_{1}(X)$ is a generator of $\Pic(X)$, 
it holds that $c_{1}(\mcK)=0$, which shows $\dim \Supp \mcT \leq 1$. 
In particular, $\mcO_{X}^{\oplus 2} = \mcK$ holds as subsheaves of $\mcF^{\perp}$, and hence $\Cok(s_{X})$ is torsion free. 
Thus $D(s_{X})=\Phi_{\lvert \mcF \rvert}^{-1}(W)$ is at most $1$-dimensional. 
Since $\pi(\Phi_{\lvert \xi \rvert}^{-1}(\ol{l})) \subset \Phi_{\lvert \mcF \rvert}^{-1}(W)$, this proves the assertion. 
\end{proof}

\subsection{Proof of Theorem~\ref{main-emb}}
If $(X,\mcF)$ satisfies one of (1) -- (3) in Theorem~\ref{main-emb}, it is clear that $\Phi_{\lvert \mcF \rvert} \colon X \to \Gr(2,H^{0}(\mcF))$ is a double covering. 
We now prove the converse.
We put $Y':=\Psi(Y)$, $X'=\Phi_{\lvert \mcF \rvert}(X)$, $\psi':=\pr^{21}_{1}|_{Y'}$. 
Then $Z = \psi'(Y') = \Phi_{\lvert \xi \rvert}(Y)$. 
By Proposition~\ref{prop-nonhypell}, $Z$ is normal and hence $\Phi_{\lvert \xi \rvert}$ factors $\ol{Y}=\Spec_{Z} {\Phi_{\lvert \xi \rvert}}_{\ast}\mcO_{Y}$ as $Y \xrightarrow{\psi} \ol{Y} \xrightarrow{g} Z$. 
Let $\ol{Z}:=\Spec_{Z}\psi'_{\ast}\mcO_{Z}$. 
Let $Y' \xrightarrow{\alpha} \ol{Z} \xrightarrow{\beta} Z$ be the Stein factorization of $\psi' \colon Y' \to Z$.  
Note that $\ol{Z}$ is normal since $Z$ is normal, and $\alpha \colon Y' \to \ol{Z}$ is a birational morphism having connected fibers. 
There is a morphism $h \colon \ol{Y} \to \ol{Z}$ over $Z$ corresponding to $\psi'_{\ast}\mcO_{Y'} \to {\Phi_{\lvert \xi \rvert}}_{\ast}\mcO_{Y}$. 
Then $Y \xrightarrow{\psi} \ol{Y} \xrightarrow{h} \ol{Z}$ is the Stein factorization of $\alpha \circ \Psi$. 
Then the following diagram arises:
\[
\begin{tikzcd}
&&\ol{Y} \arrow[rdd, bend left, "g"] \arrow[d,"h"']&& \\
&&\ol{Z} \arrow[rd,"\beta"]&&\Phi_{\lvert \xi \rvert}=g \circ \psi \\
Y \arrow[d,"\pi"'] \arrow[r,"\Psi"] \arrow[rruu, bend left, "\psi"] & Y' \arrow[rr,"\psi'"] \arrow[ru,"\alpha"] \arrow[d,"\pi'"]&& Z  \arrow[r, hook]&\P(H^{0}(\mcF)) \\
X \arrow[r,"\Phi_{\lvert \mcF \rvert}"'] & X'.&&
\end{tikzcd}
\]
\begin{lem}\label{lem-singfibMella}
The morphism $\pi'|_{\Exc(\alpha)} \colon \Exc(\alpha) \to X'$ is finite. 
%In particular, $X'$ is normal. 
%If $\Phi$ is of degree $1$, then $\Phi$ is an isomorphism. 
\end{lem}
\begin{proof}
Pick an arbitrary $\pi'$-fiber $l'$, and put $\ol{l}:=\psi'(l')$. 
Then $\ol{l}$ is a line on $\P(H^{0}(\mcF))$ and $\psi'|_{l'} \colon l' \to \ol{l}$ is an isomorphism. 
Suppose $l'$ is contained in $\Exc(\alpha)$. 
Then $\alpha(l')$ is also a smooth rational curve and contained in $\alpha(\Exc(\alpha))$. 
Then $\dim \alpha^{-1}(\alpha(l')) \geq 2$. 
Hence there is an irreducible surface $S \subset \Psi^{-1}\alpha^{-1}(\alpha(l'))$ on $Y$ such that $\Phi_{\lvert \xi \rvert} \colon Y \to \P(V)$ contracts $S$ to $\ol{l}$. 
By Lemma~\ref{lem-Sveedim}, $\pi(S)$ is a curve $C$. 
Hence $S=\P(\mcF|_{C})$ and $\dim(\psi(S)) \leq 1$ shows $\xi|_{S}$ is nef but not big. 
This contradicts $c_{1}(\mcF)>0$. 
\end{proof}
We now show that $X'$ is normal. 
Since $\Exc(\alpha)$ is finite over $X'$ by Lemma~\ref{lem-singfibMella}, 
$\alpha$ is an isomorphism at the generic point of each $\pi'$-fiber. 
Hence $Y'$ is normal at that point. 
If $X'$ is not normal at a point $x' \in X'$, then $Y'$ is non-normal along $l':={\pi'}^{-1}(x') \subset Y'$, which is a contradiction. 
Hence $X'$ is normal. 
Since $X'$ is normal, so is $Y'$. 

Assume that $\Phi_{\lvert \mcF \rvert} \colon X \to X' \subset \Gr(H^{0}(\mcF),2)$ is not a closed embedding. 
Since $X'$ is normal, this is equivalent to assuming that $\deg \Phi_{\lvert \mcF \rvert} > 1$. 
Then it follows from Proposition~\ref{prop-nonhypell} that $\deg \Psi=2$, $\deg \psi'=1$, and $(X,\mcF)$ is one of the following.
\begin{enumerate}
\item[(i)] $(X,\mcF)$ satisfies (2) of Theorem~\ref{main-emb}.
\item[(ii)] $s_2(\mcF) = 2$.
\end{enumerate}
In the latter case, $Z \simeq \P^{4}$. 
Since $Y'$ is normal, $\beta \colon \ol{Z} \to \P^{4}$ is
isomorphic and $\psi' \colon Y' \to \P^{4}$ can be identified with the contraction $\alpha$. 
In particular, $\psi'$ is an isomorphism or a divisorial contraction. 
Hence the same applies to $\psi$. 
As we pointed out in Section~\ref{subsec-hypell}, $\mcF$ is not a Fano bundle, which means that $\psi$ is not an isomorphism. 
Hence $\psi$ is divisorial. 
By Proposition~\ref{prop-divisorial}, 
$\mcF \simeq \mcO \oplus \mcO(-K_{X})$ or $(X,\mcF)$ satisfies (3) of Theorem~\ref{main-emb}. 
In the former case, the condition $\Phi_{\lvert \mcF \rvert}$ is not a closed embedding implies that $X$ is hyperelliptic. 
Hence $(X,\mcF)$ satisfies (1) of Theorem~\ref{main-emb}. 
This completes the proof of Theorem~\ref{main-emb}. \qed

\bibliography{quadric_paper_f}
\bibliographystyle{amsplain}

\end{document}